\documentclass[oneside,british]{amsart}
\usepackage[T1]{fontenc}
\usepackage[latin9]{inputenc}
\usepackage{geometry}
\usepackage{babel}
\usepackage{textcomp}
\usepackage{amsbsy}
\usepackage{amstext}
\usepackage{amsthm}
\usepackage{amssymb}
\usepackage{graphicx}
\usepackage{makecell}
\usepackage[unicode=true,pdfusetitle,
 bookmarks=true,bookmarksnumbered=false,bookmarksopen=false,
 breaklinks=false,pdfborder={0 0 1},backref=false,colorlinks=false]
 {hyperref}

\makeatletter

\providecommand{\tabularnewline}{\\}

\numberwithin{equation}{section}
\numberwithin{figure}{section}
\theoremstyle{plain}
\newtheorem*{conjecture*}{\protect\conjecturename}
\theoremstyle{plain}
\newtheorem{thm}{\protect\theoremname}
\theoremstyle{definition}
\newtheorem{defn}[thm]{\protect\definitionname}
\theoremstyle{remark}
\newtheorem{rem}[thm]{\protect\remarkname}
\theoremstyle{definition}
\newtheorem{example}[thm]{\protect\examplename}
\theoremstyle{plain}
\newtheorem{prop}[thm]{\protect\propositionname}
\theoremstyle{plain}
\newtheorem{conjecture}[thm]{\protect\conjecturename}
\theoremstyle{plain}
\newtheorem{lem}[thm]{\protect\lemmaname}

\usepackage{babel}
\numberwithin{thm}{section}
\usepackage{subcaption}
\usepackage{lipsum}

\theoremstyle{definition} 
\theoremstyle{definition} 
\theoremstyle{definition}

\DeclareMathOperator{\HH}{\mathrm{H}}
\DeclareMathOperator{\tr}{\mathrm{tr}}
\DeclareMathOperator{\hyp}{\mathrm{hyp}}
\DeclareMathOperator{\ext}{\mathrm{ext}}

\providecommand{\conjecturename}{Conjecture}
\providecommand{\definitionname}{Definition}
\providecommand{\examplename}{Example}
\providecommand{\propositionname}{Proposition}
\providecommand{\remarkname}{Remark}
\providecommand{\theoremname}{Theorem}
\providecommand{\lemmaname}{Lemma}

\makeatother

\providecommand{\conjecturename}{Conjecture}
\providecommand{\definitionname}{Definition}
\providecommand{\examplename}{Example}
\providecommand{\lemmaname}{Lemma}
\providecommand{\propositionname}{Proposition}
\providecommand{\remarkname}{Remark}
\providecommand{\theoremname}{Theorem}

\title{Minimal Surfaces, Knots, and Neural Networks}

\author{Tancredi Schettini Gherardini\texorpdfstring{$^1$}{}}
\thanks{\texorpdfstring{$^1$}{}University of Bonn and Max Planck Institute for Mathematics, Bonn, Germany. E-mail: \href{mailto:tsg@math.uni-bonn.de}{tsg@math.uni-bonn.de}}

\author{Marco Usula\texorpdfstring{$^2$}{}}
\thanks{\texorpdfstring{$^2$}{}Mathematical Institute, University of Oxford, Oxford, United Kingdom. E-mail: \href{mailto:marco.usula@maths.ox.ac.uk}{marco.usula@maths.ox.ac.uk}}

\begin{document}
\setcounter{totalnumber}{1} 

\begin{abstract}
A recent conjecture by Joel Fine posits a relationship between the
coefficients of the HOMFLY polynomial of a knot $K$ in the 3-sphere
$S^{3}$, and the signed count of minimal surfaces in hyperbolic 4-space
$\HH^{4}$ meeting the sphere at infinity at $K$, with prescribed
genus and self-intersection number. In this paper, we develop a novel
machine learning framework based on Physics-Informed Neural Networks
(PINNs) to numerically solve the minimal surface equation in hyperbolic space.
We utilise this framework to test Fine's Conjecture by constructing
near-minimal discs bounding various families of knots in $S^{3}$.
Furthermore, we develop an algorithmic method to find self-intersections
and compute their sign. For every knot analysed, the computationally
discovered minimal discs and their self-intersection numbers
align with the predictions of Fine's Conjecture, providing empirical
evidence for it. 
\end{abstract}

\maketitle
\tableofcontents{}

\section{Introduction}

The asymptotic Plateau problem is a classical, fundamental
topic in geometric analysis. It revolves around the following
question: given a closed submanifold $M^{k}\subset S^{n}$ of the $n$-dimensional
sphere, does there exist a $k+1$-dimensional, properly immersed \emph{minimal} submanifold of
hyperbolic space $\HH^{n+1}$ asymptotic to $M$ at infinity?
A foundational breakthrough in this area was achieved by Anderson
\cite{AndersonPlateau}, who provided the following general existence result: \emph{for any closed submanifold
$M\subset S^{n}$, there exists a complete, absolutely area-minimising
submanifold of $\HH^{n+1}$ asymptotic to $M$} (cf. Theorem \ref{thm:anderson} for a more precise statement).

Not very much is known about the local and global structure of the
space of solutions of the asymptotic Plateau problem. This problem
is of special interest in the case of minimal \emph{surfaces} properly
immersed in \emph{$4$-dimensional} hyperbolic space. Indeed,
there is a fascinating interplay between the study of these surfaces,
and the study of \emph{knots} in the $3$-sphere $S^{3}$. To explain
this connection, consider a knot $K\subset S^{3}$ and denote
by $\mathcal{M}\left(K\right)$ the set of properly immersed minimal
surfaces in $\HH^{4}$ asymptotic to $K$ at infinity. A priori, we only
know that $\mathcal{M}\left(K\right)$ is non-empty, by Anderson's
Theorem. How does the global structure of $\mathcal{M}\left(K\right)$
reflect the properties of the knot $K$? This is the main subject of a recent paper by Fine \cite{FineKnots}, where he lays the foundation for a systematic
study of $\mathcal{M}\left(K\right)$ and its relationship with the knot
at infinity.

Let us briefly describe Fine's main results. Let $\mathcal{M}_{g,d}\left(K\right)\subset\mathcal{M}\left(K\right)$
be the subspace of properly immersed minimal surfaces in $\HH^4$ asymptotic to $K$ at infinity, with genus
$g\in\mathbb{N}$ and self-intersection number (the signed count of
transverse double points) $d\in\mathbb{Z}$. Fine proved that, for
a \emph{generic} knot $K$, the space $\mathcal{M}_{g,d}\left(K\right)$
is an \emph{oriented $0$-dimensional manifold}; that is, a discrete
set of points, each representing a minimal surface carrying a $\pm1$
sign. Furthermore, for $g=0$ (the case of minimal discs), he proved
that $\mathcal{M}_{0,d}\left(K\right)$ is \emph{compact}. Consequently,
it consists of a \emph{finite} number of surfaces, which can then
be counted with sign. One of the major results in \cite{FineKnots}
is that this signed count is an \emph{invariant} of $K$: it does
not change if we replace $K$ by another knot $K'$, smoothly isotopic
to $K$.

In \cite{FineKnots}, Fine conjectures that the moduli spaces $\mathcal{M}_{g,d}\left(K\right)$
for $g>0$ should be compact as well, for $K$ generic; moreover,
he conjectures that the count with sign $\#\mathcal{M}_{g,d}\left(K\right)$
should be closely related to the coefficients of a fundamental polynomial
invariant of knots, namely the \emph{HOMFLY polynomial}. This knot
invariant, which we shall denote by $P\left(K\right)$, is a $\mathbb{Z}$-valued
linear combination of monomials of the form $z^{2h}a^{2k}$, where $a, z$ are formal variables,
$k\in\mathbb{Z}$, and $h\in\mathbb{N}$. The precise form of Fine's
Conjecture is the following: 
\begin{conjecture*}
For a generic knot $K\subset S^{3}$, the signed count $\#\mathcal{M}_{g,d}\left(K\right)$
corresponds to the integer coefficient of $z^{2g}a^{2\left(d+g\right)}$
in the HOMFLY polynomial $P\left(K\right)$. 
\end{conjecture*}
Proving Fine's Conjecture would be a major breakthrough in the understanding of minimal surfaces in hyperbolic $4$-space. Indeed, if the conjecture were true, the fact that the coefficient of $z^{2g} a^{2\left(g + d \right)}$ in $P\left(K\right)$ is non-zero would imply that (up to possibly slightly deforming the knot) \emph{there exists a properly immersed minimal surface in $\HH^4$, with genus $g$ and self-intersection number $d$, asymptotic to $K$ at infinity}. Note that this information goes well beyond the
existence result proved by Anderson, which does not give any information
about the genus and the self-intersection number of the absolutely
area-minimising solution. We remark that, while the coefficients of the HOMFLY polynomial are relatively easy to compute\footnote{There exist algorithms which can compute $P\left(K\right)$ from a
knot diagram of $K$.}, it is a priori very difficult to gain information on $\mathcal{M}_{g,d}\left(K\right)$; indeed, the minimal surface
equation is a complicated second-order non-linear partial differential equation, and at present we lack a general method to solve it explicitly.

The aim of this paper is to \emph{test and provide evidence for Fine's Conjecture by computationally constructing approximate minimal surfaces in $\HH^{4}$ predicted by the conjecture}. Specifically, we develop a machine learning framework designed
to:

\begin{enumerate}
  \item find approximately minimal \emph{discs} properly immersed in $\HH^{4}$, asymptotic to a prescribed knot at infinity;
  \item find the points of self-intersection of these solutions, and compute the self-intersection number.
\end{enumerate}

Let us explain more precisely how we test Fine's Conjecture. Consider a knot $K\subset S^{3}$,
and suppose that the coefficient $a^{2d}$ in the HOMFLY polynomial
$P\left(K\right)$ is non-zero. Fine's Conjecture then predicts the
existence of a minimal disc in $\HH^{4}$ asymptotic to $K$, with self-intersection
number $d$. We are able to provide evidence for
these existence predictions several times,
by finding various explicit near-minimal discs in $\HH^{4}$ bounding
the following knots: the \emph{unknot}; the \emph{torus knots} $T\left(3,2\right)$,
$T\left(5,2\right)$, $T\left(4,3\right)$, $T\left(5,3\right)$;
two solutions bounding the \emph{Figure-eight} knot; the \emph{Stevedore}
knot and its mirror image; the \emph{Three-twist} knot and its mirror
image; and the \emph{square} knot. For each of these near-minimal discs, we compute the self-intersection number and we verify its consistence with Fine's Conjecture.

As an example, consider the $T\left(3,2\right)$ torus knot, also
known as the \emph{right-handed trefoil}. Its HOMFLY polynomial is
\[
P\left(T\left(3,2\right)\right)=2a^{2}-a^{4}+a^{2}z^{2}.
\]
Using our method, we find a near-minimal disc in $\HH^{4}$
bounding $T\left(3,2\right)$ with self-intersection number $1$.
This solution is consistent with Fine's Conjecture, as the existence of such a solution is predicted by the term $2a^{2}$ in $P\left(T\left(3,2\right)\right)$.

Our machine learning framework relies on a powerful new technique
to solve partial differential equations numerically using \emph{neural
networks}; the models based on this technique are called \emph{Physics-Informed
Neural Networks}, or PINNs \cite{raissi2017physicsinformeddeeplearning}.
The idea is simple; to exemplify it, we consider a simple boundary
value problem: finding a harmonic function $f:\Omega\to\mathbb{R}$
on a smooth domain $\Omega\subset\mathbb{R}^{n}$, with boundary value
$f_{|\partial\Omega}=\gamma$. The PINN approach consists of looking
for an approximate solution on a space
$\left\{ \hat{f}_{\theta}:\theta\in\Theta\right\} $ of neural
networks $\hat{f}_{\theta}:\mathbb{R}^{n}\to\mathbb{R}$ (cf. \S\ref{subsec:PINNs})
with fixed architecture; here $\theta$ is the vector of learnable
parameters of the neural network. The training phase consists
in learning values of $\theta$ for which the corresponding neural
network $\hat{f}_{\theta}$ minimises, in the class $\left\{ \hat{f}_{\theta}:\theta\in\Theta\right\} $,
a \emph{loss function} of the form 
\[
\mathcal{L}\left(\hat{f}_{\theta};\mathcal{D}_{\Omega^{\circ}},\mathcal{D}_{\partial\Omega}\right)=\overbrace{\sum_{p\in\mathcal{D}_{\Omega^{\circ}}}\left|\Delta\hat{f}_{\theta}\right|^{2}\left(p\right)}^{\text{interior loss}}+\overbrace{\sum_{q\in\mathcal{D}_{\partial\Omega}}\left|\hat{f}_{\theta}\left(q\right)-\gamma\left(q\right)\right|^{2}}^{\text{boundary loss}};
\]
here $\mathcal{D}_{\Omega^{\circ}}\subset\Omega^{\circ}$ and $\mathcal{D}_{\partial\Omega}\subset\partial\Omega$
are samples of points in the interior and on the boundary of the domain, respectively.
This paradigm is theoretically justified by the \emph{Universal Approximation
Theorem} proved in \cite{HORNIK1989359}; this result says that every
continuous function from a compact subset $\Omega\subset\mathbb{R}^{n}$
to $\mathbb{R}^{m}$ can be approximated arbitrarily well (in the
$C^{0}$ topology) by (the restriction to $\Omega$ of) a neural network $\mathbb{R}^{n}\to\mathbb{R}^{m}$.


In our specific setting, we are looking for minimal immersions $u:D^{2}\to\overline{\HH}^{4}$
restricting on the boundary to a prescribed knot parametrisation $\gamma:\partial D^{2}\to S^{3}$.
A standard PINN approach would attempt to solve this problem by training a
neural network $D^{2}\to\overline{\HH}^{4}$ to minimise a loss function
involving an interior term and a boundary term\footnote{Note that PINNs have been recently deployed across a broad range of problems in geometric analysis and differential geometry, including minimal and embedded surfaces, such as (but not limited to): \cite{DebSanghavi2025HolographicPINN, ZhouYe2023MinimalSurfacePINN, fang2021physicsinformedneuralnetworkframework, Hashimoto2025PINNMinimal, Ashmore:2019wzb, Douglas:2020hpv, Anderson:2020hux, Jejjala:2020wcc, Larfors:2021pbb, Hirst2025AInstein, cortes2026machinelearningapproachnirenberg, deluca2024machinelearninggravitycompactifications,HalversonRuehle2023MetricFlows,Aggarwal_2024,hirst2026minimisingwillmoreenergyneural}.}. Our framework, instead,
incorporates the neural network as an internal, learnable component
of a composite geometric model, mathematically designed to automatically
produce \emph{asymptotically minimal} discs solving the boundary condition
\emph{exactly}. These design choices allow us to deal with a simpler
interior-only loss. This step proves to be crucial, as it drastically
improves the convergence of the training algorithm and the accuracy
of the near-solutions. The near-minimal surfaces presented in this paper are available (as parameters of the trained model) at \href{https://github.com/Tancredi-Schettini-Gherardini/deep_plateau.git}{this GitHub repository}, alongside the code used to produce them and investigate their properties.

The paper is organised as follows. In \S\ref{sec:minimal-surfaces},
we introduce the class of minimal submanifolds of hyperbolic space
considered in this paper, and we present Fine's Conjecture. In \S\ref{sec:machine-learning},
we describe our machine learning framework in detail, from the design
of the model to our computational strategy to compute the self-intersection
number. In \S\ref{sec:Results},
we present various examples of solutions found using our framework,
and we explain how these examples provide evidence for Fine's Conjecture.
Finally, in \S\ref{sec:Future_developments} we discuss possible
future improvements to the present framework.

\subsubsection*{Acknowledgements}

The authors wish to thank Joel Fine, for explaining various aspects
of his work and for providing continuous feedback and encouragement
on this project. The first-named author thanks Javier Gomez-Serrano for his suggestions on future directions, and would also like to acknowledge the support of the 2024 Max Planck-Humboldt Research Award bestowed on Geordie Williamson and Catharina Stroppel by the Max Planck Society and the Alexander von Humboldt Foundation. The second-named author is grateful to the Wiener--Anspach
Foundation for their support.

\section{\label{sec:minimal-surfaces}Minimal surfaces in hyperbolic 4-space
filling a knot at infinity}

In this section, we develop the geometric background of the paper.
After recalling the notion of minimal submanifold, we discuss minimal
submanifolds of hyperbolic space, and then we explain Fine's Conjecture
relating minimal surfaces in $\HH^{4}$ to knot invariants.

\subsection{\label{subsec:Minimal-submanifolds}Minimal submanifolds}

We begin by recalling the variational definition of minimal submanifold.
For simplicity, in this subsection we only consider \emph{closed}\footnote{A closed manifold is a compact manifold without boundary.}\emph{
immersed submanifolds of manifolds without boundary}.

Fix once and for all an ``ambient'' Riemannian manifold $\left(M^{n},g\right)$
without boundary. 
\begin{defn}
Let $N^{k}$ be a closed manifold. 
\begin{enumerate}
\item An \emph{immersion} of $N$ in $M$ is a smooth map $u:N\to M$ such
that, for every point $p\in N$, the differential $du_{p}:T_{p}N\to T_{u\left(p\right)}M$
is injective. 
\item An \emph{embedding} of $N$ in $M$ is an injective immersion\footnote{The notion of embedding is more subtle when the manifolds involved
are not compact.} $u:N\to M$. 
\item An \emph{immersed submanifold} of $M$ is the image of an immersion
$u:N\to M$. 
\item An \emph{embedded submanifold} of $M$ is the image of an embedding
$u:N\to M$. 
\end{enumerate}
\end{defn}

\begin{rem}
We remark that our immersed submanifolds are allowed to self-intersect. 
\end{rem}

\begin{example}
The standard parametrisation of the unit circle in the plane $\mathbb{R}^{2}$
\begin{align*}
S^{1} & \to\mathbb{R}^{2}\\
e^{i\theta} & \mapsto\left(\cos\theta,\sin\theta\right)
\end{align*}
is an embedding, while the immersion 
\begin{align*}
S^{1} & \to\mathbb{R}^{2}\\
e^{i\theta} & \mapsto\left(\sin\theta,\frac{1}{2}\sin2\theta\right)
\end{align*}
parametrises a figure-eight immersed submanifold in $\mathbb{R}^{2}$,
self-intersecting at a point (cf. Figure \ref{fig:circle_figure8}). 
\end{example}

\begin{figure}[t!]
\centering \includegraphics[width=0.8\textwidth]{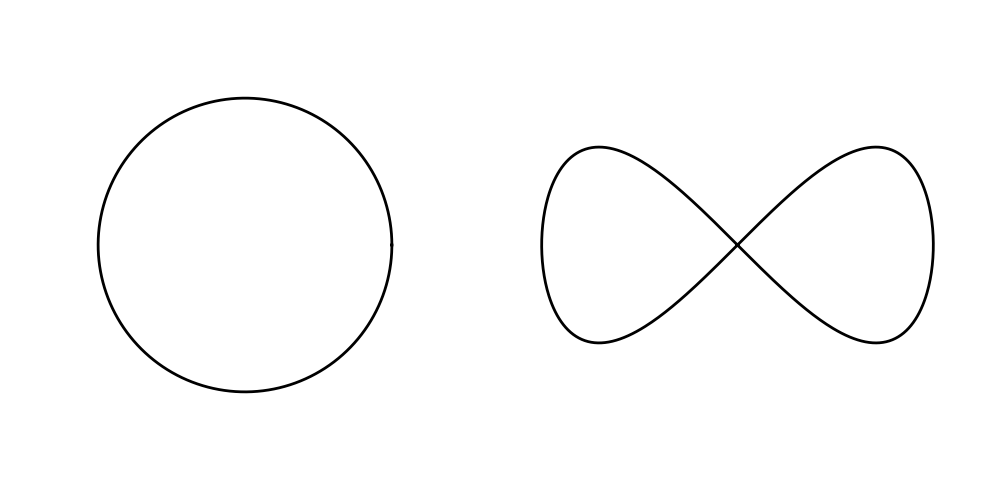}
\caption{An embedded circle, and an immersed figure-8.}
\label{fig:circle_figure8} 
\end{figure}

Given an immersion $u:N\to M$, we can parametrise the space of smooth
maps $N\to M$ sufficiently close\footnote{The space $C^{\infty}\left(N,M\right)$ is equipped with the $C^{\infty}$
topology.} to $u$ in terms of \emph{vector fields on $N$ along $u$}. More
precisely, a vector field on $N$ along $u$ is a section of the pull-back
tangent bundle $u^{*}TM$. One should think of a section $V\in C^{\infty}\left(N;u^{*}TM\right)$
as a vector field on $M$ defined along the image $u\left(N\right)$;
for every point $q\in u\left(N\right)$ of self-intersection, there
is a vector for each ``strand'' of $u\left(N\right)$ passing through
$q$, according to the multiplicity of $q$. Given such a vector field
$V$, one can displace $u\left(N\right)$ along the unit-time flow
of $V$ and obtain a new smooth map $\tilde{u}:N\to M$. If $V$ is
small enough, $\tilde{u}:N\to M$ is still an immersion, and moreover
every smooth immersion $N\to M$ in a sufficiently small neighbourhood
of $u$ in $C^{\infty}\left(N,M\right)$ can be uniquely obtained
in this way.

Let us fix an immersion $u:N\to M$. We can then pull the metric $g$
back to a metric $u^{*}g$ on $N$ defined as follows: for every $p\in N$
and every pair of vectors $v_{1},v_{2}\in T_{p}N$, we define 
\[
\left(u^{*}g\right)_{p}\left(v_{1},v_{2}\right):=g_{u\left(p\right)}\left(du_{p}\left(v_{1}\right),du_{p}\left(v_{2}\right)\right).
\]
Now that we have a metric on $N$, we can compute its \emph{$k$-dimensional
volume}: this is the integral of the Riemannian density of the pull-back
metric, 
\[
\mathcal{V}^{k}\left(u\right):=\int_{N}\omega_{u^{*}g}.
\]
In fact, the $k$-dimensional volume is invariant under reparametrisations,
i.e.~compositions $u\circ\varphi$ where $\varphi:N\to N$ is a diffeomorphism.
Thus, the $k$-dimensional volume only depends on the immersed submanifold
$u\left(N\right)$. 
\begin{defn}
We say that $u$ is a \emph{minimal immersion} if it is a critical
point of the functional above. The image $u\left(N\right)$ is then
called a \emph{minimal immersed submanifold}. 
\end{defn}

\begin{rem}
Since the $k$-volume is invariant under reparametrisations of the
domain, the property of being minimal depends only on the immersed
submanifold. 
\end{rem}

The variational definition of minimal immersion can be replaced by
an equivalent partial differential equation (PDE); we will use this
point of view in the rest of the paper. 
\begin{defn}
Given an immersion $u:N\to M$, the \emph{tension field} of $u$ is
the vector field $\tau\left(u\right)\in C^{\infty}\left(N;u^{*}TM\right)$
defined as 
\[
\tau\left(u\right):=\tr_{u^{*}g}\nabla du.
\]
\end{defn}

\begin{rem}
Let us clarify the definition of $\tau\left(u\right)$. The differential
$du$ is a smooth bundle map $TN\to TM$ covering $u$. We can equivalently
think of it as a smooth section of the vector bundle $T^{*}N\otimes u^{*}TM$,
i.e. a \emph{$1$-form on $N$ with values in $u^{*}TM$}. Now, the
bundle $u^{*}TM$ has a natural connection, i.e.~the pull-back of
the Levi-Civita connection on $TM$ associated to $g$; moreover,
since $u$ is an immersion, the bundle $T^{*}N$ has a natural connection
as well, namely the dual Levi-Civita connection of $u^{*}g$. This
defines a natural connection $\nabla$ on the tensor product $T^{*}N\otimes u^{*}TM$.
The total covariant derivative of $du$ is then a section of $T^{*}N\otimes T^{*}N\otimes u^{*}TM$;
taking the trace of the first two slots with respect to the dual metric
$u^{*}g$, we obtain a section of $u^{*}TM$. 
\end{rem}

\begin{rem}
In general, if $u:\left(N,G\right)\to\left(M,g\right)$ is a smooth
map between Riemannian manifolds, one can define the tension field
of $u$ with respect to $G$ and $g$ as $\tau\left(u\right)=\tr_{G}\nabla du$,
where $\nabla$ is now the connection on $T^{*}N\otimes u^{*}TM$
induced by the Levi-Civita connections of $G$ and $g$. Solutions
of the equation $\tau\left(u\right)=0$ are called \emph{harmonic maps}. In
the definition above, $u$ is required to be an immersion and $G$
is required to be the pull-back of $g$. In other words, \emph{an
immersion is minimal if and only if it is harmonic with respect to the pull-back metric}.
In the literature, the vector field $\tr_{u^{*}g}\nabla du$ is more
often referred to as the \emph{mean curvature} vector field of $u$. 
\end{rem}

From a variational point of view, $\tau\left(u\right)$ represents
the negative gradient of the $k$-volume functional $\mathcal{V}^{k}$,
i.e. the direction of steepest descent for \emph{$\mathcal{V}^{k}$}.
This justifies the following well-known result: 
\begin{prop}
An immersion $u:N\to\left(M,g\right)$ is minimal if and only if $\tau\left(u\right)=0$. 
\end{prop}

This point of view emphasises the fact that \emph{the minimality equation
is a quasi-linear second order elliptic equation}. This means that
the linearisation of the map $v\mapsto\tau\left(v\right)$ at an immersion
$u:N\to M$ is a second-order, elliptic partial differential operator
$J_{u}$ (in fact, a Laplace-type geometric operator) acting on sections
of $u^{*}TM$.

\subsection{\label{subsec:Conformally-compact-manifolds}Conformally compact
manifolds and minimal p-submanifolds}

The ultimate goal of this paper is to find examples of minimal submanifolds
of the hyperbolic space via machine learning. Our framework is designed to search for these solutions in a special class of
submanifolds, called \emph{p-submanifolds}. To motivate this choice,
in this short subsection we introduce the notions of \emph{conformally
compact manifold}, and of \emph{p-submanifold} of such manifolds.
Conformally compact geometry is a very active area of geometric analysis
and mathematical physics \cite{BiquardAdS}, pioneered by Fefferman--Graham
\cite{FeffermanGrahamAmbient,FeffermanGrahamConformal}, Graham--Lee
\cite{GrahamLeeEinstein}, Mazzeo--Melrose \cite{MazzeoMelroseResolvent,MazzeoPhD}
and many others. We refer to these references for many more details
on the geometry, and to \S2 of \cite{UsulaBiharmonic} for a shorter
introduction, adopting the same language used in this paper.

Let $M^{m+1}$ be a compact manifold \emph{with boundary}. 
\begin{defn}
\label{def:bdf}A \emph{boundary defining function} on $M$ is a smooth
function $\rho:M\to[0,+\infty)$ such that the zero locus $\rho=0$
is precisely the boundary $\partial M$, and moreover the differential
$d\rho$ is nowhere zero along $\partial M$. 
\end{defn}

\begin{defn}
\label{def:conformally-compact}A \emph{conformally compact metric}
on $M$ is a metric $g$ defined on the \emph{interior} $M^{\circ}$
of $M$, such that for some (hence every) boundary defining function
$\rho$ on $M$, the conformal rescaling $\rho^{2}g$ extends smoothly
to a metric on $M$. The pair $\left(M,g\right)$ is called a \emph{conformally
compact manifold}. 
\end{defn}

The prototypical example of conformally compact manifold is the hyperbolic
space. Recall that, in any dimension $n\ge2$, the \emph{hyperbolic
space} is the unique complete, simply connected Riemannian manifold
with constant sectional curvature equal to $-1$. We can describe
it explicitly as the metric 
\[
g_{\mathrm{hyp}}=\frac{4\left|d\boldsymbol{y}\right|^{2}}{\left(1-\left|\boldsymbol{y}\right|^{2}\right)^{2}}
\]
on the interior of the unit ball $B^{n}\subset\mathbb{R}^{n}$, where
$\boldsymbol{y}=\left(y^{1},...,y^{n}\right)$ are the coordinates
on $\mathbb{R}^{n}$. Observe that the function $\rho=1-\left|\boldsymbol{y}\right|^{2}$
is a boundary defining function for $B^{n}$, and $\rho^{2}g_{\mathrm{hyp}}=4\left|d\boldsymbol{y}\right|^{2}$
clearly extends to a metric on the whole of $B^{n}$. In other words,
$g_{\mathrm{hyp}}$ is a conformally compact metric on $B^{n}$. We
denote by $\HH^{n}$ the open unit ball equipped with the metric $g_{\mathrm{hyp}}$,
and by $\overline{\HH}^{n}$ its compactification. The boundary $\partial\overline{\HH}^{n}$
is the sphere $S^{n-1}$, and it is usually called the \emph{sphere
at infinity}: indeed, if $p\in\HH^{n}$ and $\gamma:[0,+\infty)\to\HH^{n}$
is a smooth curve such that $\lim_{t\to+\infty}\gamma\left(t\right)\in S^{n-1}$,
then hyperbolic length of $\gamma$ is necessarily infinite. A similar
statement holds for every conformally compact manifold.

The following concept was introduced in \cite{UsulaBiharmonic}: 
\begin{defn}
\label{def:simple-b-map}Let $N,M$ be compact manifolds with boundary.
A \emph{simple $b$-map} is a smooth map $u:N\to M$ which maps the
boundary $\partial N$ to the boundary $\partial M$, the interior
$N^{\circ}$ to the interior $M^{\circ}$, and is transverse to $\partial M$. 
\end{defn}

\begin{rem}
By definition, a simple $b$-map $u:N\to M$ restricts to a map between
boundaries, which we denote by $u_{\partial}:\partial N\to\partial M$.
Similarly, we denote by $u^{\circ}:N^{\circ}\to M^{\circ}$ the restriction
to the interiors. 
\end{rem}

We can now formulate the class of immersions and submanifolds considered
in this paper. 
\begin{defn}
Let $N,M$ be compact manifolds with boundary. 
\begin{enumerate}
\item A \emph{p-immersion} $u:N\to M$ is a simple $b$-map which is also
an immersion. 
\item A \emph{p-embedding} is an injective p-immersion. 
\item An \emph{immersed p-submanifold} of $M$ is the image of a p-immersion
$u:N\to M$. 
\item An \emph{embedded p-submanifold} of $M$ is the image of a p-embedding
$u:N\to M$. 
\end{enumerate}
\end{defn}

\begin{rem}
The letter p stands for ``product'', and reflects the fact that
if $S\subset M$ is an embedded p-submanifold, then $M$ is locally
diffeomorphic near $S$ to a neighbourhood of the zero locus of $NS$,
the normal bundle of $S$ in $M$. Since p-immersions are local p-embeddings,
the same statement holds locally for immersed p-submanifolds. These
considerations and this notation are due to Richard Melrose \cite{MelroseCorners}. 
\end{rem}

The class of p-submanifolds is natural in conformally compact geometry,
as the next proposition shows: 
\begin{prop}
\label{prop:pull-back-of-conf-compact-via-p-immersion-is-conf-compact}(Proposition
12 of \cite{UsulaIsometricEmbeddings}) Let $N,M$ be compact manifolds
with boundary, and let $u:N\to M$ be a p-immersion. If $g$ is a
conformally compact metric on $M$, then the pull-back $\left(u^{\circ}\right)^{*}g$,
a priori only defined only in the interior $N^{\circ}$, extends to a conformally
compact metric on $N$. 
\end{prop}

A conformally compact metric on a compact manifold with boundary $M$
can be equivalently and elegantly described as a smooth metric over
a vector bundle $^{0}TM$, a modified version of the usual tangent
bundle $TM$ designed precisely to de-singularise conformally compact
metrics at infinity. $^{0}TM$ is defined as the vector bundle over
$M$ whose space of smooth sections is the space of vector fields on
$M$ vanishing along the boundary (for details on this construction,
we refer to Mazzeo's PhD thesis \cite{MazzeoPhD}). As explained in detail in \S3.2 of \cite{UsulaBiharmonic}\footnote{The discussion in \S 3 of \cite{UsulaBiharmonic} involves \emph{harmonic
simple $b$-maps} between conformally compact manifolds; however,
minimal p-immersions are in particular harmonic simple b-maps, where
the metric in the domain is the pull-back metric. Since this pull-back
metric is conformally compact by Proposition \ref{prop:pull-back-of-conf-compact-via-p-immersion-is-conf-compact},
the results proved in \S 3 of \cite{UsulaBiharmonic} apply here
as well.}, if $u:N\to\left(M,g\right)$ is a p-immersion of $N$ into a conformally
compact manifold, then the tension field $\tau\left(u\right):=\tr_{u^{*}g}\nabla du$
can be interpreted as a smooth section of $u^{*}{^{0}TM}$, i.e. as
a smooth vector field on $M$ defined along $u\left(N\right)$ \emph{which
vanishes along the boundary}. Imitating the definition provided in
the previous subsection, we formulate the following 
\begin{defn}
A p-immersion (resp. p-embedding) $u:N\to\left(M,g\right)$ into a
conformally compact manifold is \emph{minimal} if $\tau\left(u\right)=0$.
A \emph{minimal immersed} (resp. \emph{embedded}) \emph{p-submanifold}
of $\left(M,g\right)$ is the image of a p-immersion (resp. p-embedding). 
\end{defn}

The minimality equation $\tau\left(u\right)=0$ for a p-immersion
into a conformally compact manifold is still a non-linear second-order
elliptic PDE, since the linearisation at $u$ of the map $v\mapsto\tau\left(v\right)$
(here $v$ ranges in the space of p-immersions $N\to M$), restricted
to the interior, is an elliptic operator $J_{u}$ of Laplace type
acting on sections of $\left(u^{\circ}\right)^{*}TM^{\circ}$. However,
this operator does \emph{not} extend to a genuinely elliptic operator
over the compactification $M$. Rather, $J_{u}$ is \emph{uniformly
degenerate} along the boundary. This aspect is not particularly important in
this paper, but it is crucial in \cite{FineKnots}. For more
details, we refer again to \cite{FineKnots} and to \S3 of \cite{UsulaBiharmonic}.

\subsection{\label{subsec:Minimal-p-submanifolds-of-Hn}Minimal p-submanifolds
of hyperbolic space}

In this paper, we will be mainly concerned with \emph{minimal
p-submanifolds of hyperbolic spaces}. In this subsection, we explain
\emph{why} we chose to work with this particular class of submanifolds.

It is easy to prove, using the maximum principle, that there cannot be
\emph{closed} minimal immersed submanifolds of $\HH^{n+1}$. Therefore,
mathematicians have focused on studying \emph{complete, properly
immersed} minimal submanifolds of $\HH^{n+1}$. A \emph{properly immersed
}minimal submanifold is the image of a proper minimal immersion $u:S\to\HH^{n+1}$,
where $S$ is an open manifold without boundary. The properness of
the map $u$ guarantees that the frontier of $u\left(S\right)$ in
the compactification $\overline{\HH}^{n+1}$ is entirely contained
in the sphere at infinity. Therefore, one can formulate an ideal boundary
value problem, known as the \emph{asymptotic} \emph{Plateau problem}:
given a set $Z\subset S^{n}$, is there a complete, properly immersed
minimal submanifold of $\HH^{n+1}$ whose limit set at infinity is
$Z$? We refer to \cite{CoskunuzerPlateau} for a survey on this problem.
A major achievement on this subject is the following celebrated \emph{existence
}result, due to Anderson: 
\begin{thm}
\label{thm:anderson}(Anderson, Theorem 3 of \cite{AndersonPlateau})
Let $M^{m-1}\subset S^{n}$ be a closed submanifold of the sphere
at infinity of $\HH^{n+1}$. Then there exists a complete, absolutely
area-minimising locally integral $m$-current $X$ in $\HH^{n+1}$,
asymptotic to $M$ at infinity. 
\end{thm}

While the regularity of Anderson's solution may be very weak, a recent
result proves that if the boundary at infinity is smoothly embedded,
and mild global regularity hypotheses are satisfied, then a solution
of the asymptotic Plateau problem must actually be \emph{polyhomogeneous} (a slight generalisation of the notion of smoothness), actually \emph{smooth}
if the domain is even-dimensional, and \emph{orthogonal to the boundary}: 
\begin{thm}
\label{thm:jared}(Marx-Kuo, corollary of Theorem 3.1 of \cite{MarxKuoRenormalized})
Let $X^{m}\subset\overline{\HH}^{n+1}$ be a $C^{m+1,\alpha}$ compact
immersed submanifold with boundary, such that $\partial X\subset S^{n}$
and $X^{\circ}\subset\HH^{n+1}$. Suppose furthermore that $X$ can
be written, near the boundary, as the graph of a function over a small
cylinder $\partial X\times[0,\varepsilon)$ embedded in a collar neighbourhood
$S^{n}\times[0,\varepsilon)$ of $S^{n}$ in $\overline{\HH}^{n+1}$.
If $X^{\circ}$ is minimal in $\HH^{n+1}$, and $\partial X$ is \emph{smoothly
embedded} in $S^{n}$, then $X$ is smooth in the interior, polyhomogeneous
up to the boundary, and intersects the boundary at infinity orthogonally.
Moreover, if $m$ is even, then $X$ is actually smooth up to the
boundary. 
\end{thm}

This last result implies that, under reasonable assumptions, \emph{every
}complete properly embedded minimal submanifold of $\HH^{n+1}$ is
in fact the interior of an immersed p-submanifold of $\overline{\HH}^{n+1}$.
This justifies our choice to restrict to immersed p-submanifolds.

Another justification, not important for this paper but crucial for
\cite{FineKnots}, is that as discussed in \S\ref{subsec:Conformally-compact-manifolds}
the minimality equation for an immersed p-submanifold of $\overline{\HH}^{n+1}$
is a non-linear uniformly degenerate elliptic partial differential equation. The elliptic theory designed to study these
equations was developed by Mazzeo--Melrose \cite{MazzeoMelroseResolvent}
and Mazzeo \cite{MazzeoPhD}; this theory provides the analytic backbone of Fine's paper.

\subsection{\label{subsec:Fine's-conjecture}Knots and Fine's conjecture}

In this subsection we discuss the main motivation for this paper,
namely testing Fine's Conjecture. This conjecture posits a relation
between minimal p-immersed surfaces on the hyperbolic $4$-space,
and \emph{knot invariants} in the $3$-sphere $S^{3}$.

First of all, let's recall the formal notion of knot and link in $S^{3}$: 
\begin{defn}
A \emph{link} is the image of an embedding $S^{1}\sqcup\cdots\sqcup S^{1}\to S^{3}$.
A \emph{knot} is a one-component link, i.e. the image of an embedding
$S^{1}\to S^{3}$. 
\end{defn}

\begin{figure}[t]
\centering \begin{subfigure}[b]{0.32\textwidth} \centering \includegraphics[width=1\textwidth]{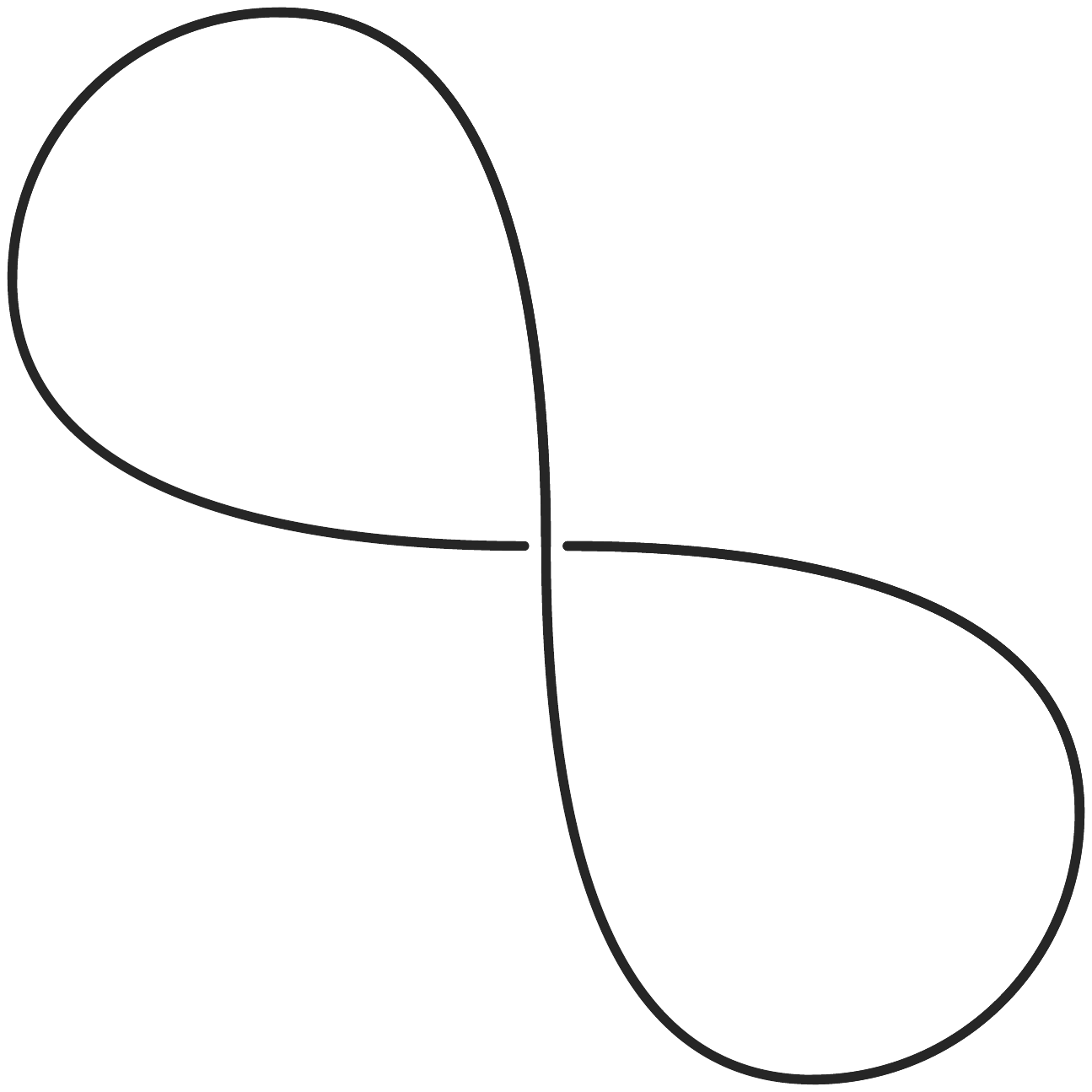}
\end{subfigure} \hfill{}\begin{subfigure}[b]{0.32\textwidth}
\centering \includegraphics[width=1\textwidth]{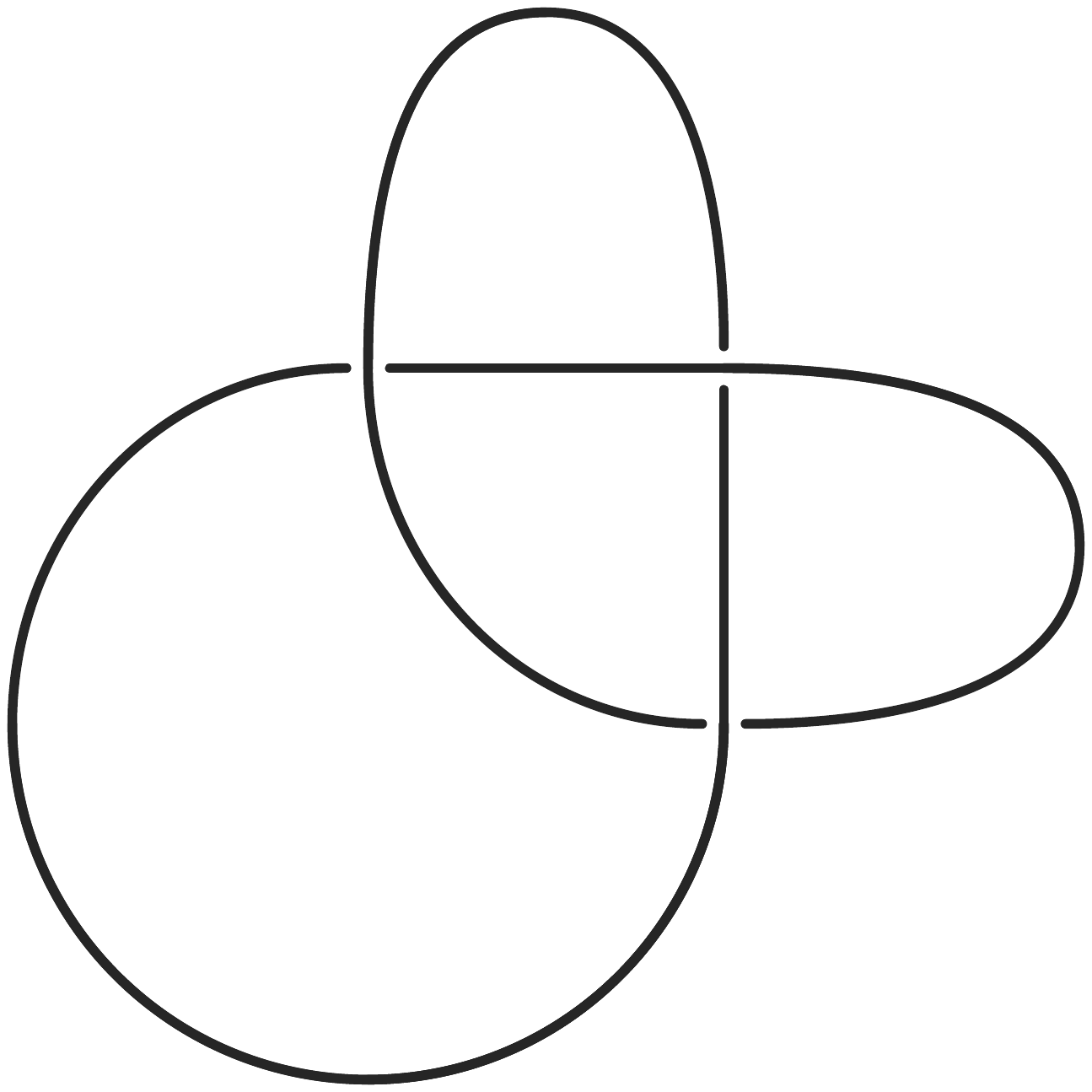}
\end{subfigure} \hfill{}\begin{subfigure}[b]{0.32\textwidth}
\centering \includegraphics[width=1\textwidth]{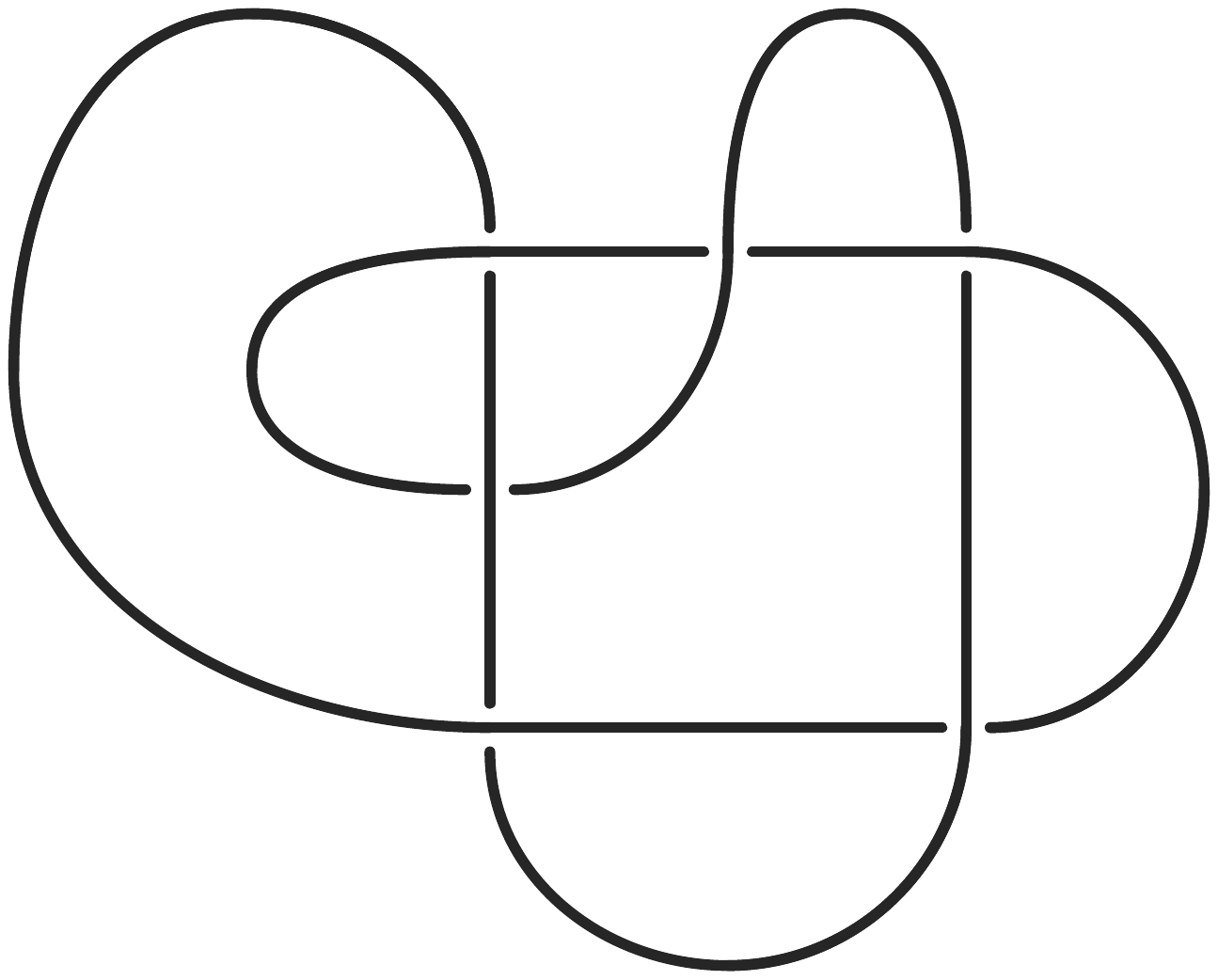}
\end{subfigure} \caption{Knot diagrams of an unknot, a trefoil knot, and a Stevedore knot.}
\label{fig:three_knots} 
\end{figure}

Some examples of knots are depicted in Figure \ref{fig:three_knots}.
We now introduce an equivalence relation on the space of links, called
\emph{isotopy}: 
\begin{defn}
Let $L_{0},L_{1}\subset S^{3}$ be two links in $S^{3}$. We say that
$L_{0}$ and $L_{1}$ are \emph{isotopic} (and we write $L_{0}\sim L_{1}$)
if there exists a smooth family $f_{t}:S^{1}\sqcup\cdots\sqcup S^{1}\to S^{3}$
of parametrised links, with $t \in \left[0,1\right]$, such that $f_{0}$ (resp. $f_{1}$) is a parametrisation
of $L_{0}$ (resp. $L_{1}$). If $L_0$ and $L_1$ are \emph{oriented}, we write $L_{0} \stackrel{\mathrm{or}}{\sim} L_{1}$ if there exists a smooth family $f_t$ as above, with $f_0$ and $f_1$ \emph{oriented} parametrisations of $L_0$ and $L_1$.
\end{defn}

The isotopy relation encodes the intuitive idea that two links are
equivalent if they can be deformed into each other in a smooth way,
without having to cut and paste them. A major problem in knot theory
consists in deciding whether two links are isotopic or not. Link \emph{invariants}
provide tools to distinguish them up to isotopy: 
\begin{defn}
A \emph{link invariant} is a map $P:\mathcal{L}\to\mathcal{X}$ from
the space $\mathcal{L}$ of links in $S^{3}$ to a set $\mathcal{X}$,
such that if $L_{1}\sim L_{2}$ then $P\left(L_{1}\right)=P\left(L_{2}\right)$.
Similarly, an \emph{oriented link invariant} is a map $P:\mathcal{L}_{\mathrm{or}}\to\mathcal{X}$ from the space of oriented links $\mathcal{L}_{\mathrm{or}}$,
such that if $L_{1}\stackrel{\mathrm{or}}{\sim}L_{2}$ then $P\left(L_{1}\right)=P\left(L_{2}\right)$ 
\end{defn}

A way to prove that two links $L_{1}$ and $L_{2}$ are \emph{not}
isotopic is to exhibit a knot invariant $P:\mathcal{L}\to\mathcal{X}$
for which $P\left(L_{1}\right)\not=P\left(L_{2}\right)$. For example,
a straightforward link invariant is the number of connected components.

Many important link invariants are \emph{polynomials}; among the most
important ones, we cite the \emph{Alexander polynomial}
(cf. \cite{AlexanderTopologicalInvariants}) and the \emph{Jones polynomial
}(cf. \cite{JonesPolynomialInvariant,JonesHecke}). We are especially
interested in a two-variable generalisation of both the Alexander
and the Jones polynomials, known as the \emph{HOMFLY polynomial} \cite{HOMFLY}.
This invariant can be described as a map $P:\mathcal{L}_{\mathrm{or}}\to\mathbb{Z}\left[a^{\pm1},z^{\pm1}\right]$
from the space of oriented links $\mathcal{L}_{\mathrm{or}}$ to the
space of Laurent polynomials in two variables $a,z$, with coefficients
in $\mathbb{Z}$. This invariant is usually defined axiomatically
in terms of \emph{link diagrams}\footnote{A link diagram of an oriented link $L\subset S^{3}$ is a 2-dimensional
projection of $L$ into $\mathbb{R}^{2}$ with only transverse self-intersections,
all with multiplicity $2$; moreover, at each double point, the diagram
specifies which strand goes ``over'' and which strand goes ``under'',
cf. Figure \ref{fig:three_knots}.}and \emph{Skein relations}: 
\begin{enumerate}
\item $P\left(\mathrm{unknot}\right)=1$; 
\item consider three oriented links $L_{+},L_{-},L_{0}$, represented by
three link diagrams which coincide everywhere except for a single
crossing, where $L_{+},L_{-},L_{0}$ differ as in Figure \ref{fig:skein}; then the HOMFLY polynomials of $L_{+},L_{-},L_{0}$ are related
by the equation 
\[
aP\left(L_{+}\right)-a^{-1}P\left(L_{-}\right)=zP\left(L_{0}\right).
\]

\end{enumerate}

\begin{figure}[t]
\centering \includegraphics[width=0.5\textwidth]{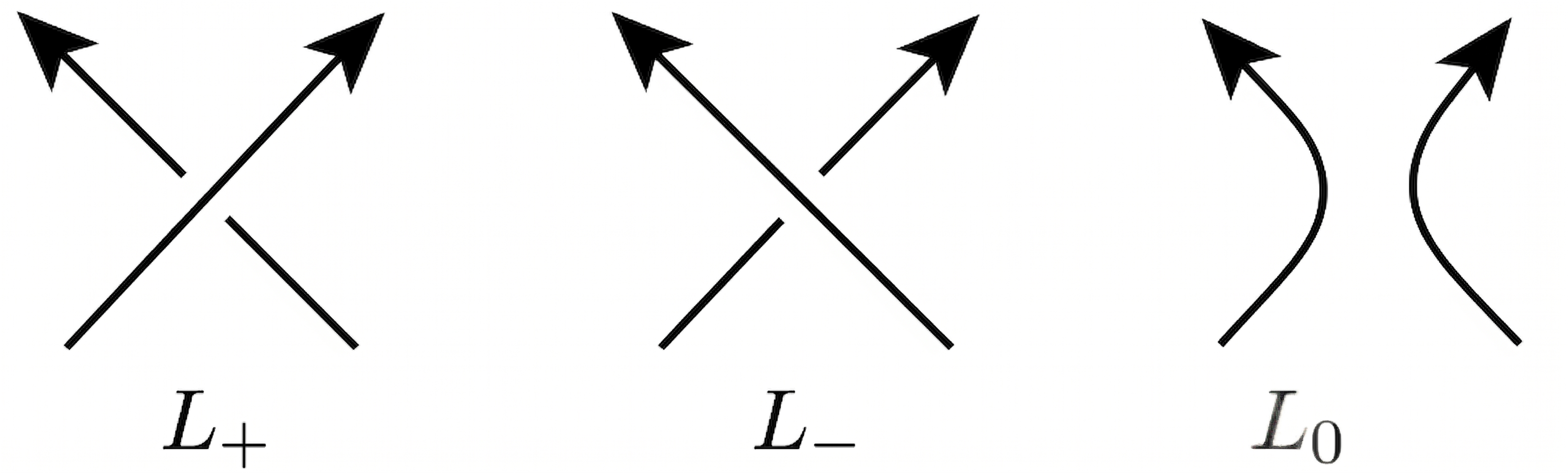}
\caption{Different crossing types in the Skein relations.}
\label{fig:skein} 
\end{figure}

One can prove that these two axioms indeed determine a well-defined
oriented link invariant. A complete discussion of the HOMFLY polynomial
goes beyond the scope of this paper, and it is not very important
for what follows. What we need to remember for the rest of the paper
is that: 
\begin{enumerate}
\item the HOMFLY polynomial of a \emph{knot} does not depend on the orientation; 
\item if $K\subset S^{3}$ is a knot, then the HOMFLY polynomial of $K$
consists only of terms of the form $z^{2k}a^{2l}$, where $k\in\mathbb{N}$
and $l\in\mathbb{Z}$; 
\item if $K\subset S^{3}$ is a knot, and $K^{*}$ is its \emph{mirror image}
(that is, a knot obtained from $K$ by an orientation-reversing diffeomorphism
of $S^{3}$), then the HOMFLY polynomials of $K$ and $K^{*}$ are
related by the equation 
\[
P\left(K^{*}\right)\left(a,z\right)=P\left(K\right)\left(a^{-1},z\right);
\]
\item the HOMFLY polynomial is \emph{computable}: there exists an algorithm
which computes the HOMFLY polynomial of a link $L\subset S^{3}$ in
finite time, given a link diagram of $L$. 
\end{enumerate}
We can now come back to minimal submanifolds in $\HH^{4}$, and to
Fine's work \cite{FineKnots}. Observe that, if $\Sigma$
is a compact, connected, oriented surface with \emph{connected} boundary,
and $u:\Sigma\to\overline{\HH}^{4}$ is a p-immersion whose boundary
map $u_{\partial}$ is an embedding, then the image $u\left(\Sigma\right)$
is a p-immersed submanifold of $\overline{\HH}^{4}$ intersecting
the boundary at infinity at an oriented knot $K=u_{\partial}\left(\partial\Sigma\right)\subset S^{3}$.
Fine's Conjecture proposes a relation between the coefficients of the
HOMFLY polynomial of $K$, and an appropriate \emph{count} of the
number of minimal p-immersed surfaces in $\overline \HH^{4}$ bounding $K$
as above.

In order to express Fine's Conjecture properly, we need to introduce
the notion of \emph{double point}. 
\begin{defn}
\label{def:double-point}Let $u:\Sigma\to\overline{\HH}^{4}$ be a
p-immersion. A \emph{self-intersection} of $u\left(\Sigma\right)$
is a point $q\in\overline{\HH}^{4}$ such that $u^{-1}\left(q\right)$
contains more than a point. A self-intersection $q$ of $u\left(\Sigma\right)$
is a \emph{double point} if $u^{-1}\left(q\right)$ consists of exactly
two points $p_{1},p_{2}\in\Sigma$, and the images of the differentials
$du_{p_{i}}:T_{p_{i}}\Sigma\to T_{q}\overline{\HH}^{4}$ span the
whole of $T_{q}\overline{\HH}^{4}$. 
\end{defn}

We remark that a \emph{generic} p-immersion $u:\Sigma\to\overline{\HH}^{4}$,
whose boundary restriction $u_{\partial}:\partial\Sigma\to S^{3}$
is an embedding, self-intersects only at a \emph{finite number of
double points} \emph{contained in the interior}. More precisely, let
$V$ be a vector field on $\overline{\HH}^{4}$ along $u$ (i.e. a
section of $u^{*}T\overline{\HH}^{4}$) which is \emph{tangent} to
the boundary at infinity $S^{3}$. If $V$ is close enough to the
zero section of $u^{*}T\overline{\HH}^{4}$, the one-time flow of
$V$ deforms $u$ into another p-immersion $\tilde{u}:\Sigma\to\overline{\HH}^{4}$
whose boundary restriction $\tilde{u}_{\partial}$ is an embedding.
For a \emph{generic} choice of such a $V$, the deformed p-immersion
$\tilde{u}$ self-intersects only at a finite number of double points.
This is a general phenomenon about middle-dimensional immersed submanifolds:
for example, Figure \ref{fig:triple_point_desingularized} shows an
immersed curve in $\mathbb{R}^{2}$ with a triple point, and a generic
deformation which resolves the triple point into three double points.

\begin{figure}[t!]
\centering \includegraphics[width=0.8\textwidth]{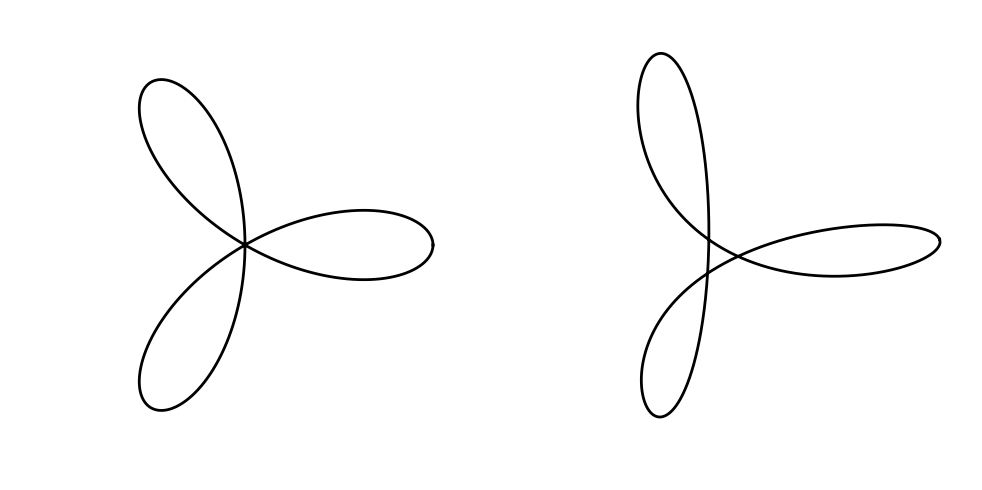}
\caption{An immersed circle in $\mathbb{R}^{2}$ with a triple point, resolved
into three double points by a generic deformation.}
\label{fig:triple_point_desingularized} 
\end{figure}

Assume now that our p-immersion $u:\Sigma\to\overline{\HH}^{4}$ is
such that the boundary map $u_{\partial}$ is an embedding, and that
$u\left(\Sigma\right)$ self-intersects only at a finite number of double
points. Fix also an orientation on $\Sigma$ and $\HH^{4}$. We can
then assign to every double point a \emph{sign}. More precisely, let
$q\in\HH^{4}$ be a double point of $u$, and let $\left\{ p_{1},p_{2}\right\} =u^{-1}\left(q\right)$.
Consider a positively oriented basis $v_{1},v_{2}$ of $T_{p_{1}}\Sigma$,
and a positively oriented basis $v_{3},v_{4}$ of $T_{p_{2}}\Sigma$.
Then the quadruple $w_{1},w_{2},w_{3},w_{4}$, with $w_{1}=d_{p_{1}}u\left(v_{1}\right)$,
$w_{2}=d_{p_{1}}u\left(v_{2}\right)$, $w_{3}=d_{p_{2}}u\left(v_{3}\right)$
and $w_{4}=d_{p_{2}}u\left(v_{4}\right)$, forms a basis of $T_{q}\HH^{4}$.
We assign to $q$ a sign $+1$ (resp. $-1$) if $w_{1},w_{2},w_{3},w_{4}$
is a \emph{positively} (resp. \emph{negatively}) \emph{oriented} basis
of $T_{q}\HH^{4}$. It is easy to check that this definition is well-posed.
In fact, note that the definition does not depend on the orientation
of $\Sigma$, but only on the orientation of $\HH^{4}$. 
\begin{defn}
Let $u:\Sigma\to\overline{\HH}^{4}$ be a p-immersion whose boundary
map $u_{\partial}:\partial\Sigma\to S^{3}$ is an embedding. Assume
that $u$ self-intersects at a finite number of double points. The
\emph{self-intersection number} of $u\left(\Sigma\right)$ is the
sum of the signs of its double points. 
\end{defn}

\begin{rem}
Observe that the self-intersection number is a property of the
p-immersed submanifold $u\left(\Sigma\right)$; indeed, if $\varphi:\Sigma\to\Sigma$
is a diffeomorphism, then $u$ and
$u\circ\varphi$ have exactly the same self-intersections, and their
signs agree. 
\end{rem}

\begin{defn}
Fix an oriented link $L\subset S^{3}$, and two numbers $g\in\mathbb{N}$
and $d\in\mathbb{Z}$. We denote by $\mathcal{M}_{g,d}\left(L\right)$
the moduli space of connected, oriented, p-immersed minimal surfaces
$S\subset\overline{\HH}^{4}$ satisfying the following properties: 
\begin{enumerate}
\item $S=u\left(\Sigma\right)$, where $\Sigma$ is a compact, connected,
oriented genus $g$ surface with boundary, and $u:\Sigma\to\overline{\HH}^{4}$
is a p-immersion whose boundary map $u_{\partial}:\partial\Sigma\to S^{3}$
is an embedding parametrising the oriented link $L$; 
\item $S$ self-intersects at a finite number of double points in $\HH^{4}$; 
\item $S$ has self-intersection number $d$. 
\end{enumerate}
\end{defn}


The following result\footnote{For expositional clarity, here we ignore regularity issues, which
are irrelevant for this work.} is a consequence of Theorems 3.32, 3.36 and Proposition 3.42 in \cite{FineKnots}: 
\begin{thm}
(Fine) For a generic oriented link $L\in\mathcal{L}_{\mathrm{or}}$,
the moduli space $\mathcal{M}_{g,d}\left(L\right)$ is a naturally
oriented $0$-dimensional manifold. 
\end{thm}

Observe that an oriented $0$-dimensional manifold is just a discrete
set of points, each of which is labelled by a $\pm1$. If this manifold
were \emph{compact}, it would consist of a \emph{finite} number of points
labelled by $\pm1$; we could then ``count them with sign'', that is,
add their labels and get an integer. We can finally formulate Fine's
Conjecture: 
\begin{conjecture}
(Fine) Let $\mathcal{K}$ be the space of knots. Then, for $K\in\mathcal{K}$
generic, the moduli space $\mathcal{M}_{g,d}\left(K\right)$ is a
\textbf{compact}, naturally oriented $0$-dimensional manifold. In
this case, the signed count of minimal surfaces in $\mathcal{M}_{g,d}\left(K\right)$
is an oriented knot invariant, and it coincides with the coefficient
of $a^{2\left(g+d\right)}z^{2g}$ in the HOMFLY polynomial of $K$. 
\end{conjecture}

\begin{rem}
In \cite{FineKnots}, Fine does not explicitly formulate the above
conjecture, but he explains (cf. \S1.1 and \S1.6 of \cite{FineKnots})
why there should be a close relation between the signed count of
minimal surfaces in $\mathcal{M}_{g,d}\left(K\right)$, and the coefficients
of the HOMFLY polynomial of $K$. We are grateful to him for sharing
with us the precise formulation of the conjecture. 
\end{rem}

\begin{rem}
\label{rem:Part-of-Fine-conjecture-is-proved}Part of Fine's Conjecture
has been proved by Fine himself in \cite{FineKnots}. The main result
is that, for generic $K\in\mathcal{K}$, the moduli spaces $\mathcal{M}_{0,d}\left(K\right)$
(that is, the moduli spaces of p-immersed minimal \emph{discs} in $\overline{\HH}^{4}$
bounding $K$) is a compact oriented $0$-dimensional manifold, and
its signed count is an oriented knot invariant. 
\end{rem}

Proving Fine's Conjecture would be a major breakthrough in the understanding
of minimal surfaces in hyperbolic $4$-space. Indeed, take a generic
knot $K$. While the coefficient of $z^{2g}a^{2\left(g+d\right)}$
in the HOMFLY polynomial of $K$ is a computable\footnote{It may be computationally costly to compute the HOMFLY polynomial:
the best known algorithms are exponential in the number of crossings
of the given link diagram.} quantity, the signed count of elements of $\mathcal{M}_{g,d}\left(K\right)$
is \emph{virtually impossible} to compute, due to the absence of general
methods to solve non-linear partial differential equations. For this reason, if true, the conjecture
should not be seen as a tool to gain information about the knot itself,
but rather as a tool to predict the \emph{existence} of minimal surfaces
in $\HH^{4}$ \emph{with specific genus and self-intersection number},
bounding a given knot in $S^{3}$. This would be a major achievement;
indeed, to the authors' knowledge, at present the only known existence
result is the previously mentioned theorem of Anderson (cf. Theorem
\ref{thm:anderson}). However, Anderson's theorem does not tell us
the genus of the area-minimising solution, nor does it give us any
information about the number of solutions or their self-intersection
numbers.

To explain better how one can use Fine's Conjecture to formulate predictions
on minimal surfaces, we consider for example the \emph{right-handed}
\emph{trefoil knot}, denoted in Alexander--Briggs (A--B) notation\footnote{The A--B notation organises prime knots based on their crossing number,
that is, the minimal number of crossings of a diagram of the knot.
The name $a_{n}$ indicates `the $n$-th knot with $a$ crossings'.} as $3_{1}$. This knot is described by the second knot diagram in
Figure \ref{fig:three_knots}. The HOMFLY polynomial of $3_{1}$ is
$\left(2a^{2}-a^{4}\right)+a^{2}z^{2}$. Therefore, Fine's Conjecture
predicts that a generic right-handed trefoil is bounded by: 
\begin{enumerate}
\item (term $2a^{2}$) \emph{at least two} distinct p-immersed minimal discs
with self-intersection number $1$; 
\item (term $-a^{4}$) \emph{at least one} p-immersed minimal disc with
self-intersection number $2$; 
\item (term $a^{2}z^{2}$) \emph{at least one} p-immersed genus $1$ minimal
surface with self-intersection number $0$. 
\end{enumerate}
Thus, if true, Fine's Conjecture would provide much more information
than the simple existence result guaranteed by Anderson's Theorem.

Another interesting example is the \emph{Stevedore knot}, usually
denoted by $6_{1}$, and described by the third knot diagram in Figure
\ref{fig:three_knots}. The HOMFLY polynomial of $6_{1}$ is $\left(a^{-2}-a^{2}+a^{4}\right)+\left(-1-a^{2}\right)z^{2}$.
As in the previous case, the non-zero terms in this polynomial predict
the existence of various p-immersed minimal surfaces bounding this
knot. However, observe that the HOMFLY polynomial \emph{does not have
a constant term}. This means that Fine's Conjecture does \emph{not}
predict the existence of a p-immersed minimal disc with self-intersection
number $0$. On the other hand, it is known that the Stevedore knot
is \emph{smoothly slice}: this means that it is bounded by a smooth,
but not necessarily minimal, p-embedded disc in $\overline{\HH}^{4}$.
Putting together these two observations, Fine's Conjecture predicts that the area-minimising
solution $S$ guaranteed by Anderson's Theorem has either positive
genus or, if it is an embedded disc, it \emph{must} be accompanied
by another p-immersed minimal disc $S'$ with self-intersection number
$0$, and with sign opposite to that of $S$ in the natural orientation of
$\mathcal{M}_{0,0}\left(6_{1}\right)$. Again, this information is
highly non-trivial and virtually inaccessible by direct analytical
means.

\section{\label{sec:machine-learning}The Machine Learning approach: PINNs}

This section constitutes the heart of the paper. After a general introduction
to physics-informed neural networks (PINNs), we describe in detail
our machine learning framework to construct minimal p-submanifolds of
$\overline{\HH}^{4}$. We then describe the training strategy. Finally,
we describe a method to find and analyse the self-intersections of
the minimal p-submanifolds found using our framework.

\subsection{\label{subsec:PINNs}Physics-Informed Neural Networks}

We start by recalling the notion of multi-layered perceptron: 
\begin{defn}[Multi-Layered Perceptron]
\label{def:MLP} Let $L\geq1$ and $d_{0},d_{1},\ldots,d_{L}\in\mathbb{N}$
be positive integers, and let $\sigma\colon\mathbb{R}\to\mathbb{R}$
be a continuous non-polynomial function, called \emph{activation function}.
For each $\ell=0,\ldots,L-1$, let 
\[
A_{\ell}\colon\mathbb{R}^{d_{\ell}}\to\mathbb{R}^{d_{\ell+1}},\qquad A_{\ell}(x)=W_{\ell}x+b_{\ell},
\]
be an affine map with \emph{weight matrix} $W_{\ell}\in\mathbb{R}^{d_{\ell+1}\times d_{\ell}}$
and \emph{bias vector} $b_{\ell}\in\mathbb{R}^{d_{\ell+1}}$. The
\emph{multi-layered perceptron} (MLP) associated with this data is
\[
\mathrm{NN}_{\theta}\colon\mathbb{R}^{d_{0}}\longrightarrow\mathbb{R}^{d_{L}},\qquad\mathrm{NN}_{\theta}\;=\;A_{L-1}\circ(\sigma\circ A_{L-2})\circ\cdots\circ(\sigma\circ A_{0}),
\]
where $\sigma$ acts component-wise. The parameters $\theta=\{(W_{\ell},b_{\ell})\}_{\ell=0}^{L-1}$
are called the \emph{learnable parameters} of the MLP. We call $d_{0}$
and $d_{L}$ the \emph{input} and \emph{output dimensions}, $L$ the
\emph{depth}, and $(d_{0},\ldots,d_{L})$ the \emph{architecture}. 
\end{defn}

Common choices of activation functions are $\sigma=\tanh$ and $\sigma=\mathrm{SiLU}$
(the \emph{sigmoid linear unit}, $t\mapsto t/(1+e^{-t})$); both are
smooth, bounded, and non-polynomial. For any fixed input $x\in\mathbb{R}^{d_{0}}$,
the map $\theta\mapsto\mathrm{NN}_{\theta}(x)$ is smooth in the parameters,
a fact that is crucial for gradient-based training. The theoretical
justification for using MLPs as function approximators is the following
classical result: 
\begin{thm}[Universal Approximation Theorem, {\cite{HORNIK1989359}}]
\label{thm:UAT} Let $K\subset\mathbb{R}^{d_{0}}$ be compact, let
$\sigma\colon\mathbb{R}\to\mathbb{R}$ be continuous and non-polynomial,
and let $f\in C(K;\mathbb{R}^{d_{2}})$. Then, for every $\varepsilon>0$,
there exist a width $d_{1}\in\mathbb{N}$ and parameters $\theta$
such that the single-hidden-layer MLP $\mathrm{NN}_{\theta}\colon\mathbb{R}^{d_{0}}\to\mathbb{R}^{d_{2}}$
with activation function $\sigma$ and architecture $\left(d_{0},d_{1},d_{2}\right)$
satisfies 
\[
\sup_{x\in K}\bigl|\mathrm{NN}_{\theta}(x)-f(x)\bigr|<\varepsilon.
\]
\end{thm}

In other words, the class of depth-two MLPs is $C^{0}$-dense in $C(K;\mathbb{R}^{d_{L}})$
for every compact $K$. 
\begin{rem}
Two remarks on the previous result are in order. First, the required
width $d_{1}$ may depend sensitively on $\varepsilon$ and on $f$;
in bad cases, $d_{1}$ could increase exponentially as $\varepsilon\to0$.
Second, analogous results hold for \emph{deep and narrow} networks:
provided the width satisfies $d_{1}\geq d_{0}+1$, a sufficiently
deep MLP can also approximate any continuous function to arbitrary
accuracy \cite{kidger2020universalapproximationdeepnarrow,hanin2018approximatingcontinuousfunctionsrelu}.
In practice, deep networks with a moderate fixed width are preferred:
depth provides a form of representational efficiency that width alone
cannot replicate, and the total parameter count for a prescribed approximation
quality is typically far smaller. 
\end{rem}

We can now give a brief definition of a \emph{physics-informed neural
network} (PINN) \cite{raissi2017physicsinformeddeeplearning}: in
its simplest form, a PINN consists of an MLP whose parameters are
optimised so that the output approximates a solution of a prescribed
PDE. The more precise setup is as follows.

Let $\Omega\subset\mathbb{R}^{d}$ be a smooth domain. Suppose that
we are interested in solving a PDE $\mathcal{F}\left[u\right]=0$
for a function $u:\Omega\to\mathbb{R}^{m}$, subject to a boundary
condition $\mathcal{B}\left[u\right]=0$ on $\partial\Omega$. The
idea is to look for an approximate solution $\hat{u}=\mathrm{NN}_{\theta}$
in the class of MLPs with fixed architecture $\left(d,...,m\right)$
and activation function $\sigma$. Schematically, the optimisation
procedure consists in: 
\begin{enumerate}
\item drawing a (pseudo-)random set of interior collocation points $\left\{ x_{i}\right\} _{i=1}^{N_{\Omega}}\subset\Omega^{\circ}$
and a set of boundary collocation points $\left\{ x_{j}'\right\} _{j=1}^{N_{\partial\Omega}}\subset\partial\Omega$; 
\item optimising the learnable parameters $\theta$ in order to minimise
the loss function 
\begin{equation}
\mathcal{L}(\theta)\;=\;\frac{1}{N_{\Omega}}\sum_{i=1}^{N_{\Omega}}\bigl|\mathcal{F}[\mathrm{NN}_{\theta}](x_{i})\bigr|^{2}\;+\;\frac{1}{N_{\partial\Omega}}\sum_{j=1}^{N_{\partial\Omega}}\bigl|\mathcal{B}[\mathrm{NN}_{\theta}](x'_{j})\bigr|^{2}.\label{eq:PINN_loss}
\end{equation}
\end{enumerate}
Evaluating the loss function requires computing the partial derivatives
$\partial^{\alpha}\mathrm{NN}_{\theta}(x)$ at the collocation points,
with respect to the \emph{input} coordinates --- not the parameters
$\theta$, which are the variables of the optimisation. Two na{\"i}ve
approaches fail: differentiating the closed-form expression for $\mathrm{NN}_{\theta}$
symbolically produces expressions that grow exponentially with the
depth $L$, while finite differences incur a truncation error of order
$O(h^{p})$ and become prohibitively expensive for derivatives of
order higher than one.

\emph{Automatic differentiation} (AD) \cite{baydin2018automatic,griewank2008evaluating}
avoids both difficulties. Since $\mathrm{NN}_{\theta}$ is a composition
of elementary arithmetic operations --- addition, multiplication,
$\exp$, $\tanh$, and so on --- for which exact derivative rules
are known, one can propagate derivative information through the \emph{computational
graph} alongside the primal evaluation. This yields exact derivatives
(up to machine precision) at a computational cost proportional to
evaluating $\mathrm{NN}_{\theta}$ itself.

Geometrically, AD evaluates the chain rule in two natural directions.
Forward-mode AD computes the differential (or pushforward), mapping
an input tangent vector $v\in\mathbb{R}^{d_{0}}$ to the output tangent
vector $d(\mathrm{NN}_{\theta})_{x}(v)\in\mathbb{R}^{d_{L}}$. This
is computationally efficient when $d_{0}\ll d_{L}$. Conversely, Reverse-mode
AD (or backpropagation) computes the pullback. It maps an output cotangent
vector (or $1$-form) $w\in\mathbb{R}^{d_{L}}$ to the input cotangent
vector $d(\mathrm{NN}_{\theta})_{x}^{*}(w)\in\mathbb{R}^{d_{0}}$.
This is highly efficient when $d_{L}\ll d_{0}$, exactly what is needed
to compute the gradient $\nabla_{\theta}\mathcal{L}$ of a scalar
loss function. Furthermore, these modes compose naturally: applying
them iteratively yields higher-order derivatives, such as the Hessian,
while maintaining a cost proportional to the original function evaluation.

When the PDE is second order (as in our case), we need to compute
both the Jacobian $J_{\hat{u}}\in\mathbb{R}^{m\times d}$ and the
Hessian $H_{\hat{u}}\in\mathbb{R}^{m\times d\times d}$ of $\hat{u}=\mathrm{NN}_{\theta}$
at all collocation points, with respect to the input coordinates.
One reverse sweep provides the full $J_{\hat{u}}$, and two forward
sweeps --- one per input coordinate --- differentiate that result
to give $H_{\hat{u}}$.

The differential operators in $\mathcal{F}$ and $\mathcal{B}$ are
applied to $\mathrm{NN}_{\theta}$ via AD; no mesh or quadrature scheme
is needed. The gradient $\nabla_{\theta}\mathcal{L}$ is likewise
computed by AD --- now in the reverse mode with respect to $\theta$
--- and used to update the parameters iteratively. The standard optimiser
is \emph{Adam} \cite{kingma2014adam}, a stochastic first-order method
with adaptive per-coordinate learning rates; a quasi-Newton method
such as L-BFGS \cite{liu1989limited} is often used for final
refinement once a moderate loss level has been reached.

Let us illustrate the framework with a simple example before specialising
to our setting. 
\begin{example}[Laplace--Dirichlet problem]
\label{ex:laplace_pinn} Let $\Omega\subset\mathbb{R}^{d}$ be a
smooth bounded domain and $\varphi\in C^{\infty}(\partial\Omega)$.
The Dirichlet boundary value problem 
\[
\begin{cases}
\Delta f=0 & \text{in} \Omega,\\
f\big|_{\partial\Omega}=\varphi,
\end{cases}
\]
has a unique smooth solution $f$. A PINN approximates $f$ by choosing
an MLP $\mathrm{NN}_{\theta}\colon\mathbb{R}^{d}\to\mathbb{R}$ and
minimising 
\[
\mathcal{L}(\theta)\;=\;\frac{1}{N_{\Omega}}\sum_{i=1}^{N_{\Omega}}\Bigl(\,\sum_{k=1}^{d}\partial_{x^{k}}^{2}\,\mathrm{NN}_{\theta}(x_{i})\Bigr)^{2}\;+\;\frac{1}{N_{\partial\Omega}}\sum_{j=1}^{N_{\partial\Omega}}\bigl(\mathrm{NN}_{\theta}(x'_{j})-\varphi(x'_{j})\bigr)^{2},
\]
where the Laplacian is evaluated by applying AD twice to $\mathrm{NN}_{\theta}$.
Note that this is a \emph{semi-supervised} scheme: the boundary term
fits the network to the known data $\varphi$ (supervised), while
the interior term is entirely unsupervised --- only the PDE is imposed,
with no a priori knowledge of $f$ in the interior of $\Omega$. 
\end{example}

A common difficulty in implementing \eqref{eq:PINN_loss} is choosing
the relative weight of the two terms: if the interior penalty is too
large, the boundary condition may be poorly satisfied; if the boundary
penalty dominates, the network may fit the boundary data at the cost
of ignoring the PDE in the interior. By employing some careful design
choices, \emph{we completely avoid this balance issue}. Our
problem consists in solving a second order PDE (the minimality equation
for a p-immersion $u:D^{2}\to\overline{\HH}^{4}$) subject to a Dirichlet-type
boundary condition (the boundary map $u_{\partial}:S^{1}\to S^{3}$
must be a prescribed oriented knot embedding). As we shall see in
the next subsection, our model \emph{satisfies the boundary condition}
\emph{exactly for every value of the learnable parameters}. Therefore,
the loss only contains the interior term (cf. Equation \eqref{eq:training_loss}).

\subsection{\label{subsec:model_description}Model description}

In this subsection, we describe our model in detail. The model produces,
\emph{by design}, simple $b$-maps (cf. Definition \ref{def:simple-b-map})
$u:D^{2}\to\overline{\HH}^{n+1}$ from the topologically closed unit
disc $D^{2}\subset\mathbb{R}^{2}$ to the compactified hyperbolic
space, whose boundary map $u_{\partial}:S^{1}\to S^{n}$ is a prescribed
embedding. 
\begin{rem}
We restrict to minimal \emph{discs}, rather than higher genus domains,
mainly for simplicity. Maps from the disc are particularly simple
to implement, and moreover part of Fine's Conjecture (cf. \S\ref{subsec:Fine's-conjecture}
and Remark \ref{rem:Part-of-Fine-conjecture-is-proved}). While the
model itself is perfectly valid for higher genus surfaces, the implementation
requires substantial modifications to include this case; we shall discuss these modifications elsewhere.
\end{rem}

\begin{rem}
Since we assume $n\geq2$, the boundary embedding $S^{1}\to S^{n}$
cannot have a dense image in $S^{n}$, and therefore we can assume without
loss of generality that $u_{\partial}\left(S^{1}\right)$ does not
contain the north pole of $S^{n}$. Thanks to this assumption, we
can use \emph{half-space coordinates} for hyperbolic space; these
coordinates are particularly convenient for computations. The hyperbolic
space $\HH^{n+1}$ is seen as the interior of the upper half-space
$\mathbb{R}_{\geq0}\times\mathbb{R}^{n}$; calling $\left(X,\boldsymbol{Y}\right)=\left(X,Y^{1},...,Y^{n}\right)$
the coordinates on $\mathbb{R}_{\geq0}\times\mathbb{R}^{n}$, the
hyperbolic metric in these coordinates is 
\[
g_{\hyp}=\frac{dX^{2}+\left|d\boldsymbol{Y}\right|^{2}}{X^{2}}.
\]
Of course this metric is equivalent to the hyperbolic metric in the
interior of the unit ball: the map 
\[
\Phi:\left(X,\boldsymbol{Y}\right)\mapsto\left(\frac{X^{2}+\left|\boldsymbol{Y}\right|^{2}-1}{\left(X+1\right)^{2}+\left|\boldsymbol{Y}\right|^{2}},\frac{2\boldsymbol{Y}}{\left(X+1\right)^{2}+\left|\boldsymbol{Y}\right|^{2}}\right)\in B^{n+1}
\]
is an isometry from the half-space model to the unit ball model.
Moreover, $\Phi$ extends to a map from the closed upper half-space
$\mathbb{R}_{\geq0}\times\mathbb{R}^{n}$ to the closed unit ball
$B^{n+1}$, and the restriction of $\Phi$ to the boundary $\mathbb{R}^{n}=\left\{ X=0\right\} $
is a conformal diffeomorphism $\Phi:\mathbb{R}^{n}\to S^{n}\backslash\left\{ N\right\} $,
where $N=\left(1,\boldsymbol{0}\right)$ is the north pole in $S^{n}$.
It is easy to check that $\Phi$ is precisely the inverse of the stereographic
projection. 
\end{rem}

The model generates maps $D^{2}\to\overline{\HH}^{n+1}$ which depend
essentially on five objects $\gamma$, $\rho$, $\ext$, $k$, $\mathrm{NN}$,
described in detail below. While $\gamma,\rho,\ext,k$ should be thought
of as \emph{hyperparameters} of the model, the object $\mathrm{NN}$
is a MLP which represents the only learnable part of the map. By
an \emph{instance} of the model, we mean a choice of the objects $\gamma$,
$\rho$, $\ext$, $k$, $\mathrm{NN}$ and of the learnable parameters
of $\mathrm{NN}$.

Every instance of the model is a map of the form 
\begin{align*}
u_{\gamma,\rho,\ext,k,\mathrm{NN}}:D^{2} & \to\overline{\HH}^{n+1}\\
\left(x,y\right) & =\left(\rho\exp\left(\mathrm{NN}^{X}\right),\ext\left(\gamma\right)+\rho^{k}\mathrm{NN}^{\boldsymbol{Y}}\right)\left(x,y\right).
\end{align*}
Here $\rho\exp\left(\mathrm{NN}^{X}\right):D^{2}\to\mathbb{R}_{\geq0}$
is the $X$-component of the map, $\ext\left(\gamma\right)+\rho^{k}\mathrm{NN}^{\boldsymbol{Y}}:D^{2}\to\mathbb{R}^{n}$
is the $\boldsymbol{Y}$-component, and: 
\begin{enumerate}
\item $\gamma:S^{1}\to\mathbb{R}^{n}$ is an embedding; 
\item $\rho$ is a boundary defining function for $D^{2}$ (cf. Definition
\ref{def:bdf}); 
\item $\ext$ is an ``extension operator'', associating to $\gamma:S^{1}\to\mathbb{R}^{n}$
a map $\ext\left(\gamma\right):D^{2}\to\mathbb{R}^{n}$ such that
$\ext\left(\gamma\right)_{|\partial D^{2}}=\gamma$; 
\item $k\in\left\{ 1,2\right\} $ is the decay exponent of the multiplier
$\rho^{k}$ in the $\boldsymbol{Y}$-component of the map; 
\item finally, $\mathrm{NN}:\mathbb{R}^{2}\to\mathbb{R}^{n+1}$ is a simple
MLP (cf. Definition \ref{def:MLP}); the first component is $\mathrm{NN}^{X}:\mathbb{R}^{2}\to\mathbb{R}$,
while the remaining components form a map $\mathrm{NN}^{\boldsymbol{Y}}:\mathbb{R}^{2}\to\mathbb{R}^{n}$. 
\end{enumerate}
Let us comment on \emph{why} we designed our model this way. First
of all, every instance of the model described above is by construction
a simple $b$-map $D^{2}\to\overline{\HH}^{n+1}$. To see this, we
recall the following observation from \cite{UsulaBiharmonic}: 
\begin{lem}
(Definition 15 and Remark 16 of \cite{UsulaBiharmonic}) A smooth
map $u:M\to N$ between compact manifolds with boundary is a simple
$b$-map if and only if the pull-back of some boundary defining function
on $N$ is a boundary defining function on $M$. 
\end{lem}

To check that $u_{\gamma,\rho,\ext,k,\mathrm{NN}}$ is indeed a simple
$b$-map, observe that $X$ is a local boundary defining function
for $\overline{\HH}^{n+1}$, defined away from the north pole; since
the image of $u_{\gamma,\rho,\ext,k,\mathrm{NN}}$ does not contain
the north pole, the pull-back of $X$ is $\rho\exp\left(\mathrm{NN}^{X}\right)$; since $\rho$ is by construction a boundary defining function for $D^{2}$,
the function $\rho\exp\left(\mathrm{NN}^{X}\right)$ is always a boundary
defining function for $D^{2}$ as well. Therefore, $u_{\gamma,\rho,\ext,k,\mathrm{NN}}$
is indeed a simple $b$-map $D^{2}\to\overline{\HH}^{n+1}$.

Note that the equation $\rho=0$ defines the boundary $\partial D^{2}=S^{1}$,
so the boundary restriction of $u_{\gamma,\rho,\ext,k,\mathrm{NN}}$
is precisely the embedding $\gamma:S^{1}\to\mathbb{R}^{n}$. In our
main application, we have $n=3$, and $\gamma:S^{1}\to\mathbb{R}^{3}=S^{3}\backslash\left\{ N\right\} $
is an explicit parametrisation of an oriented knot. Thanks to this
choice, our model does not need to learn how to satisfy the boundary
condition; it is already satisfied \emph{exactly}, regardless of the
values of the learnable parameters.

\subsubsection{\label{subsec:The-hyperparameter-choice}The hyperparameter choice}

We now discuss our choice of the hyperparameters $\gamma$, $\rho$,
$\ext$, $k$. Since we are mainly interested in finding minimal p-immersions
into $\overline{\HH}^{4}$, we will focus on this case from now on.

As already discussed, $\gamma:S^{1}\to\mathbb{R}^{3}$ is our boundary
condition, namely a smooth parametrisation of an oriented knot $K\subset\mathbb{R}^{3}$.
By construction, each instance of the model $u_{\gamma,\rho,\ext,k,\mathrm{NN}}$
restricts to $\gamma$ on the boundary. The choices of $\rho$ and
$\ext$ are more delicate, while the choice of $k\in\left\{ 1,2\right\} $
is essentially dictated by the choice of $\ext$.

Briefly speaking, $\rho$ and $\ext$ determine a ``starting point''
for our search towards a minimal p-immersion. More precisely, if we
set all the learnable parameters of $\mathrm{NN}$ to zero, the map
$u_{\gamma,\rho,\ext,k,\mathrm{NN}}$ reduces to 
\begin{align*}
D^{2} & \to\overline{\HH}^{4}\\
\left(x,y\right) & =\left(\rho,\ext\left(\gamma\right)\right)\left(x,y\right).
\end{align*}
Therefore, we can think of $\mathrm{NN}$ as a \emph{learnable perturbation}
of the map $\left(\rho,\ext\left(\gamma\right)\right)$; the training
algorithm is designed to learn an instance of $\mathrm{NN}$ which
perturbs $\left(\rho,\ext\left(\gamma\right)\right)$ to a minimal
(or more precisely, \emph{as minimal as possible} in the sense specified
in \S\ref{subsec:The-training-regime}) p-immersion. For this reason,
it is quite important to choose $\rho$ and $\ext$ appropriately:
indeed, if $\left(\rho,\ext\left(\gamma\right)\right)$ is already
``asymptotically minimal'', this facilitates the convergence of
the algorithm.

Let us first note that the choice of $\rho$ is less important than
the choice of $\ext$. Indeed, two boundary defining functions $\rho$
and $\tilde{\rho}$ differ only by a smooth, strictly positive multiplicative
factor. For this reason, in principle different choices of $\rho$
can be absorbed by the term $\exp\left(\mathrm{NN}^{X}\right)$, which
is learnable. Of course, in practice the choice of $\rho$ still plays
a role, because different choices determine different starting points
$\left(\rho,\ext\left(\gamma\right)\right)$ for the training algorithm.
A natural choice for $\rho$ is the boundary defining function $\rho=1-r^{2}$,
where $r=\sqrt{x^{2}+y^{2}}$. Another natural choice, which experimentally
gives better results, is given by 
\[
\rho_{\mathrm{st}}:=\frac{1-r^{2}}{1+r^{2}}.
\]
We call $\rho_{\mathrm{st}}$ the \emph{stereographic} boundary defining
function. This function arises from the simplest example of a minimal
p-embedded disc in $\overline{\HH}^{4}$. Indeed, consider the round
parametrisation of the unknot, 

\[
\gamma_{\mathrm{unknot}}\left(\theta\right)=\left(\cos\theta,\sin\theta,0\right).
\]

It is easy to check that the map $u:D^{2}\to\overline{\HH}^{4}$ written
in polar coordinates as 

\[
u_{\HH^{2}}\left(\theta,r\right)=\left(\frac{1-r^{2}}{1+r^{2}},\frac{2r\cos\theta}{1+r^{2}},\frac{2r\sin\theta}{1+r^{2}},0\right)
\]

is a \emph{totally geodesic} (hence minimal) p-embedding of $D^{2}$
in $\overline{\HH}^{4}$. In fact, the pull-back metric via this map
is precisely the hyperbolic metric on $\overline{\HH}^{2}$. We see
that in this example, the first component of the map is precisely
the stereographic\footnote{We chose the name \emph{stereographic} because the first three components
of the map $u\left(\theta,r\right)$ represent the inverse of the
stereographic projection from $D^{2}$ to the upper hemisphere of
the unit sphere $S^{2}$ in $\mathbb{R}^{3}$.} boundary defining function $\rho_{\mathrm{st}}$.

The most important hyperparameter choice is the extension operator
$\ext$. As discussed above, the role of $\ext$ is to extend the
chosen parametrised knot $\gamma:S^{1}\to\mathbb{R}^{3}$ to a map
$\ext\left(\gamma\right):D^{2}\to\mathbb{R}^{3}$. We first describe
a na{\"i}ve extension operator: 
\begin{defn}
We define the \emph{stereographic extension operator} as 
\[
\ext_{\mathrm{st}}\left(\gamma\right):\left(r,\theta\right)\mapsto\frac{2r}{1+r^{2}}\gamma\left(\theta\right).
\]
\end{defn}

This definition is justified by the observation that the totally geodesic
p-embedding $u_{\HH^{2}}:\overline{\HH}^{2}\to\overline{\HH}^{4}$
described above can be written as 
\[
u_{\HH^{2}}=\left(\rho_{\mathrm{st}},\ext_{\mathrm{st}}\left(\gamma_{\mathrm{unknot}}\right)\right).
\]
In other words, if we choose $\gamma=\gamma_{\mathrm{unknot}}$, $\rho=\rho_{\mathrm{st}}$,
and $\ext=\ext_{\mathrm{st}}$, then \emph{the instance of $u_{\gamma,\rho,\ext,k,\mathrm{NN}}$
obtained by setting all the learnable parameters of $\mathrm{NN}$
to zero is precisely the totally geodesic solution $u_{\HH^{2}}$}.
The problem with this na{\"i}ve approach is that, for a generic parametrised
knot $\gamma:S^{1}\to\mathbb{R}^{3}$, the extended map $\ext_{\mathrm{st}}\left(\gamma\right):D^{2}\to\mathbb{R}^{3}$
is only $C^{0}$. Since $\mathrm{NN}$ is by design smooth, this extension
choice \emph{cannot} produce smooth maps $D^{2}\to\overline{\HH}^{4}$
unless $\ext\left(\gamma\right)$ is smooth. This observation is verified
by computational experiments: if we run our algorithm for a generic
knot using the extension operator $\ext_{\mathrm{st}}$, we obtain
maps which look minimal near the boundary, but exhibit a
concentrated error near the origin of $D^{2}$.

We obtain much better results using an ansatz inspired by another
feature of \emph{$u_{\HH^{2}}$}. First of all, observe that every
smooth map $\gamma:S^{1}\to\mathbb{R}^{3}$ admits a unique smooth\emph{
harmonic} extension to $D^{2}$, namely there is a unique smooth solution
$\Gamma:D^{2}\to\mathbb{R}^{3}$ of the boundary value problem 
\[
\begin{cases}
\Delta\Gamma & =0\\
\Gamma_{|S^{1}} & =\gamma
\end{cases}.
\]
Now, consider again the parametrisation of the unknot $\gamma_{\mathrm{unknot}}:S^{1}\to\mathbb{R}^{3}$.
It is easy to check that the unique harmonic extension of this map
is 
\[
\Gamma_{\mathrm{unknot}}:\left(r,\theta\right)\mapsto\left(r\cos\theta,r\sin\theta,0\right).
\]
Therefore, the totally geodesic embedding $u_{\HH^{2}}:\overline{\HH}^{2}\to\overline{\HH}^{4}$
can be written as 
\[
u\left(\theta,r\right)=\left(\frac{1-r^{2}}{1+r^{2}},\frac{2\Gamma_{\mathrm{unknot}}}{1+r^{2}}\right).
\]
This leads us to the following definition: 
\begin{defn}
Let $\gamma:S^{1}\to\mathbb{R}^{3}$ be a smooth map. We define the
\emph{stereoharmonic extension} $\ext_{\mathrm{st}-\mathrm{h}}\left(\gamma\right):D^{2}\to\mathbb{R}^{3}$
of $\gamma$ as 
\[
\ext_{\mathrm{st}-\mathrm{h}}\left(\gamma\right):=\frac{2\Gamma}{1+r^{2}},
\]
where $\Gamma$ is the unique harmonic extension $D^{2}\to\mathbb{R}^{3}$
of $\gamma$. 
\end{defn}

If $\gamma=\gamma_{\mathrm{unknot}}$, $\rho=\rho_{\mathrm{st}}$,
and $\ext=\ext_{\mathrm{st}}$, then we have again 
\[
u_{\HH^{2}}=\left(\rho_{\mathrm{st}},\ext_{\mathrm{st}-\mathrm{h}}\left(\gamma_{\mathrm{unknot}}\right)\right).
\]
However, this time, the initial approximation is smooth for \emph{every
}smooth parametrisation $\gamma$ of an oriented knot. Therefore,
we can reasonably expect to be able find near-minimal p-immersions
bounding knots different from the unknot.

While the stereoharmonic extension brings much better results than
the na{\"i}ve stereographic extension, it has an important drawback. Recall
from Theorem \ref{thm:jared} that a minimal p-submanifold of $\overline{\HH}^{n+1}$
must intersect the boundary at infinity \emph{orthogonally}. In fact,
a p-submanifold $\Sigma\subset\overline{\HH}^{n+1}$ is orthogonal
to the boundary if and only if $\Sigma$ is \emph{asymptotically minimal}.
More precisely, let $\iota:\Sigma\to\overline{\HH}^{n+1}$ be the
inclusion, which is by construction a p-embedding. As explained in
\S\ref{subsec:Conformally-compact-manifolds}, the tension field
$\tau\left(\iota\right)=\tr_{\iota^{*}g_{\mathrm{hyp}}}\nabla d\iota$
is a vector field on $\overline{\HH}^{4}$ along $\Sigma$, vanishing
along the boundary of $\Sigma$. In other words, it is a smooth section
of $\iota^{*}{^{0}T\overline{\HH}^{4}}$. This bundle is equipped
with a metric, namely the pull-back of the hyperbolic metric seen
as a smooth metric on $^{0}T\overline{\HH}^{4}$. The pointwise squared
length of the tension field $\left|\tau\left(\iota\right)\right|^{2}$
is therefore a smooth function on $\Sigma$, \emph{which does not
necessarily vanish along the boundary}; the fact that $\tau\left(\iota\right)$
vanishes along the boundary is counterbalanced exactly by the fact
that the hyperbolic metric \emph{blows up} along the boundary. It
is well-known that $\Sigma$ is orthogonal to the boundary if and
only if $\Sigma$ is \emph{asymptotically minimal}, that is, $\left|\tau\left(\iota\right)\right|^{2}$
vanishes identically along the boundary.

Now consider the simple $b$-maps produced by our model, 
\[
u_{\gamma,\rho,\ext,k,\mathrm{NN}}=\left(\rho\exp\left(\mathrm{NN}^{X}\right),\ext\left(\gamma\right)+\rho^{k}\mathrm{NN}^{\boldsymbol{Y}}\right).
\]
Since $\gamma:S^{1}\to\mathbb{R}^{3}$ is an embedding, $u_{\gamma,\rho,\ext,k,\mathrm{NN}}$
is locally a p-embedding near the boundary of $\partial D^{2}$. However,
it is not guaranteed that $u_{\gamma,\rho,\ext,k,\mathrm{NN}}\left(D^{2}\right)$
intersects the boundary of $\overline{\HH}^{4}$ orthogonally. We can design the $\ext$ operator in order to make this happen:

\begin{lem}
\label{lem:normal-derivative}Let $\gamma:S^{1}\to\mathbb{R}^{n}$
be an embedding. Consider a simple $b$-map 
\begin{align*}
u:D^{2} & \to\mathbb{R}_{\geq0}\times\mathbb{R}^{n}\\
u & =\left(\rho,v\right)
\end{align*}
where $\rho$ is a boundary defining function for $D^{2}$ and $v:D^{2}\to\mathbb{R}^{n}$
is a smooth map such that $v_{|S^{1}}=\gamma$. Then the image $u\left(D^{2}\right)$,
which is an embedded p-submanifold of $\mathbb{R}_{\geq0}\times\mathbb{R}^{n}$
in a sufficiently small neighbourhood $[0,\varepsilon)\times\mathbb{R}^{n}$
of the boundary, is orthogonal to the boundary $\mathbb{R}^{n}$ if
and only if the normal derivative of $v$ along $S^{1}$ is tangent
to the curve $\gamma\left(S^{1}\right)\subset\mathbb{R}^{n}$. 
\end{lem}


\begin{proof}
Call again $X,\boldsymbol{Y}$ the coordinates on the half-space $\mathbb{R}_{\geq0}\times\mathbb{R}^{n}$.
Call $\partial_{\theta}$ the canonical counterclockwise unit vector
field on $D^{2}$ along $S^{1}$ tangent to the boundary, and call
$\partial_{r}$ the unit vector field on $D^{2}$ along $S^{1}$ normal
and outward-pointing. $u\left(D^{2}\right)$ is orthogonal to the
boundary $X=0$ if and only if, for every $p\in S^{1}$, the space
$du_{p}\left(T_{p}D^{2}\right)=\mathrm{span}\left(\left\{ du_{p}\left(\partial_{r}\right),du_{p}\left(\partial_{\theta}\right)\right\} \right)$
contains the normal vector $\partial_{X}$. By construction, we have
$du_{p}\left(\partial_{r}\right)=\left(\partial_{r}\rho\right)\partial_{X}+\partial_{r}v$
and $du_{p}\left(\partial_{\theta}\right)=\left(\partial_{\theta}\rho\right)\partial_{X}+\partial_{\theta}v$.
Since $\rho$ is a boundary defining function, $\rho=0$ along the
boundary of $D^{2}$, and therefore $\partial_{\theta}\rho\equiv0$.
On the other hand, since $\rho=\left(1-r^{2}\right)e^{\varphi}$ for
some smooth function $\varphi$, we have on the boundary $\partial_{r}\rho=-c$ for
some strictly positive function $c$. Moreover, since $v_{|S^{1}}=\gamma$, we have $\partial_{\theta}v=\dot{\gamma}\left(p\right)$.
Summarising, 
\[
du_{p}\left(T_{p}D^{2}\right)=\mathrm{span}\left(-c\partial_{X}+\partial_{r}v,\dot{\gamma}\left(p\right)\right).
\]
Therefore the only possibility for which $\partial_{X}\in du_{p}\left(T_{p}D^{2}\right)$
is that either $\partial_{r}v=0$ or $\partial_{r}v=\mu\dot{\gamma}\left(p\right)$
for some non-zero $\mu$. 
\end{proof}
We now apply this lemma to our case. If we want the image of the simple
$b$-map 
\[
u_{\gamma,\rho,\ext,k,\mathrm{NN}}=\left(\rho\exp\left(\mathrm{NN}^{X}\right),\ext\left(\gamma\right)+\rho^{k}\mathrm{NN}^{\boldsymbol{Y}}\right)
\]
to be orthogonal\footnote{Note that orthogonality to the boundary of $\overline{\HH}^{4}$ is
equivalent, in half-space coordinates, to orthogonality to the boundary
with respect to the Euclidean metric.} to the boundary of $\overline{\HH}^{4}$, we need $\partial_{r}\left(\ext\left(\gamma\right)+\rho^{k}\mathrm{NN}^{\boldsymbol{Y}}\right)$
to be tangent to $\gamma$. Observe that $\partial_{r}\left(\rho^{k}\right)=k\rho^{k-1}\partial_{r}\rho$,
so if $k>1$ the normal derivative above reduces to the condition
that $\partial_{r}\left(\ext\left(\gamma\right)\right)$ is proportional
to the velocity vector $\dot{\gamma}$.

This leads us to our refined ansatz. Recall that
a \emph{biharmonic function} on $D^{2}$ is a solution of the PDE
$\Delta^{2}f=0$. Clearly, harmonic functions are biharmonic; however,
the PDE $\Delta^{2}f=0$ is of fourth order, and one can find a unique
solution $f$ subject to \emph{both} Dirichlet and Neumann boundary
conditions. More precisely, given two smooth functions $\gamma,\delta:S^{1}\to\mathbb{R}^{n}$,
there exists a unique smooth solution $\Gamma:D^{2}\to\mathbb{R}^{n}$
of the boundary value problem 
\[
\begin{cases}
\Delta^{2}\Gamma & =0\\
\Gamma_{|S^{1}} & =\gamma\\
\partial_{r}\Gamma_{|S^{1}} & =\delta
\end{cases}.
\]

\begin{defn}
Let $\gamma:S^{1}\to\mathbb{R}^{3}$ be a smooth embedding. We define
the \emph{stereobiharmonic extension} $\ext_{\mathrm{st}-\mathrm{bh}}\left(\gamma\right):D^{2}\to\mathbb{R}^{3}$
of $\gamma$ as 
\[
\ext_{\mathrm{st}-\mathrm{bh}}\left(\gamma\right):=\frac{2\Gamma}{1+r^{2}},
\]
where $\Gamma$ is the unique biharmonic extension $D^{2}\to\mathbb{R}^{3}$
of $\gamma$ which satisfies $\partial_{r}\Gamma_{|S^{1}}=\gamma$. 
\end{defn}

This definition is justified by the following 
\begin{prop}
Let $\gamma:S^{1}\to\mathbb{R}^{3}$ be an embedding. Choose $\ext=\ext_{\mathrm{st}-\mathrm{bh}}$
and $k=2$. Then every instance of the model $u_{\gamma,\rho,\ext,k,\mathrm{NN}}$
(for any choice of the boundary defining function $\rho$) defines,
near the boundary of $\overline{\HH}^{4}$, a p-embedded surface orthogonal
to the boundary. 
\end{prop}

\begin{proof}
The $\boldsymbol{Y}$-component of $u_{\gamma,\rho,\ext,k,\mathrm{NN}}$
is $\frac{2\Gamma}{1+r^{2}}+\rho^{2}\mathrm{NN}^{\boldsymbol{Y}}$.
Taking the $\partial_{r}$ derivative and observing that $\partial_{r}\rho^{2}=\left(2\rho\right)\partial_{r}\rho$
vanishes along the boundary, we obtain along the boundary 
\begin{align*}
\partial_{r}\left(\frac{2\Gamma}{1+r^{2}}+\rho^{2}\mathrm{NN}^{\boldsymbol{Y}}\right) & =\left(\partial_{r}\left(\frac{2}{1+r^{2}}\right)\Gamma+\frac{2}{1+r^{2}}\partial_{r}\Gamma\right)_{|r=1}\\
 & =\left(-\frac{4r}{\left(1+r^{2}\right)^{2}}\Gamma+\frac{2}{1+r^{2}}\partial_{r}\Gamma\right)_{|r=1}\\
 & =-\gamma+\gamma\\
 & =0.
\end{align*}
The orthogonality claim then follows from Lemma \ref{lem:normal-derivative}. 
\end{proof}
The previous proposition clarifies how to choose the hyperparameter
$k$, namely the order of decay of the factor $\rho^{k}$ in the definition
of $u_{\gamma,\rho,\ext,k,\mathrm{NN}}$: if the chosen extension
operator $\ext$ does \emph{not} guarantee the image of $u_{\gamma,\rho,\ext,k,\mathrm{NN}}$
to be orthogonal to the boundary, we should choose $k=1$ in order
to allow $\mathrm{NN}$ to potentially correct for this; if the chosen
extension operator does guarantee orthogonality, we can choose $k=2$.
In practice, we experimentally noted that using the stereobiharmonic
extension operator $\ext_{\mathrm{st}-\mathrm{bh}}$ paired with the
stereographic boundary defining function $\rho_{\mathrm{st}}$ and
the decay exponent $k=2$ provides the best results.

To conclude, consider again the totally geodesic example 
\[
u_{\HH^{2}}\left(\theta,r\right)=\left(\frac{1-r^{2}}{1+r^{2}},\frac{2r\cos\theta}{1+r^{2}},\frac{2r\sin\theta}{1+r^{2}},0\right).
\]
The $\boldsymbol{Y}$-component of this map is 
\[
\frac{2}{1+r^{2}}\left(r\cos\theta,r\sin\theta,0\right)
\]
and 
\[
\partial_{r}\left(r\cos\theta,r\sin\theta,0\right)=\left(\cos\theta,\sin\theta,0\right)=\gamma.
\]
Since $\Gamma:\left(r,\theta\right)\mapsto\left(r\cos\theta,r\sin\theta,0\right)$
is harmonic, it is in particular biharmonic, and therefore it solves
the boundary value problem 
\[
\begin{cases}
\Delta^{2}\Gamma & =0\\
\Gamma_{|S^{1}} & =\gamma\\
\partial_{r}\Gamma_{|S^{1}} & =\gamma
\end{cases}.
\]
In other words, if we choose $\rho=\rho_{\mathrm{st}}$, $\ext=\ext_{\mathrm{st}-\mathrm{bh}}$,
$k=2$, and $\gamma=\left(\cos\theta,\sin\theta,0\right)$, then the
instance of the model $u_{\gamma,\rho,\ext,k,\mathrm{NN}}$ obtained
by setting all the parameters of $\mathrm{NN}$ to zero is again the
totally geodesic solution $u_{\HH^{2}}$.

\subsection{\label{subsec:The-training-regime}The training regime}

We now describe how the parameters $\theta$ of the network $\mathrm{NN}:\mathbb{R}^{2}\to\mathbb{R}^{4}$
are optimised so that the corresponding $b$-map $u_{\gamma,\rho,\ext,k,\mathrm{NN}}$
approximates a minimal p-immersion. Let us denote $u_{\gamma,\rho,\ext,k,\mathrm{NN}}$
by $u_{\theta}$ for simplicity in what follows. As explained in the
previous subsection, the boundary condition is satisfied exactly for
every $\theta$, so the training is governed entirely by the interior
PDE loss, namely 
\begin{equation}
\mathcal{L}(u_{\theta};\mathcal{D})\;=\;\frac{1}{N}\sum_{i=1}^{N}\bigl|\tau\left(u_{\theta}\right)(p_{i})\bigr|^{2},\label{eq:training_loss}
\end{equation}
where $\mathcal{D}:=\{p_{i}\}_{i=1}^{N}$ is a (pseudo-)random sample
of interior collocation points in $D^{2}$ drawn at the beginning
of training. As explained in \S\ref{subsec:Minimal-p-submanifolds-of-Hn},
since $u_{\theta}$ is a p-immersion near the boundary, $\tau\left(u_{\theta}\right)$
is a smooth section of $u_{\theta}^{*}{^{0}T\overline{\HH}^{4}}$;
the pointwise norm above is therefore computed with respect to the
hyperbolic metric, namely the bundle metric on $u_{\theta}^{*}{^{0}T\overline{\HH}^{4}}$.

Thanks to the simple form of the hyperbolic metric in half-space coordinates,
we can spell out in detail the expression of $\tau\left(u_{\theta}\right)$.
In general, a section $\nu$ of ${^{0}T\overline{\HH}^{4}}$ can be
written in half-space coordinates $\left(X,Y_{1},Y_{2},Y_{3}\right)$
as 
\[
\nu=\nu^{X}X\partial_{X}+\sum_{i=1}^{3}\nu^{Y_{i}}X\partial_{Y_{i}},
\]
where the $\nu^{X},\nu^{Y_{i}}$ are smooth functions. Since the vector
fields $X\partial_{X}$, $X\partial_{Y_{i}}$ are orthonormal with
respect to the hyperbolic metric, we have 
\[
\left|\nu\right|^{2}=\left(\nu^{X}\right)^{2}+\sum_{i=1}^{3}\left(\nu^{Y_{i}}\right)^{2}.
\]
Therefore, writing our tension field $\tau\left(u_{\theta}\right)$
as $\tau^{X}\left(u_{\theta}\right)X\partial_{X}+\sum_{i}\tau^{Y_{i}}\left(u_{\theta}\right)X\partial_{Y_{i}}$,
we have 
\[
\mathcal{L}\left(\theta\right)=\left(\tau^{X}\left(u_{\theta}\right)\right)^{2}+\sum_{i=1}^{3}\left(\tau^{Y_{i}}\left(u_{\theta}\right)\right)^{2}.
\]
Letting $u_{\theta}=(X_{\theta},\boldsymbol{Y}_{\theta})$, where
$X_{\theta}=\rho\exp(\mathrm{NN}^{X})\geq0$ and $\boldsymbol{Y}_{\theta}=(Y_{\theta,1},Y_{\theta,2},Y_{\theta,3})$,
elementary computations yield 
\begin{align}
\tau^{X}(u_{\theta}) & \;=\;\frac{1}{X_{\theta}}\Bigl[\Delta_{g}X_{\theta}+\frac{1}{X_{\theta}}\Bigl(\sum_{k=1}^{n}|dY_{\theta,k}|_{g}^{2}-|dX_{\theta}|_{g}^{2}\Bigr)\Bigr],\label{eq:tau_X}\\[6pt]
\tau^{Y_{k}}(u_{\theta}) & \;=\;\frac{1}{X_{\theta}}\Bigl[\Delta_{g}Y_{\theta,k}-\frac{2}{X_{\theta}}\,\langle dX_{\theta},\,dY_{\theta,k}\rangle_{g}\Bigr],\quad k=1,2,3.\label{eq:tau_Y}
\end{align}
Here $g$ is the pull-back of the hyperbolic metric, and $\Delta_{g}$
is the Laplace--Beltrami operator associated to $g$, acting on a
scalar function $f:D^{2}\to\mathbb{R}$ explicitly by 
\begin{equation}
\Delta_{g}f\;=\;g^{ab}\,\partial_{ab}^{2}f\;+\;\bigl(\partial_{a}g^{ab}\bigr)\partial_{b}f\;+\;g^{ab}\bigl(\partial_{a}\log\sqrt{\det g}\bigr)\partial_{b}f\,.\label{eq:LB}
\end{equation}

With this loss in place, the training pipeline consists of two sequential
phases: a stochastic gradient descent phase with Adam, and a full-batch
quasi-Newton refinement with L-BFGS. We describe each in turn, and
then record the concrete hyperparameter choices used in our experiments
at the end of this section.

The main training phase minimises $\mathcal{L}(\theta)$ in \eqref{eq:training_loss}
using the Adam optimiser \cite{kingma2014adam}, a stochastic first-order
method with adaptive per-coordinate learning rates. The learning rate
is annealed according to the cosine schedule 
\begin{equation}
\eta_{t}\;=\;\eta_{\min}+\tfrac{1}{2}(\eta_{0}-\eta_{\min})\Bigl(1+\cos\!\Bigl(\tfrac{\pi t}{T}\Bigr)\Bigr),\qquad t=0,1,\ldots,T,\label{eq:cosine_lr}
\end{equation}
where $\eta_{0}$ is the initial learning rate, $\eta_{\min}=10^{-2}\,\eta_{0}$
is the terminal learning rate, and $T$ is the total number of epochs.
At the beginning of this phase, we draw a batch of $N_{\mathrm{data}}$
interior collocation points, at which the PDE residual is evaluated.
These points are drawn by setting $r=U^{1/2}$, $\varphi=2\pi V$
with $U,V$ being uniform distributions on the interval $(0,1)$,
and returning $(r\cos\varphi,r\sin\varphi)$, yielding the standard
uniform distribution\footnote{The option to introduce a parametric exponent $b\in(0,1]$ controlling
the radial bias ($b=1/2$ recovers the uniform distribution, while
$b\to0$ concentrates mass near the boundary) is also available in
the codebase, but not used; as well as a more targeted distribution
which adds to the uniformly generated background points
a targeted sampling around a given point in $D^{2}$.} on $D^{2}$. Then, at each epoch, the collocation points are partitioned into mini-batches
of size $B$, and one Adam step per mini-batch is performed. After
every epoch the scheduler advances \eqref{eq:cosine_lr} and the epoch-average
loss is recorded; the parameter snapshot achieving the lowest epoch-average
loss over the entire run is retained and restored at the end of training.

Once Adam has brought the loss to a moderate level, we switch to a
full-batch quasi-Newton refinement using the limited-memory Broyden--Fletcher--Goldfarb--Shanno
method (L-BFGS) \cite{liu1989limited,nocedal1980updating}. Unlike
Adam, L-BFGS uses curvature information accumulated over the last
$m$ iterates to approximate the inverse Hessian of $\mathcal{L}$;
the resulting direction is a much better approximation to the Newton
step and can achieve superlinear convergence near a smooth minimiser.
The full-batch loss is evaluated on a new uniform sample of $N_{\mathrm{LBFGS}}$
interior collocation points, drawn once at the beginning of the L-BFGS
phase. The line search is the strong Wolfe criterion \cite{wolfe1969convergence},
which guarantees sufficient descent and curvature conditions at each
step. The iterations terminate when either the gradient norm falls
below $\delta_{g}=10^{-12}$, the parameter change falls below $\delta_{\theta}=10^{-14}$,
or the maximum number of steps $T_{\mathrm{LBFGS}}$ is reached.

All experiments are run in double precision (\texttt{float64}). The
PDE residual involves the Hessian of the map, so second-order derivatives
are computed by composing forward-mode and reverse-mode AD. In single
precision, the additional rounding in this composition causes the
residual to saturate at a level of approximately $10^{-4}$ even for
a near-exact solution; double precision brings this floor down to
$10^{-12}$, which is the effective gradient-norm tolerance we can
target for future improvements (cf. \S\ref{sec:Future_developments}).

The interior model $\mathrm{NN}:\mathbb{R}^{2}\to\mathbb{R}^{4}$
is an MLP in the sense of Definition \ref{def:MLP} with $\tanh$
activation functions, input dimension $d_{0}=2$, output dimension
$d_{L}=4$, depth $L$, and constant width $d_{1}=\cdots=d_{L-1}=W$.
The defaults used in our experiments are $L=4$ hidden layers and
$W=64$ units per layer, giving $(2 \times 64 + 64) + 3 \times (64 \times 64 + 64) + (64 \times 4 + 4) = 12932$ 
trainable parameters. We find that this architecture is sufficient
for the purpose of the investigation presented in this paper. For
reference, Table \ref{tab:hyperparams} collects the default hyperparameter
values used across all experiments reported in this paper. Departures
from these defaults are noted in the captions of the individual figures. The concrete implementation of the neural network and training strategy discussed above, in \texttt{torch}, is available at the \href{https://github.com/Tancredi-Schettini-Gherardini/deep_plateau.git}{GitHub repository}. We note that a full training (Adam and L-BFGS) takes roughly an hour on a personal laptop, running only on CPU.

\begin{table}[ht]
\centering 
\global\long\def\arraystretch{1.25}%
\begin{tabular}{llll}
\hline 
\textbf{Symbol}  & \textbf{Description}  & \textbf{Default value}  & \textbf{Phase} \tabularnewline
\hline 
$W$  & Hidden-layer width  & $64$  & All \tabularnewline
$L$  & Depth (number of hidden layers)  & $4$  & All \tabularnewline
$\sigma$  & Activation function  & $\tanh$  & All \tabularnewline
$N_{{\rm data}}$  & Collocation pool size  & $2^{14}$  & 1, 2 \tabularnewline
$B$  & Mini-batch size  & $2^{10}$  & 1\tabularnewline
$T_{2}$  & Adam epochs  & $10000$  & 1\tabularnewline
$\eta_{0}$  & Initial learning rate  & $10^{-3}$  & 1 \tabularnewline
$\eta_{\min}$  & Minimum learning rate  & $10^{-5}$  & 1\tabularnewline
$N_{{\rm LBFGS}}$  & L-BFGS collocation size  & $2^{14}$  & 2\tabularnewline
$T_{{\rm LBFGS}}$  & L-BFGS max iterations  & $10000$  & 2\tabularnewline
$m$  & L-BFGS history size  & $100$  & 2\tabularnewline
$\delta_{g}$  & Gradient-norm tolerance  & $10^{-12}$  & 2\tabularnewline
$\delta_{\theta}$  & Parameter-change tolerance  & $10^{-14}$  & 2\tabularnewline
\hline 
\end{tabular}\caption{Default hyperparameter values used in all experiments. All experiments
run in double precision (\texttt{float64}).}
\label{tab:hyperparams} 
\end{table}

\subsection{\label{subsec:Double-point-analysis}Double point analysis}

As explained in \S\ref{subsec:Fine's-conjecture}, for a generic
knot $K\subset S^{3}$ the p-immersed minimal surfaces in $\overline{\HH}^{4}$
bounding $K$ intersect themselves transversely at a finite number
of double points, all contained in the interior. Now, our trained model instances produce parametrised, near-minimal,
p-immersions $u_{\theta}:D^{2}\to\overline{\HH}^{4}$. In this subsection,
we explain how to detect the self-intersections of $u_{\theta}\left(D^{2}\right)$
and compute the self-intersection number.

The method is conceptually simple. Since our model produces maps from
$D^{2}$ to the half-space model of hyperbolic $4$-space, we can
think of them as immersions $D^{2}\to\mathbb{R}^{4}$. Now, consider
a generic immersion $u:D^{2}\to\mathbb{R}^{4}$ such that $u\left(D^{2}\right)$
self-intersects transversely at a finite number of double points.
Then the pairs $\left(p_{1},p_{2}\right)\in D^{2}\times D^{2}$ mapped
by $u$ to double points are exactly the zeroes of the map 
\begin{align*}
F_{u}:D^{2}\times D^{2} & \to\mathbb{R}^{4}\\
\left(p_{1},p_{2}\right) & \mapsto u\left(p_{1}\right)-u\left(p_{2}\right)
\end{align*}
away from the diagonal $\Delta\subset D^{2}\times D^{2}$. We can
then frame the problem of finding double points as the problem of
solving the equation $F_{u}\left(p_{1},p_{2}\right)=0$; this can
easily be done by Newton's method. However, this method requires an
`initial guess', and very different initial guesses may (and should)
lead to different double points. Therefore, we divide the double point
analysis in two phases: \emph{candidate generation} and \emph{refinement
to true solutions} (up to machine precision).

\subsubsection{Candidate generation}

Our method to generate candidate double points relies on the notion
of \emph{self-proximity map}: 
\begin{defn}
\label{def:self-proximity-map}The \emph{self-proximity map} of $u:D^{2}\to\mathbb{R}^{4}$
at threshold $\varepsilon>0$ is the map 
\begin{align*}
\mu_{\varepsilon}:D^{2} & \to[0,+\infty)\\
p & \mapsto\inf_{\begin{matrix}p'\in D^{2}\\
\left|\left|p-p'\right|\right|>\varepsilon
\end{matrix}}\left|\left|u\left(p\right)-u\left(p'\right)\right|\right|.
\end{align*}
\end{defn}


\begin{rem}
For simplicity, we use the flat metrics on both $D^{2}$
and $\mathbb{R}^{4}$. 
\end{rem}

By construction, the self-proximity map is very small at points where
the surface is very close to itself. In particular, two points $p_{1},p_{2}\in D^{2}$
with $\left|\left|p_{1}-p_{2}\right|\right|>\varepsilon$ satisfy
$u\left(p_{1}\right)=u\left(p_{2}\right)$ if and only if $\mu_{\varepsilon}\left(p_{1}\right)=\mu_{\varepsilon}\left(p_{2}\right)=0$.
Concretely, if all the double points of $u:D^{2}\to\mathbb{R}^{4}$
have pre-image pair $\left(p_{1},p_{2}\right)$ satisfying $\left|\left|p_{1}-p_{2}\right|\right|>\varepsilon$,
then a heat-map of the function $\mu_{\varepsilon}$ should exhibit
a finite number of isolated dark regions, one for each double point
pre-image. This is exactly what we see in our experiments (cf. for
example Figure \ref{fig:3_1}).

The algorithm to find candidate pairs $\left(p,p'\right)$ solving
$F_{u}\left(p,p'\right)=0$ is the following: 
\begin{enumerate}
\item generate a regularly spaced grid $\left\{ p_{i}\right\} $ of points
in the interior of $D^{2}$; 
\item choose a domain threshold $\varepsilon>0$ and a codomain threshold
$\tau>0$; 
\item iterate over the pairs $\left(p_{i},p_{j}\right)$, retaining as candidates
only those for which $\left|\left|p_{i}-p_{j}\right|\right|>\varepsilon$
and $\left|\left|u\left(p_{i}\right)-u\left(p_{j}\right)\right|\right|<\tau$. 
\end{enumerate}
The domain threshold $\varepsilon>0$ is needed in order to prevent
the algorithm from choosing candidates very close to the diagonal.
The codomain threshold $\tau>0$ is needed in order to prevent the
algorithm from generating too many candidates. Typically, the candidates
cluster together around the regions where $\mu_{\varepsilon}$ is
small (cf. Figure \ref{fig:3_1}).

\subsubsection{Refinement via Newton's method}

Now that we have a list of candidate double points $\left\{ \left(p_{i},p_{j}\right)\right\} $,
we simply run the Newton algorithm to solve $F_{u}\left(p,p'\right)=0$,
once per candidate. Recall that the convergence of the Newton method
is quadratic, so this procedure is computationally cheap. \emph{We
remark that the solutions found by the Newton method are not constrained
to be in the initial grid}. Typically, the Newton algorithm starting
at a candidate double point $\left(p_{i},p_{j}\right)$ converges
to a pair $\left(p,p^{*}\right)\in D^{2} \times D^{2}$ for which $u\left(p\right)\approx u\left(p^{*}\right)$
up to machine precision (the order of magnitude of $\left|\left|u\left(p\right)-u\left(p^{*}\right)\right|\right|$
is $10^{-16}$). As expected, many candidates converge to the same
double point, and therefore we identify solutions which agree up to
machine precision. 
\begin{rem}
Note that, a priori, the Newton algorithm could converge to a point
of the diagonal of $D^{2}\times D^{2}$; in practice, this never happens
in our experiments, because the candidates are well-separated, and
already close enough to a true double point. 
\end{rem}

In this phase, we also record the information necessary to determine
the \emph{sign} of the double point. First of all, recall that the
sign of a double point $\left(p_{1},p_{2}\right)$ of a generic immersion
$u:D^{2}\to\mathbb{R}^{4}$ depends on the choice of the orientation
of $\mathbb{R}^{4}$. We use the standard orientation on $\mathbb{R}^{4}$.
Using this convention, if $\left(p_{1},p_{2}\right)$ is the pre-image
pair of a double point of $u:D^{2}\to\overline{\HH}^{4}$, then its
sign coincides with the sign of the determinant of the Jacobian matrix
of $F_{u}$ at $\left(p_{1},p_{2}\right)$. The computation of the
determinant also functions as a sanity check for the transversality
of the self-intersection of $u\left(D^{2}\right)$ at the given double
point: if the intersection is truly transverse, then this determinant
should be significantly different from $0$ compared to machine
precision.

\section{\label{sec:Results}Results}

\begin{figure}[t]
\centering \begin{subfigure}[b]{0.24\textwidth} \centering \includegraphics[width=1\textwidth]{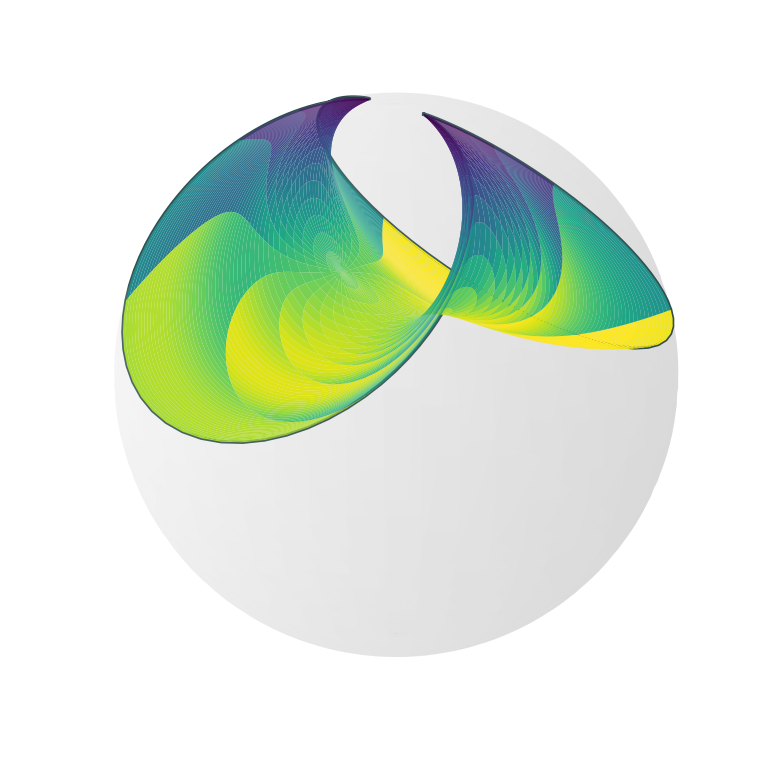}
\end{subfigure} \hfill{}\begin{subfigure}[b]{0.24\textwidth}
\centering \includegraphics[width=1\textwidth]{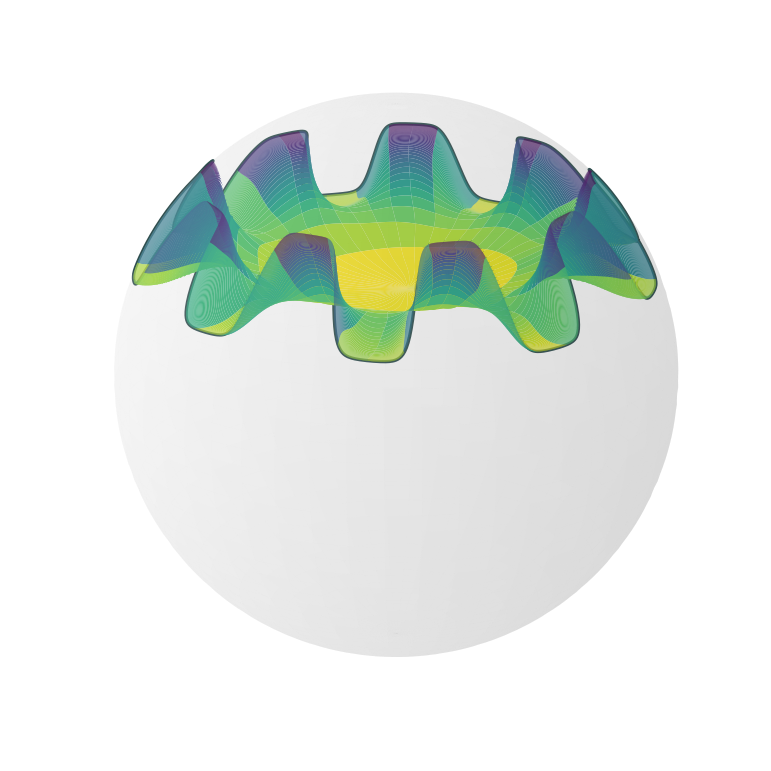}
\end{subfigure} \hfill{}\begin{subfigure}[b]{0.24\textwidth}
\centering \includegraphics[width=1\textwidth]{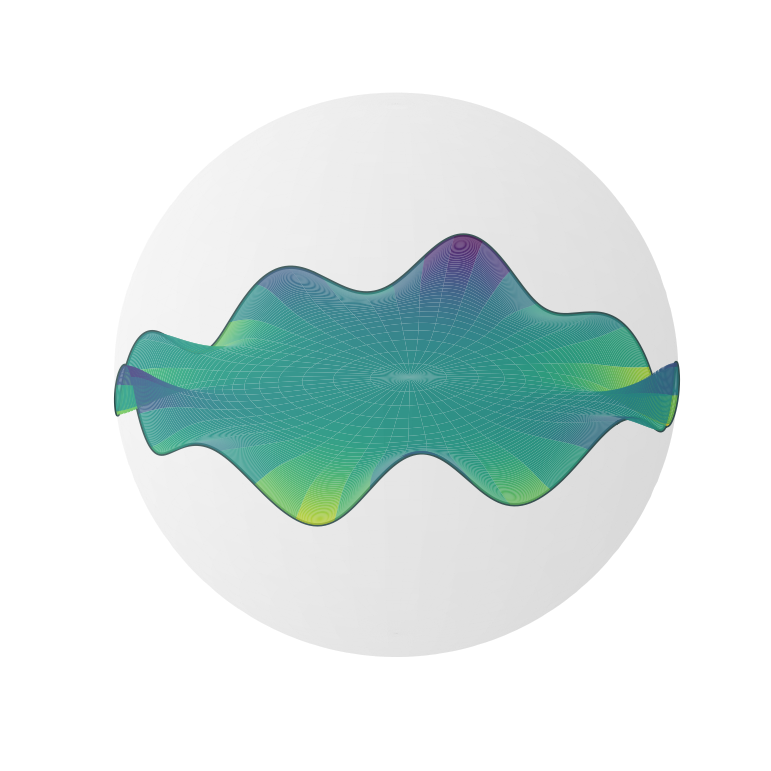}
\end{subfigure} \hfill{}\begin{subfigure}[b]{0.24\textwidth}
\centering \includegraphics[width=1\textwidth]{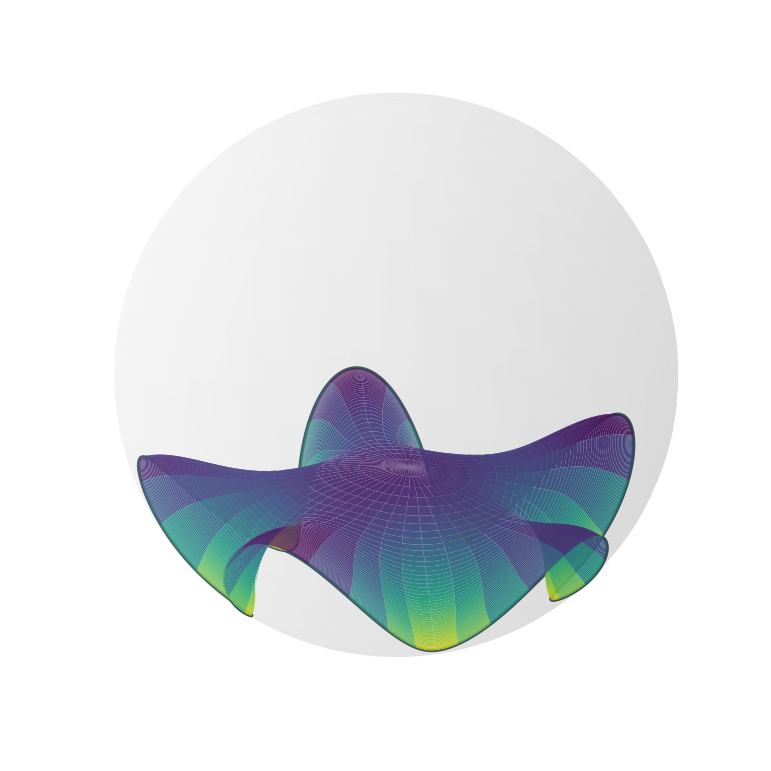}
\end{subfigure} \caption{Four minimal discs in $\HH^{3}$.}
\label{fig:minimal_discs_in_H3} 
\end{figure}

\begin{figure}[t]
\centering \begin{subfigure}[b]{0.24\textwidth} \centering \includegraphics[width=1\textwidth]{figures/knot_diagrams/3_1}
\caption{$3_{1}$}
\label{fig:diagram_3_1} \end{subfigure} \hfill{}\begin{subfigure}[b]{0.24\textwidth}
\centering \includegraphics[width=1\textwidth]{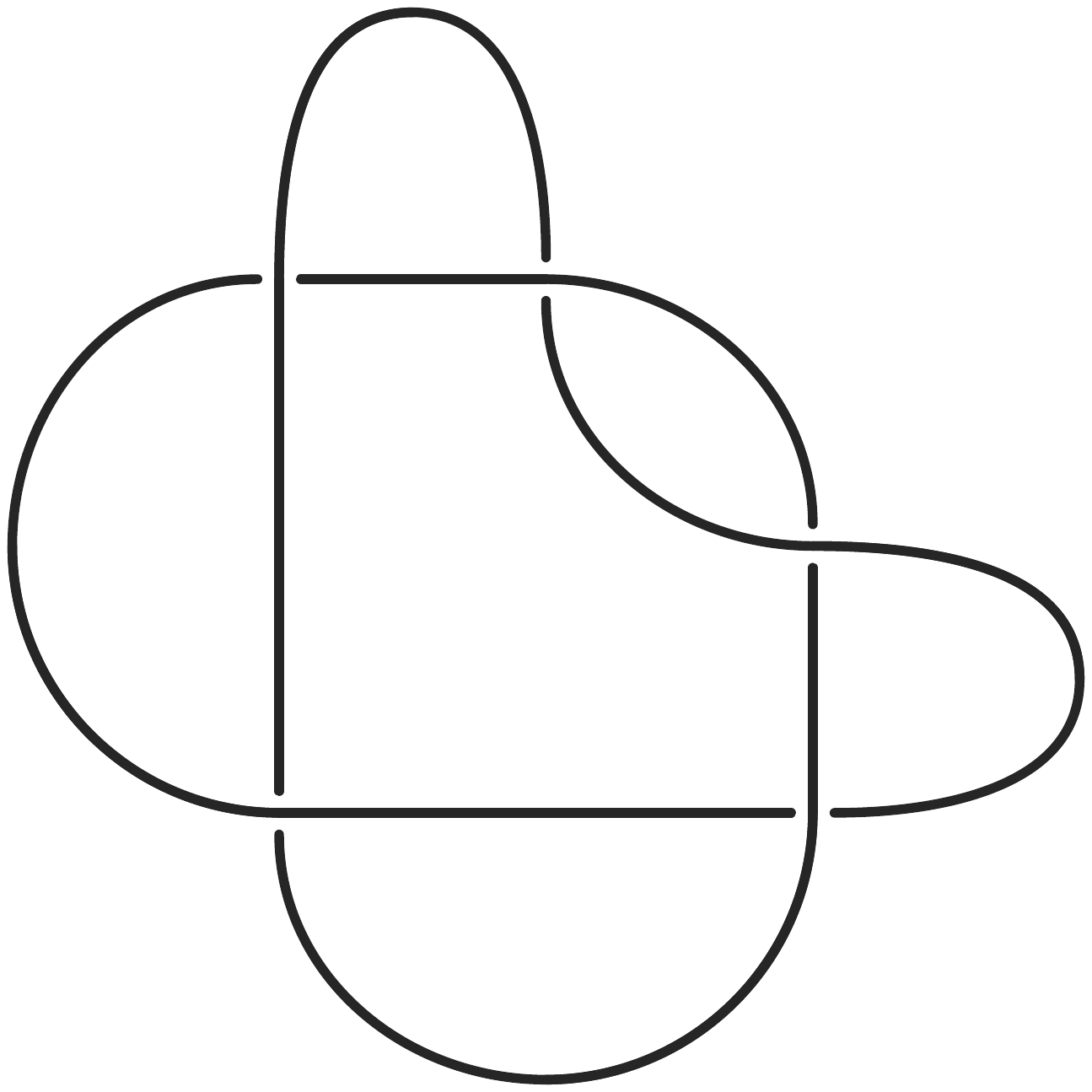}
\caption{$5_{1}$}
\label{fig:diagram_5_1} \end{subfigure} \hfill{}\begin{subfigure}[b]{0.24\textwidth}
\centering \includegraphics[width=1\textwidth]{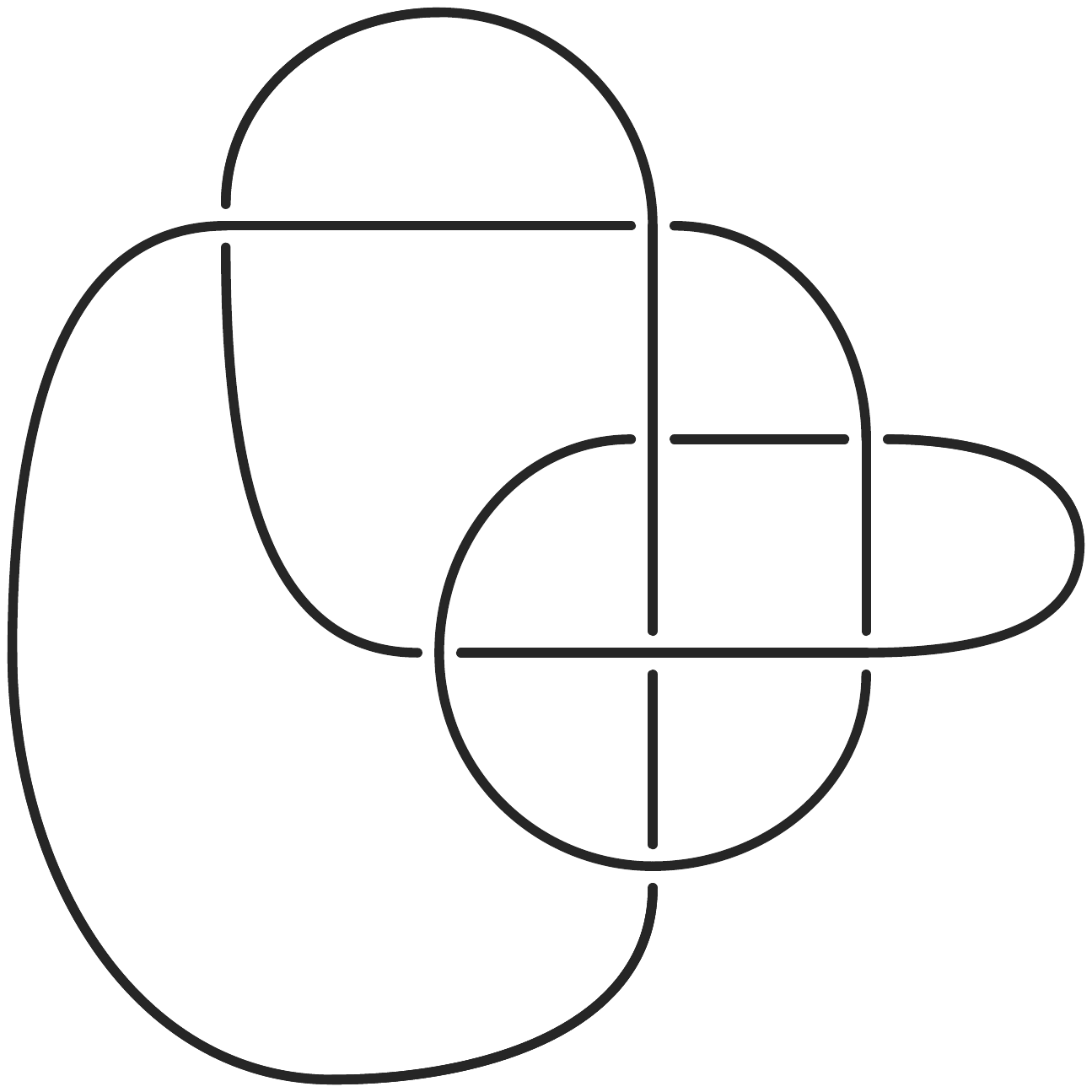}
\caption{$8_{19}$}
\label{fig:diagram_8_19} \end{subfigure} \hfill{}\begin{subfigure}[b]{0.24\textwidth}
\centering \includegraphics[width=1\textwidth]{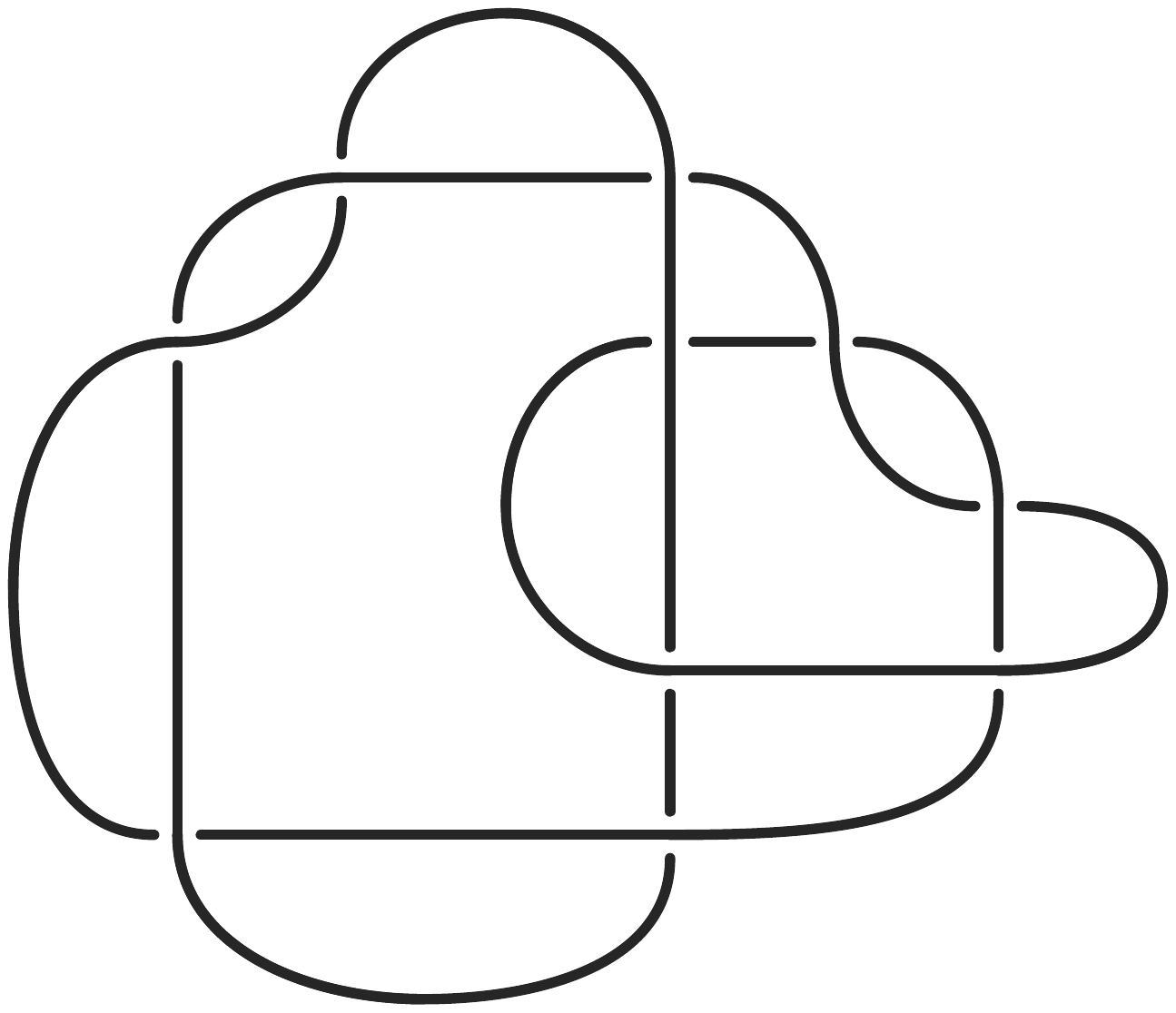}
\caption{$10_{124}$}
\label{fig:diagram_10_124} \end{subfigure}

\begin{subfigure}[b]{0.24\textwidth} \centering \includegraphics[width=1\textwidth]{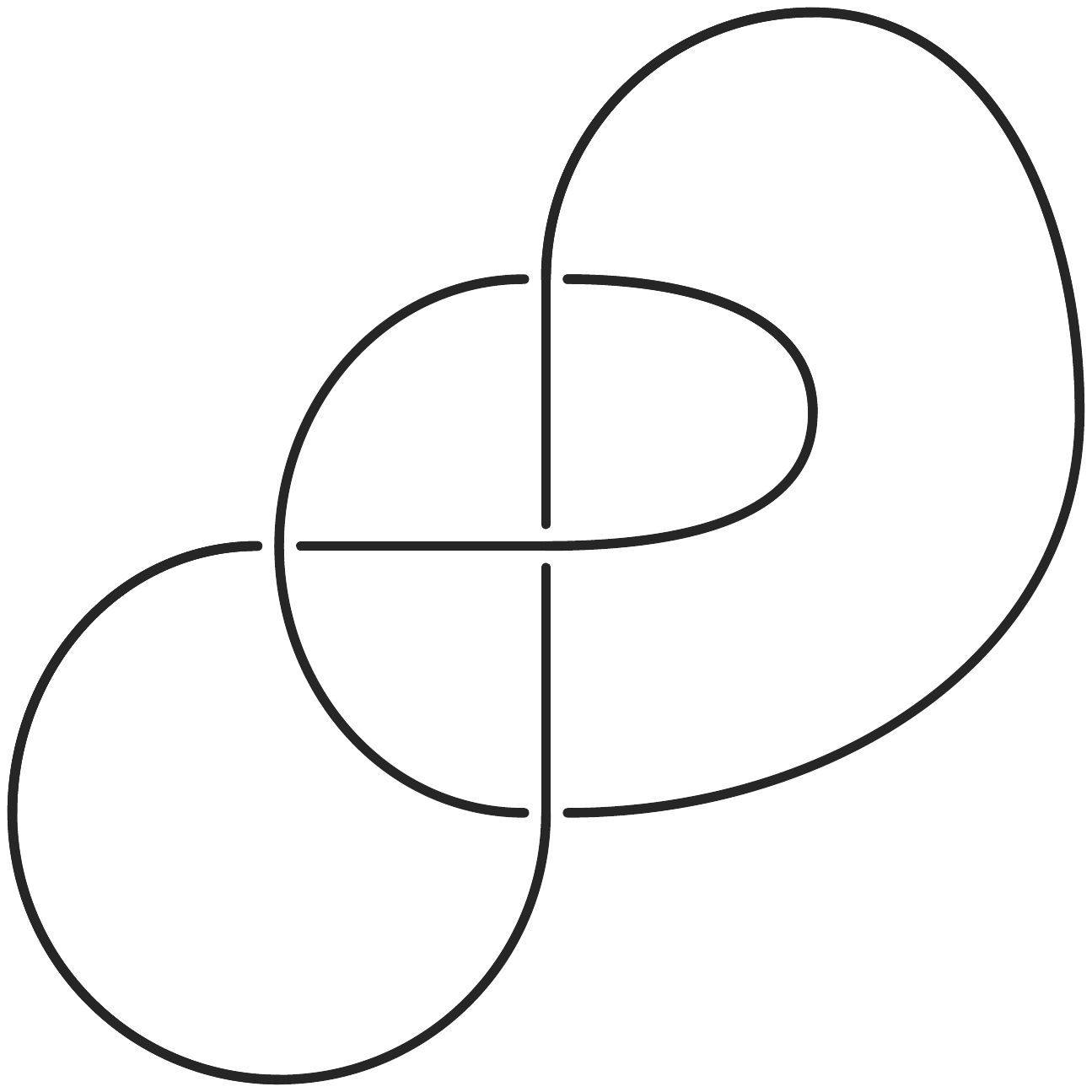}
\caption{$4_{1}$}
\label{fig:diagram_4_1} \end{subfigure} \hfill{}\begin{subfigure}[b]{0.24\textwidth}
\centering \includegraphics[width=1\textwidth]{figures/knot_diagrams/6_1}
\caption{$6_{1}$}
\label{fig:diagram_6_1} \end{subfigure} \hfill{}\begin{subfigure}[b]{0.24\textwidth}
\centering \includegraphics[width=1\textwidth]{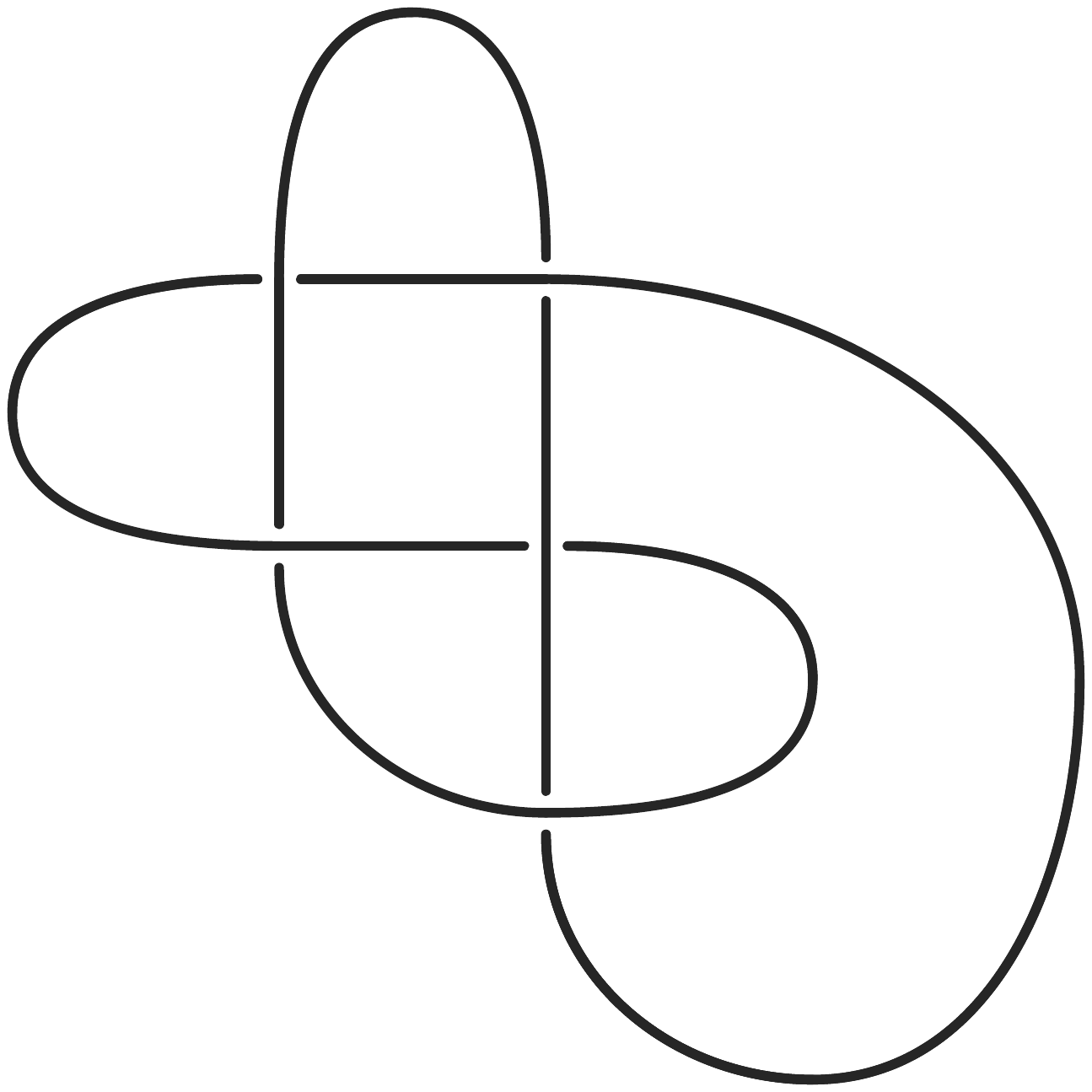}
\caption{$5_{2}$}
\label{fig:diagram_5_2} \end{subfigure} \hfill{}\begin{subfigure}[b]{0.24\textwidth}
\centering \includegraphics[width=1\textwidth]{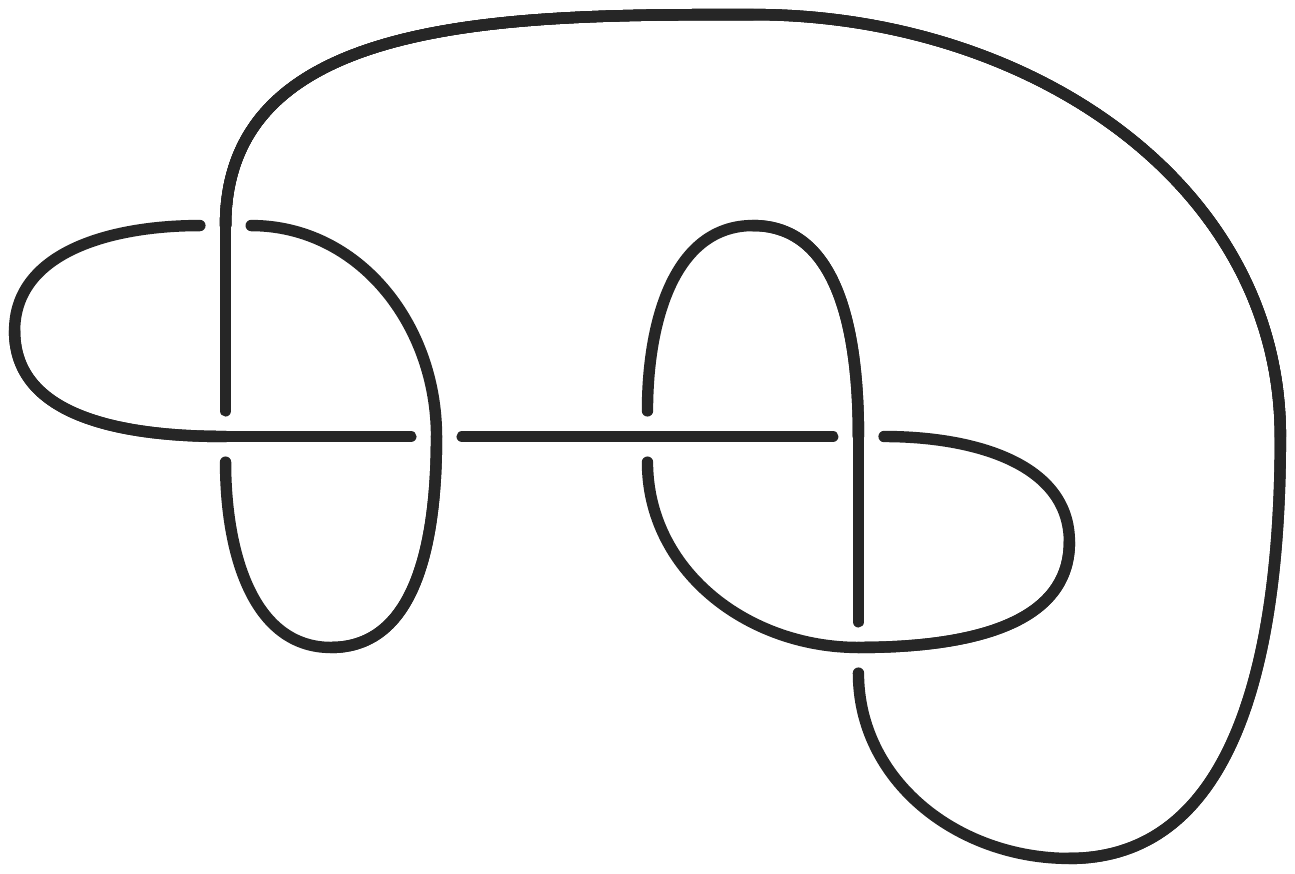}
\caption{$3_{1}\#3_{1}^{*}$}
\label{fig:diagram_square} \end{subfigure} \caption{Knot diagrams of the knots tested in this paper.}
\label{fig:knot_diagrams} 
\end{figure}

In this section, we discuss some of the solutions produced by our
method. The framework is, in principle, able to find p-immersed minimal
discs in $\overline{\HH}^{n+1}$, for any $n\geq2$; for example,
in Figure \ref{fig:minimal_discs_in_H3}, we show various
p-embedded minimal discs in $\overline{\HH}^{3}$ bounded by prescribed simple
closed curves in $S^{2}$, constructed using our framework. However, the main motivation
behind this work is testing Fine's Conjecture (cf. \S\ref{subsec:Fine's-conjecture}),
so all the examples discussed in detail below are p-immersed minimal
discs in $\overline{\HH}^{4}$. The parameters of the trained neural networks for each of these cases, as well as the others reported in Table \ref{tab:summary} but not discussed in detail, are available at the \href{https://github.com/Tancredi-Schettini-Gherardini/deep_plateau.git}{GitHub repository}, together with a notebook which reproduces all the plots below and shows further insights.

Let us discuss in detail how we test Fine's Conjecture. Given a knot $K\subset S^{3}$,
consider its HOMFLY polynomial, written in the form 
\[
P\left(K\right)=\sum_{\begin{smallmatrix}g\in\mathbb{N}\\
d\in\mathbb{Z}
\end{smallmatrix}}c_{g,d}z^{2g}a^{2\left(g+d\right)}.
\]
Fine's Conjecture asserts that, for $K$ generic, $c_{g,d}$ is precisely
the count with sign of minimal p-immersed surfaces in $\overline{\HH}^{4}$
of genus $g$ and self-intersection number $d$, meeting the sphere
at infinity at $K$. Therefore, if $c_{0,d}\not=0$, the conjecture
predicts the existence of minimal p-immersed discs meeting $K$ at
infinity, with self-intersection number $d$. In this section we
describe solutions found by our method, meeting the sphere at infinity
at the following knots: 
\begin{enumerate}
\item unknots; 
\item various torus knots; 
\item the Stevedore, Figure-eight, Three-twist, and square knot. 
\end{enumerate}
The diagrams of these knots are shown in Figure \ref{fig:knot_diagrams}.
For each solution, we compute the self-intersection number, following
the method described in \S\ref{subsec:Double-point-analysis}. The
upshot is the following: \emph{every solution we find is predicted
by Fine's Conjecture}. We remark that the solutions described here
are but a few of the examples we produced; the model exhibits
a remarkable flexibility, provided that the boundary knot can be explicitly
parametrised.

\subsubsection*{Knot deformations}

As discussed in \S\ref{subsec:Fine's-conjecture}, it is important
to allow our boundary knots to be \emph{generic}. On the other hand,
often the available explicit knot parametrisations exhibit non-generic
symmetry. This could be problematic, in that symmetric knots could
lead to non-generic solutions, exhibiting branched points or self-intersections
of higher multiplicity (an example is shown in \S\ref{subsec:Torus-knots}).
We solve this problem as follows. Given a knot parametrisation $\gamma:S^{1}\to\mathbb{R}^{3}$,
we \emph{deform} it by adding a perturbation term $\sigma\delta$,
with $\sigma\in\mathbb{R}^{+}$ a scale parameter and $\delta:S^{1}\to\mathbb{R}^{3}$
a map of the form 
\[
\delta\left(\theta\right)=\sum_{n=1}^{K}A_{n}\cos\left(n\theta\right)+B_{n}\sin\left(n\theta\right).
\]
Here $K\in\mathbb{N}$ is fixed, and the $A_{i},B_{i}\in\mathbb{R}^{3}$
are randomly chosen vectors with entries in $\left[0,1\right]$. For
small enough values of $\sigma$, the map $\gamma+\sigma\delta$ parametrises
a knot isotopic to $\gamma\left(S^{1}\right)$, and breaks the symmetry
of $\gamma$.

\subsubsection*{Evaluation of solutions}

Let us explain how we evaluate the goodness of a solution $u:D^{2}\to\overline{\HH}^{4}$
found by our model. In the training stage (cf. \S\ref{subsec:The-training-regime}),
we use the following loss function: given a random sample $\mathcal{D}=\left\{ p_{i}\right\} _{i=1}^{N}$
of $N$ points drawn from the interior of the unit disc, the loss
is 
\[
\mathcal{L}\left(u;\mathcal{D}\right)=\frac{1}{N}\sum_{i=1}^{N}\left|\tau\left(u\right)\right|^{2}\left(p_{i}\right),
\]
where $\tau\left(u\right)$ is the tension field of $u$ and $\left|\tau\left(u\right)\right|^{2}$
is its squared pointwise norm with respect to the hyperbolic metric.
At the evaluation stage, we perform a \emph{Monte Carlo simulation}
of the distribution of $\mathcal{D}\mapsto\mathcal{L}\left(u;\mathcal{D}\right)$:
we draw $S$ independent samples $\mathcal{D}_{i}$ of size $N$,
we compute $\mathcal{L}\left(u;\mathcal{D}_{i}\right)$, and we record
the average and standard deviation of these losses. This step is important
in order to make sure that the model is not \emph{overfitting}, that
is, it is not just learning to minimise $\left|\tau\left(u\right)\right|^{2}$
over the specific points used at the training stage.

As a benchmark, we use the metrics obtained by training a solution
$u_{\mathrm{unknot}}$ whose boundary value is a deformed unknot in $\mathbb{R}^{3}$:
specifically, we use a boundary value $\gamma_{\mathrm{unknot}}+\sigma\delta$,
where $\gamma_{\mathrm{unknot}}:S^{1}\to\mathbb{R}^{3}$ is the parametrisation
$\theta\mapsto\left(\cos\theta,\sin\theta,0\right)$ and $\sigma\delta$
is a small perturbation as described above. The rationale behind this
choice is that $\gamma_{\mathrm{unknot}}$ is the boundary value of
the exact solution 
\[
u_{\HH^{2}}\left(\theta,r\right)=\left(\frac{1-r^{2}}{1+r^{2}},\frac{2r\cos\theta}{1+r^{2}},\frac{2r\sin\theta}{1+r^{2}},0\right),
\]
parametrising a totally geodesic $\HH^{2}$ in $\HH^{4}$. By a simple
application of the implicit function theorem, we know \emph{a priori}
that for $\sigma\delta$ small enough, the parametrised unknot $\gamma_{\mathrm{unknot}}+\sigma\delta$
is the boundary value of a true solution close to $u_{\HH^{2}}\left(\theta,r\right)$;
moreover, we expect our model to be capable of approximating this
solution quite well, since the solution $u_{\HH^{2}}\left(\theta,r\right)$
belongs to the class of maps expressible exactly by our model (cf.
\S\ref{subsec:The-hyperparameter-choice}). Thus, in some sense,
the expected losses of these `deformed minimal $\HH^{2}$s' should be
considered as loose lower bounds for the expected loss of a solution
with more complicated boundary values.

At the end of the section, Table \ref{tab:summary} collects a summary of various
metrics for each of the examples we found.

\subsection{Unknot}

\setcounter{totalnumber}{2} 

\begin{figure}[t]
\centering \begin{subfigure}[b]{0.48\textwidth} \centering \includegraphics[width=1\textwidth]{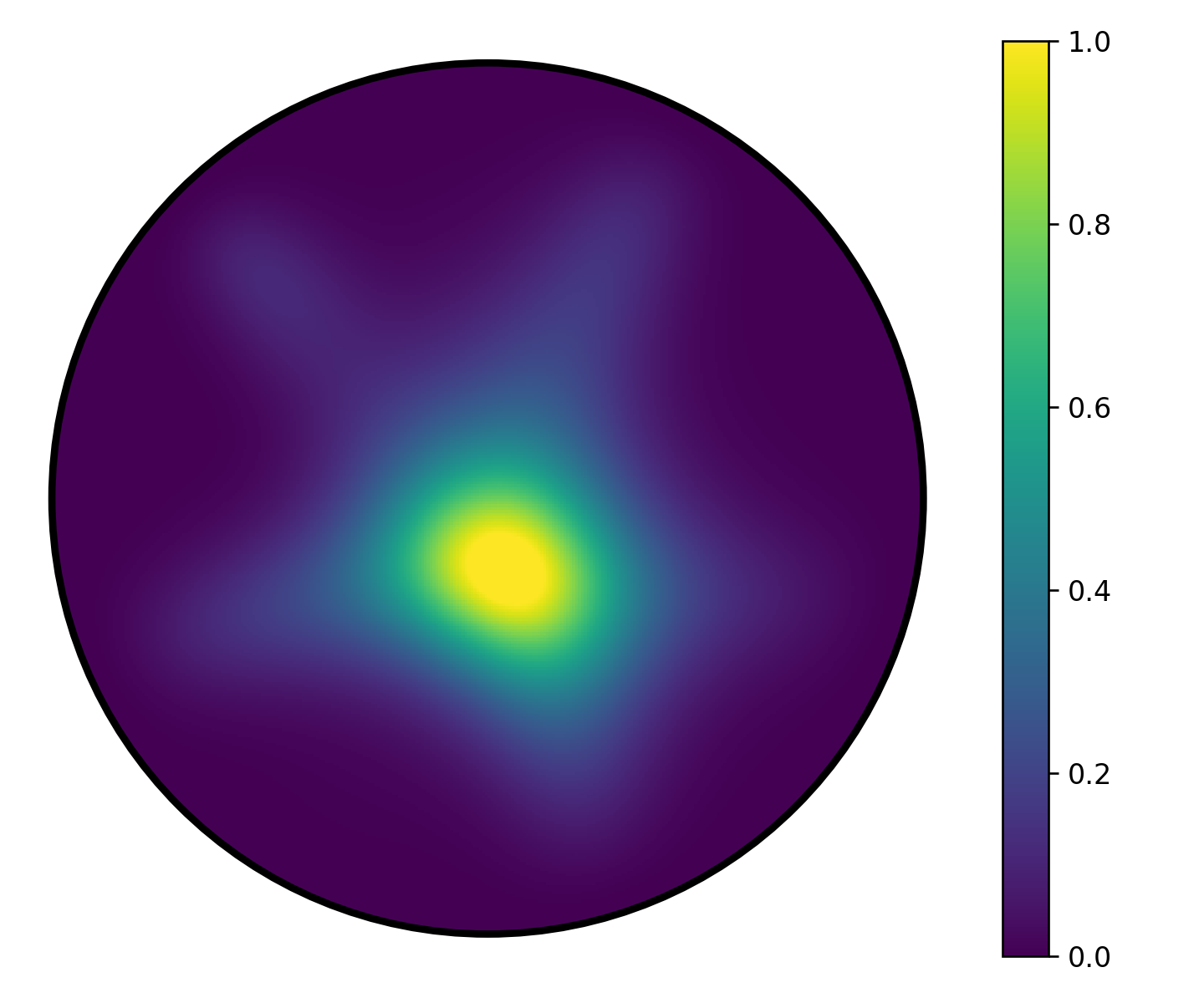}
\end{subfigure} \hfill{}\begin{subfigure}[b]{0.48\textwidth}
\centering \includegraphics[width=1\textwidth]{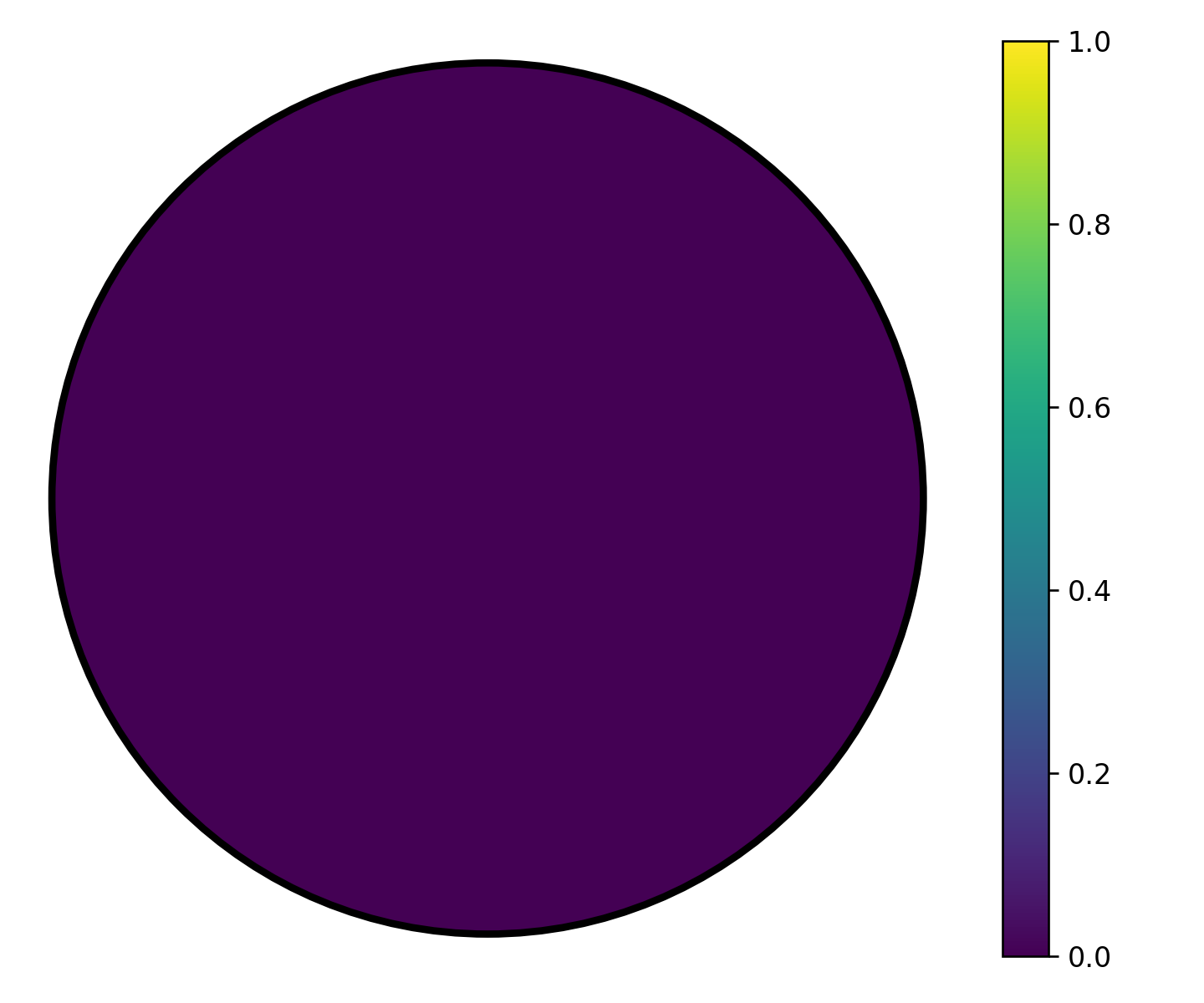}
\end{subfigure} \caption{$\left|\tau\left(u_{\mathrm{unknot}}\right)\right|^{2}$ before (left)
and after training (right). Note how the surface is minimal near the
boundary even before training, by construction, as described in \S
\ref{subsec:model_description}. A more quantitative assessment of
the error after training is provided in Figure \ref{fig:unknot_montecarlo}.}
\label{fig:unknot_errors} 
\end{figure}

\begin{figure}[t]
\centering \includegraphics[width=0.7\textwidth]{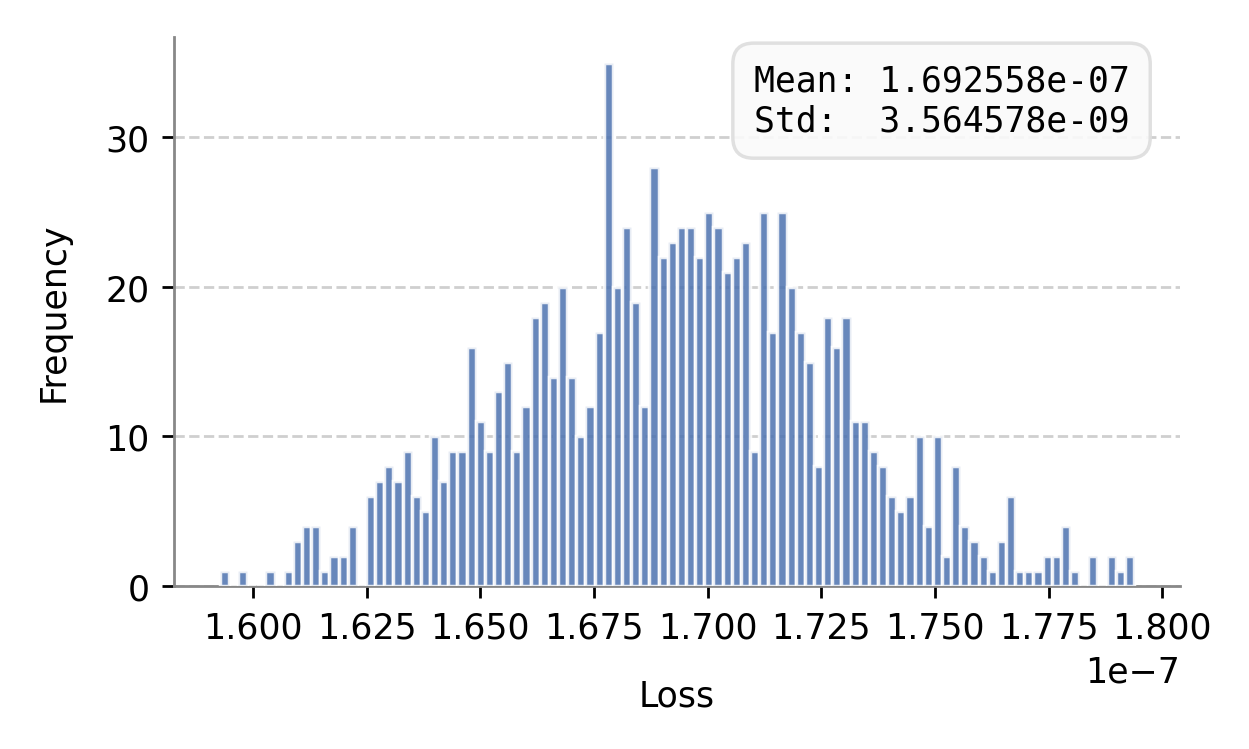}
\caption{Monte Carlo simulation of $\mathcal{L}\left(u_{\mathrm{unknot}};\cdot\right)$
with $1000$ samples of size $2^{14}$.}
\label{fig:unknot_montecarlo} 
\end{figure}

\setcounter{totalnumber}{1} 

As explained above, we start by discussing a solution $u_{\mathrm{unknot}}$
meeting the sphere at infinity at an \emph{unknot}. We use a parametrisation
of the form $\gamma_{\mathrm{unknot}}+\sigma\delta$, with 
\[
\gamma_{\mathrm{unknot}}\left(\theta\right)=\left(\cos\theta,\sin\theta,0\right)
\]
and with $\sigma\delta$ a small perturbation as described above.
In Figure \ref{fig:unknot_montecarlo}, we can see a plot of the Monte
Carlo simulation of the map $\mathcal{D}\mapsto\mathcal{L}\left(u_{\mathrm{unknot}};\mathcal{D}\right)$,
with $1000$ samples of size $2^{14}$ (the same size used in the
training process). In Figure \ref{fig:unknot_errors}, we can see
heat-maps of the function $\left|\tau\left(u_{\mathrm{unknot}}\right)\right|^{2}$
on $D^{2}$, before and after training. Note that, as expected, $\left|\tau\left(u_{\mathrm{unknot}}\right)\right|^{2}$
is exactly zero along the boundary even before training; this confirms empirically
that our framework indeed produces asymptotically minimal discs.
Our algorithm for the search of double points does not find any solution,
confirming that the solution is indeed embedded.

\subsection{\label{subsec:Torus-knots}Torus knots}

Here we discuss solutions meeting the sphere at infinity along various
\emph{torus knots}. 
\begin{defn}
A \emph{torus knot} is a knot that lies entirely on the surface of
a torus in $\mathbb{R}^{3}$. A $\left(p,q\right)$ \emph{torus knot}
is a torus knot that winds counterclockwise exactly $p$ times around
the rotation axis of the torus, and exactly $q$ times around the
meridian of the torus. 
\end{defn}

A parametrisation of the $\left(p,q\right)$ torus knot is given by
\[
\gamma_{p,q}\left(\theta\right)=\left(\begin{matrix}\left(R+r\cos\left(q\theta\right)\right)\cos\left(p\theta\right)\\
\left(R+r\cos\left(q\theta\right)\right)\sin\left(p\theta\right)\\
r\sin\left(q\theta\right)
\end{matrix}\right),
\]
where $R,r$ are the radii of the torus.

\subsubsection*{The $3_{1}$ knot}

\begin{figure}[t]
\centering \includegraphics[width=0.8\textwidth]{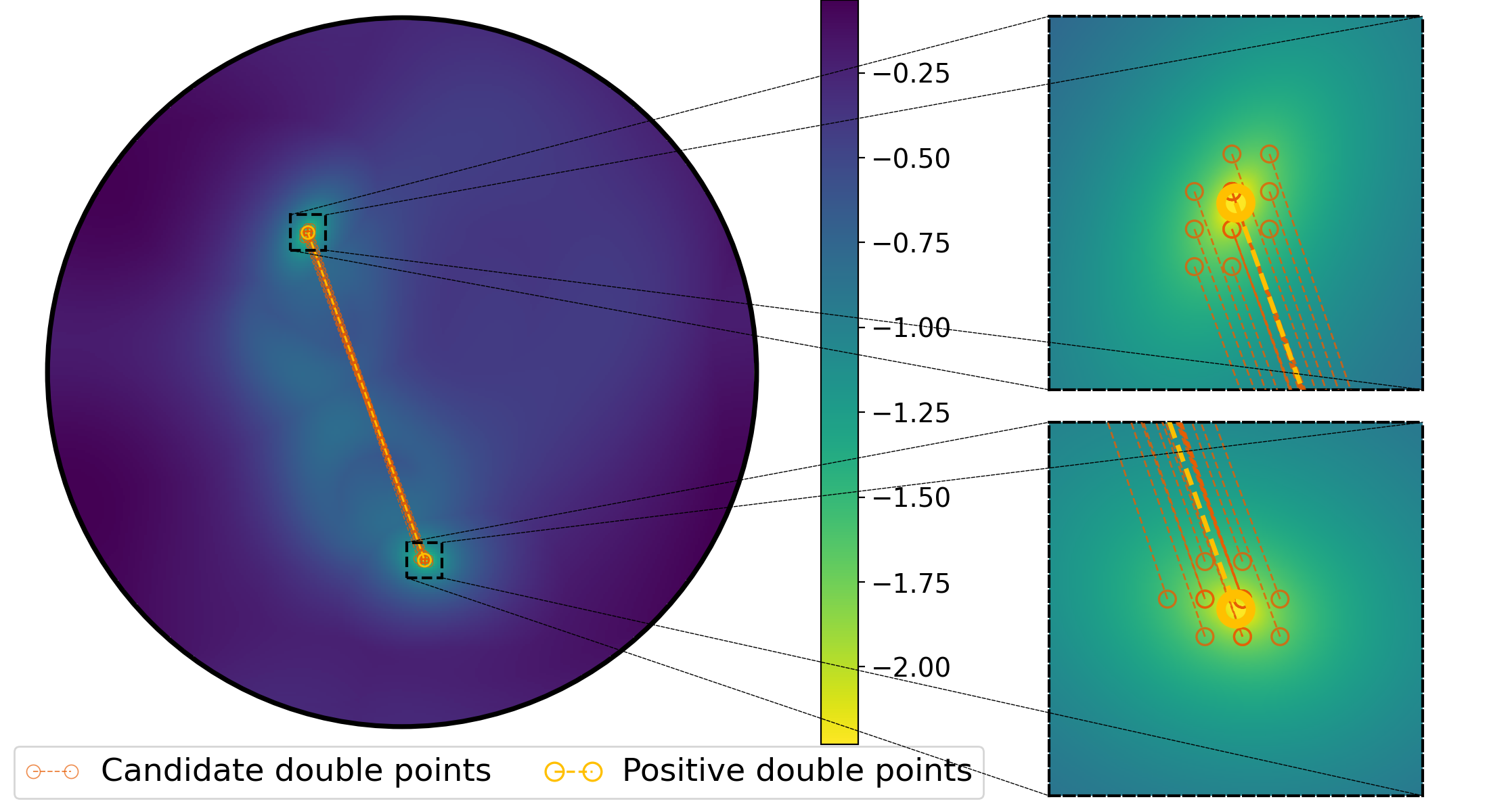}
\caption{Log self-proximity map of $u_{3,2}$ and its double point.}
\label{fig:3_1} 
\end{figure}

We start by discussing a solution $u_{3,2}$ meeting the boundary
at infinity at a $\left(3,2\right)$ torus knot, also known as the
\emph{right-handed} \emph{trefoil} and denoted by $3_{1}$ in A--B
notation. We parametrise $3_{1}$ as $\gamma_{3,2}+\sigma\delta$
for a small perturbation $\sigma\delta$. Figure \ref{fig:diagram_3_1}
shows a diagram of this knot. In order to test Fine's Conjecture,
we compute the double points of $u_{3,2}$ using the technique explained
in \S\ref{subsec:Double-point-analysis}. In Figure \ref{fig:3_1},
we see a plot of the self-proximity map of $u_{3,2}$ (cf. Definition
\ref{def:self-proximity-map}). We can clearly see the presence of
two isolated and concentrated regions where $\mu_{\varepsilon}\left(u_{3,2}\right)$
is very small; these points represent the approximate positions of
the pre-images of a double point of $u_{3,2}$. Indeed, running our
double-point-finder algorithm, we find various candidate double point pre-images
(visible in orange in Figure \ref{fig:3_1}), all clustered together. These pairs of points all converge, via Newton's Method, to a unique pair representing a double point (shown in bright yellow
in Figure \ref{fig:3_1}). The determinant of the Jacobian of the
map $D^{2}\times D^{2}\to\mathbb{R}^{4}$ given by $\left(x,x'\right)\mapsto u_{3,2}\left(x\right)-u_{3,2}\left(x'\right)$, evaluated at this pair, is positive; moreover, the Euclidean distance $\left|\left|u_{\theta}\left(x_{1}\right)-u_{\theta}\left(x_{2}\right)\right|\right|$
is $\approx10^{-16}$, that is $0$ up to machine precision. We can
then conclude that $u_{3,2}$ is a near-minimal p-immersion $D^{2}\to\overline{\HH}^{4}$
with self-intersection number $1$.

The HOMFLY polynomial of $3_{1}$ is 
\[
P\left(3_{1}\right)=2a^{2}-a^{4}+a^{2}z^{2},
\]
and therefore the presence of the term $2a^{2}$ predicts the existence
of a minimal disc with self-intersection number $1$; our near-solution
is consistent with this prediction.

We can repeat the same exact analysis on a near-solution meeting the
sphere at infinity at $3_{1}^{*}$, the \emph{mirror image} of the
trefoil knot. The result is very similar to the one above; the only
difference is that the double point has a \emph{negative} sign. This
is in perfect accordance with Fine's Conjecture: indeed, 
\[
P\left(3_{1}^{*}\right)=2a^{-2}-a^{-4}+a^{-2}z^{2},
\]
so the solution we found is predicted by the presence of the term
$2a^{-2}$.

\subsubsection*{The $8_{19}$ knot}

\begin{figure}[t]
\centering \begin{subfigure}[b]{0.49\textwidth} \centering \includegraphics[width=1\textwidth]{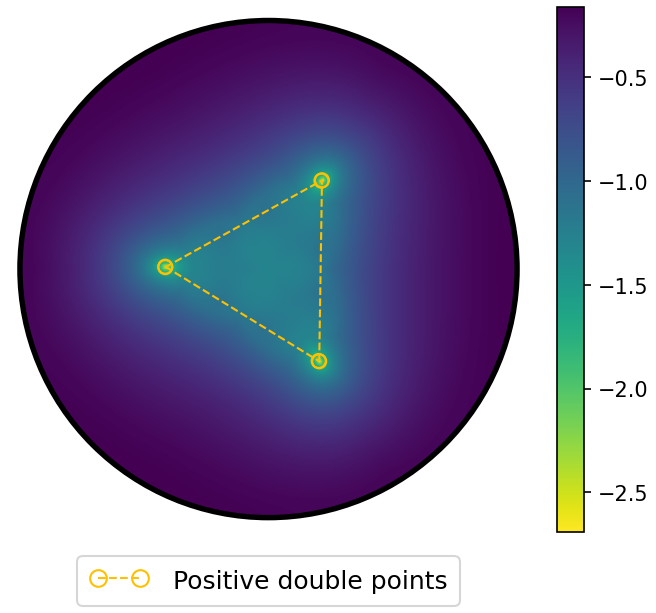}
\caption{Unperturbed $8_{19}$.}
\label{fig:8_19_unperturbed} \end{subfigure} \begin{subfigure}[b]{0.49\textwidth}
\centering \includegraphics[width=1\textwidth]{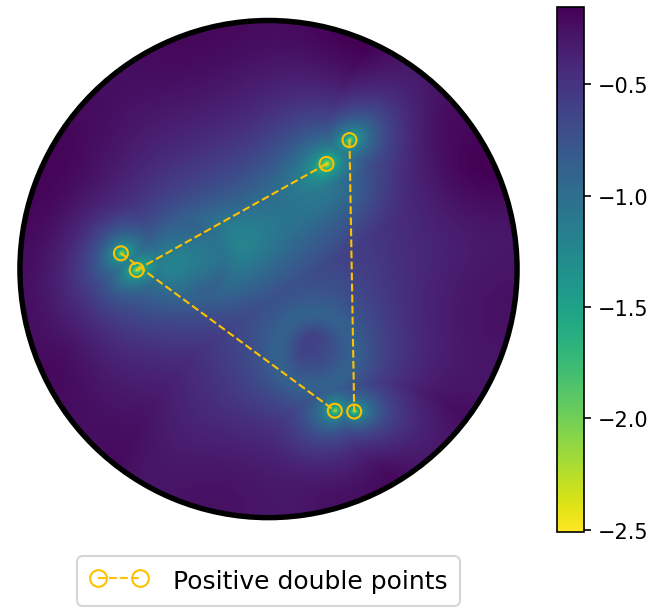}
\caption{Perturbed $8_{19}$.}
\label{fig:8_19} \end{subfigure} \caption{Log self-proximity maps of $u_{4,3},\tilde{u}_{4,3}$ and their double
points. It should compared with Figure \ref{fig:triple_point_desingularized},
as discussed in \S \ref{subsec:Minimal-p-submanifolds-of-Hn}.}
\label{fig:self_proximity_8_19} 
\end{figure}

We discuss another solution, which exemplifies the importance of deforming
the boundary knot. The $\left(4,3\right)$ torus knot is indicated,
in A--B notation, as $8_{19}$. We trained a solution $u_{4,3}$,
with boundary value given by the \emph{unperturbed} parametrisation
$\gamma_{4,3}$. Figure \ref{fig:diagram_8_19} shows a diagram of
this knot, while Figure \ref{fig:8_19_unperturbed} shows the heat-map
of the self-proximity map $\mu_{\varepsilon}\left(u_{4,3}\right)$.
As we can see, $u_{4,3}$ exhibits a \emph{triple point}, that is, a self-intersection
with \emph{three} pre-images. This behaviour is non-generic, and it
is due to the fact that the parametrisation $\gamma_{4,3}$ above
is non-generically symmetric. On the other hand, if we use as boundary
value a \emph{deformed} parametrisation $\gamma_{4,3}+\sigma\delta$,
we obtain a solution $\tilde{u}_{4,3}$ with \emph{three double points}:
Figure \ref{fig:8_19} shows the self-proximity map $\mu_{\varepsilon}\left(\tilde{u}_{4,3}\right)$.
We can clearly see that the triple point of $u_{4,3}$ is de-singularised
into the three double points of $\tilde{u}_{4,3}$; this is exactly
the expected behaviour, analogous to that observed in the one-dimensional
example discussed after Definition \ref{def:double-point}. These
double points are all positive, so $\tilde{u}_{4,3}$ has self-intersection
number $3$. The $8_{19}$ knot has HOMFLY polynomial 
\begin{align*}
P\left(8_{19}\right) & =\left(5a^{6}-5a^{8}+a^{10}\right)+\left(10a^{6}-5a^{8}\right)z^{2}+\left(6a^{6}-a^{8}\right)z^{4}+a^{6}z^{6};
\end{align*}
the term $5a^{6}$ predicts the existence of a minimal disc with self-intersection
number $3$, and our solution is consistent with this prediction.

Dually, the HOMFLY polynomial of the mirror image $8_{19}^{*}$ contains
the term $5a^{-6}$, predicting the existence of a minimal disc with
self-intersection number $-3$; we trained a near-solution with boundary
value a parametrisation of $8_{19}^{*}$; as expected, this solution
exhibits three \emph{negative} double points.

\subsubsection*{Other torus knots}

\begin{figure}[t]
\centering \begin{subfigure}[b]{0.49\textwidth} \centering \includegraphics[width=1\textwidth]{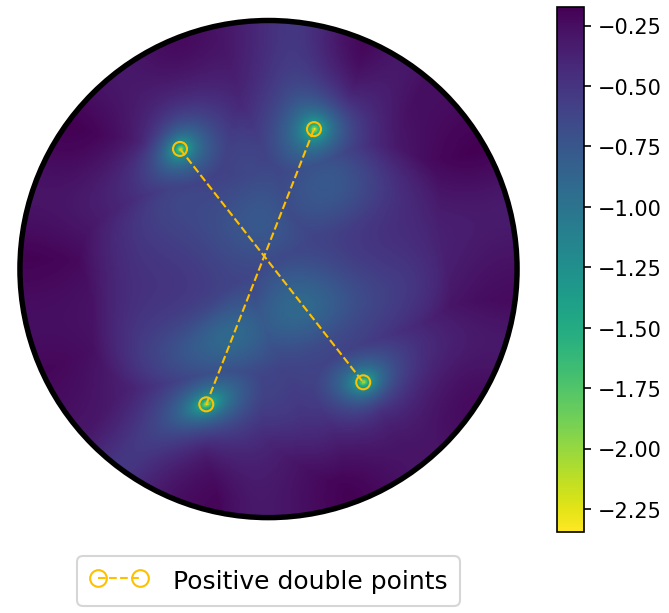}
\caption{$5_{1}$}
\label{fig:5_1} \end{subfigure} \begin{subfigure}[b]{0.49\textwidth}
\centering \includegraphics[width=1\textwidth]{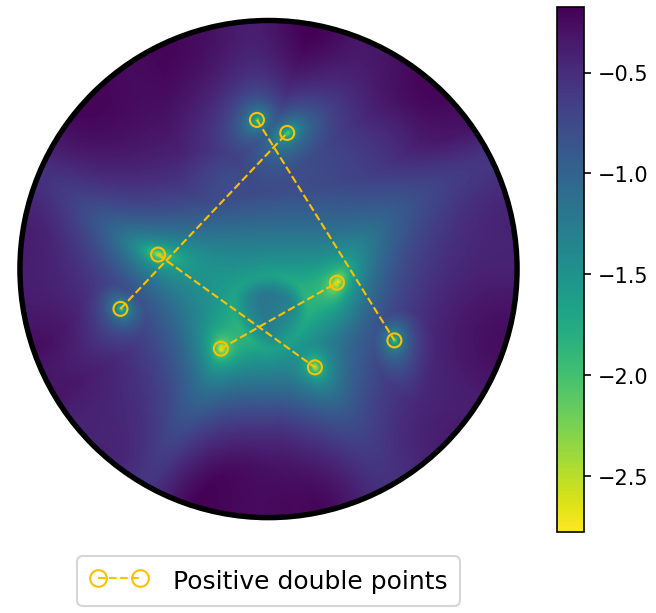}
\caption{$10_{124}$}
\label{fig:10_124} \end{subfigure} \label{fig:self_proximity_torus}
\caption{Log self-proximity maps of $u_{5,2},u_{5,3}$ and their double points,
matching the prediction from Fine's Conjecture.}
\end{figure}

We briefly report two more solutions, namely a solution $u_{5,2}$
bounding a slightly deformed $\left(5,2\right)$ torus knot ($5_{1}$
in A--B notation) and a solution $u_{5,3}$ bounding a deformed $\left(5,3\right)$
torus knot ($10_{124}$ in A--B notation). Knot diagrams of these
knots are shown in Figures \ref{fig:diagram_5_1} and \ref{fig:diagram_10_124}.
The second example shows that our model is capable of finding solutions
bounding quite complex knots, as $10_{124}$ has crossing number $10$.

Again, the solutions we found are predicted by Fine's Conjecture. The
HOMFLY polynomials are 
\begin{align*}
P\left(5_{1}\right) & =3a^{4}-2a^{6}+O\left(z^{2}\right)\\
P\left(10_{124}\right) & =7a^{8}-8a^{10}+2a^{12}+O\left(z^{2}\right);
\end{align*}
the solution $u_{5,2}$ has $2$ positive double points, and is predicted
by the term $3a^{4}$ in $P\left(5_{1}\right)$, while the solution
$u_{5,3}$ has $4$ positive double points and is predicted by the
term $7a^{8}$. We also trained solutions bounding parametrisations
of the mirror images $5_{1}^{*}$ and $10_{124}^{*}$. Again, since
$P\left(K^{*}\right)\left(a,z\right)=P\left(K\right)\left(a^{-1},z\right)$,
the polynomial $P\left(5_{1}^{*}\right)$ contains the term $3a^{-4}$
while $P\left(10_{124}^{*}\right)$ contains the term $7a^{-8}$.
Indeed, our solutions have $2$ and $4$ negative double points, exactly
as predicted (cf. Figures \ref{fig:5_1} and \ref{fig:10_124}).

\subsection{Other knots}

\begin{figure}[t]
\centering \begin{subfigure}[b]{0.49\textwidth} \centering \includegraphics[width=1\textwidth]{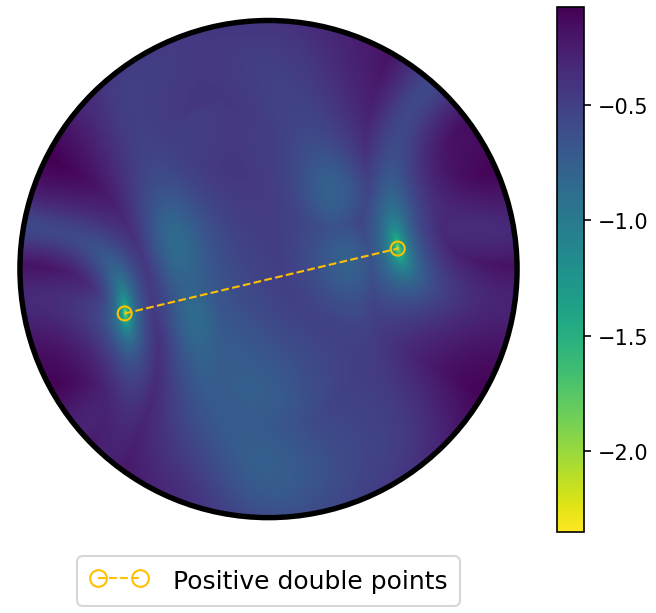}
\caption{$4_{1}$}
\label{fig:4_1} \end{subfigure} \begin{subfigure}[b]{0.49\textwidth}
\centering \includegraphics[width=1\textwidth]{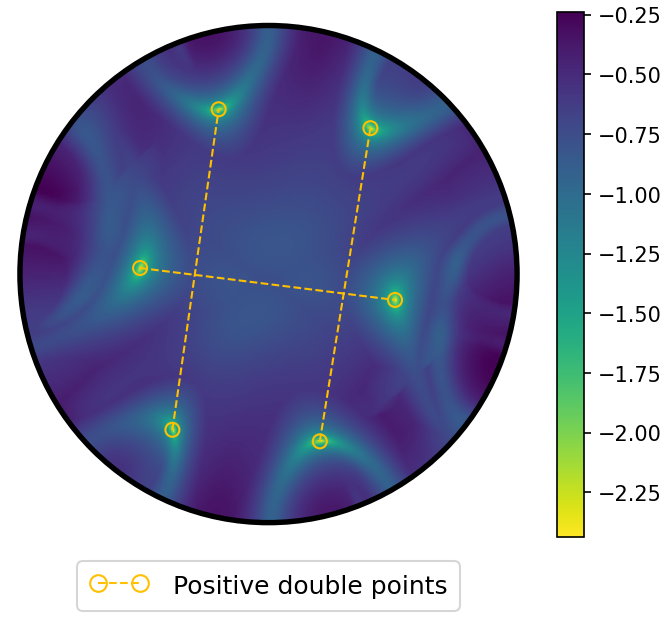}
\caption{$5_{2}$}
\label{fig:5_2} \end{subfigure}

\begin{subfigure}[b]{0.49\textwidth} \centering \includegraphics[width=1\textwidth]{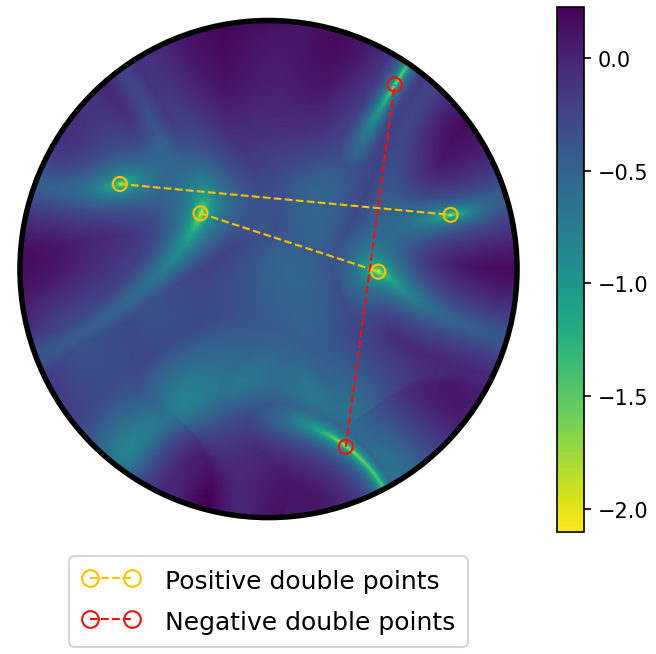}
\caption{$6_{1}$}
\label{fig:6_1} \end{subfigure} \begin{subfigure}[b]{0.49\textwidth}
\centering \includegraphics[width=1\textwidth]{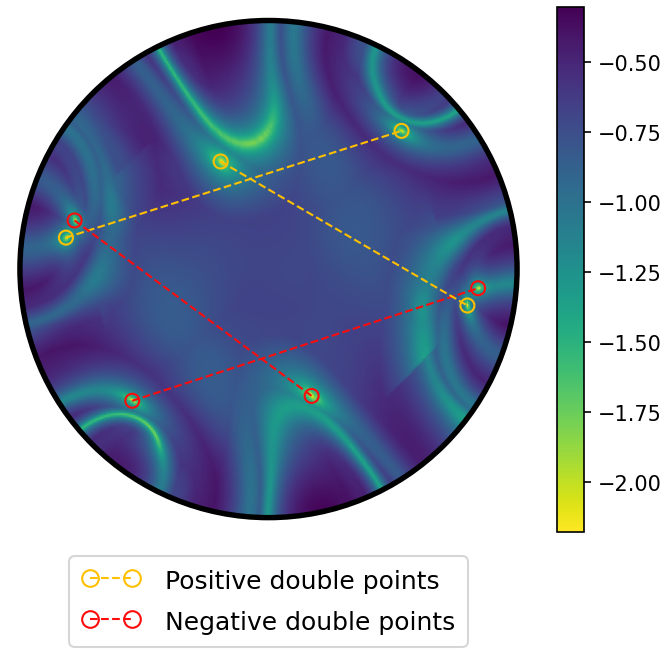}
\caption{$3_{1}\#3_{1}^{*}$}
\label{fig:square} \end{subfigure} \caption{Log self-proximity maps of $u_{\mathrm{figure8}}$ ($4_{1}$), $u_{\mathrm{3twist}}$
($5_{2}$), $u_{\mathrm{Stevedore}}$ ($6_{1}$), $u_{\mathrm{square}}$
($3_{1}\#3_{1}^{*}$) and their double points. Each result is predicted
by Fine's Conjecture.}
\end{figure}

We discuss four more examples, namely the \emph{Figure-eight} knot,
the \emph{Three-twist} knot, the \emph{Stevedore} knot and the \emph{square} knot, a non-prime knot with
crossing number $6$. None of these examples is a torus knot, and
each of them explores different aspects of Fine's Conjecture.

\subsubsection*{The $4_{1}$ knot}

The \emph{Figure-eight} knot, known as $4_{1}$ in A--B notation,
is the simplest \emph{achiral} knot: it is isotopic to its mirror
image. A diagram of this knot is shown in Figure \ref{fig:diagram_4_1}.
The HOMFLY polynomial is 
\[
P\left(4_{1}\right)=a^{-2}-1+a^{2}-z^{2}.
\]
According to Fine's Conjecture, the coefficient $-1$ predicts the
existence of a minimal disc with self-intersection number $0$. However,
we know that $4_{1}$ cannot bound an \emph{embedded} (not necessarily
minimal) disc. The smallest genus of an embedded surface bounding
a knot is known as the \emph{smooth slice index} of the knot, and
it is a knot invariant; the smooth slice index of $4_{1}$ is known
to be $1$. Therefore, Fine's Conjecture predicts the existence of
a minimal disc with an \emph{even} number of self-intersections, whose
signs cancel out. We are not able to provide evidence for this prediction;
rather, we find a solution $u_{\mathrm{figure8}}$ bounding a deformation
of the embedding 
\begin{align*}
\gamma_{\mathrm{figure8}}\left(\theta\right) & =\left(\begin{matrix}\left(1+\frac{1}{2}\cos\left(2\theta\right)\right)\cos\left(3\theta\right)\\
\left(1+\frac{1}{2}\cos\left(2\theta\right)\right)\sin\left(3\theta\right)\\
\frac{1}{2}\sin\left(4\theta\right)
\end{matrix}\right),
\end{align*}
with one positive double point (cf. Figure \ref{fig:4_1}). This solution
is predicted by the presence of the term $a^{2}$ in the HOMFLY polynomial.
Similarly, if we change the sign of the boundary parametrisation $\gamma_{\text{figure8}}$
(obtaining another valid parametrisation of $4_{1}$ by achirality),
we find a solution with a negative double point. This solution is
predicted by the presence of the term $a^{-2}$.

\subsubsection*{The $5_{2}$ knot}

The \emph{Three-twist} knot is known as $5_{2}$ in A--B notation.
A diagram of this knot is shown in Figure \ref{fig:diagram_5_2}.
The HOMFLY polynomial is 
\[
P\left(5_{2}\right)=\left(a^{2}+a^{4}-a^{6}\right)+\left(a^{2}+a^{4}\right)z^{2}.
\]
Note that there are only positive powers of $a$; the same chiral
asymmetry holds for $\left(p,q\right)$ torus knots with $p,q>0$.
The smallest degree monomial in $a$ in the HOMFLY polynomial of the
$\left(p,q\right)$ torus knot with $p,q>0$ is 
\[
\frac{1}{p+q}\binom{p+q}{p}a^{\left(p-1\right)\left(q-1\right)}.
\]
In all the examples we considered, we were able to find a solution
with $\frac{\left(p-1\right)\left(q-1\right)}{2}$ positive double
points, predicted by the term above. The number $\frac{\left(p-1\right)\left(q-1\right)}{2}$
is also the smooth slice genus of this knot. Since every double point
can be de-singularised by replacing it with a $1$-handle, the number
$\frac{\left(p-1\right)\left(q-1\right)}{2}$ is the \emph{smallest
possible number of double points} that an immersed disc bounding the
$\left(p,q\right)$ torus knot can have. Thus, the solutions we find
are in some sense `as simple as possible'.

In the case of the $5_{2}$ knot, while the slice genus is $1$, we
do find a solution $u_{\text{3twist}}$ with \emph{three} positive
self-intersections, predicted by the term $-a^{6}$ in the HOMFLY
polynomial (cf. Figure \ref{fig:5_2}). The boundary value of our
solution is a deformation of the Lissajous parametrisation 
\[
\gamma_{\mathrm{3twist}}\left(\theta\right)=\left(\begin{matrix}-\cos\left(3\theta+0.7\right)\\
-\cos\left(2\theta+0.2\right)\\
-\cos\left(7\theta\right)
\end{matrix}\right).
\]
Similarly, training the model with boundary value a similar parametrisation
of $5_{2}^{*}$, we find a solution with three \emph{negative} double
points, in accordance with the presence of the term $-a^{-6}$ in
$P\left(5_{2}^{*}\right)$.

\subsubsection*{The $6_{1}$ knot}

The \emph{Stevedore knot} is known as $6_{1}$ in A--B notation.
A diagram of this knot is shown in Figure \ref{fig:diagram_6_1}.
The HOMFLY polynomial of $6_{1}$ is 
\[
P\left(6_{1}\right)=\left(a^{-2}-a^{2}+a^{4}\right)+\left(-1-a^{2}\right)z^{2}.
\]
This example is in some sense the opposite of the one concerning the
Figure-eight. While $P\left(6_{1}\right)$ does not have a constant
term, $6_{1}$ is known to be \emph{smoothly slice}, that is, it is
known to be the boundary of an \emph{embedded} (not necessarily minimal)
disc. We are able to find a minimal disc bounding a deformation of
the following Lissajous parametrisation of $6_{1}$: 
\[
\gamma_{\mathrm{Stevedore}}\left(\theta\right)=\left(\begin{matrix}-\cos\left(3\theta+1.5\right)\\
-\cos\left(2\theta+0.2\right)\\
-\cos\left(5\theta\right)
\end{matrix}\right).
\]
Our solution has two positive double points, and one negative double
point (cf. Figure \ref{fig:6_1}); therefore, its self-intersection
number is $1$, and this solution is indeed predicted by the term
$-a^{2}$ in $P\left(6_{1}\right)$.

\begin{rem}
	The HOMFLY polynomial of the Stevedore knot does not contain a constant term. However, this does \emph{not} mean that Fine's Conjecture predicts the \emph{non-existence} of an embedded minimal disc bounding $6_1$. Rather, Fine's Conjecture predicts that generically there is an \emph{even} number of solutions with self-intersection number $0$, whose signs cancel out. In particular, it is plausible that appropriately arranged Stevedore knots in $S^3$ could bound a p-embedded minimal disc in $\HH^4$. In \cite{HassSliceRibbon}, Hass proved that every \emph{ribbon}\footnote{A knot $K \subset S^3$ is \emph{ribbon} if there exists an embedded disc $M \subset B^4$ meeting the boundary at $K$, for which the function $f : B^4 \to \mathbb R$ given by $\mathbf x \mapsto ||\mathbf x||^2$ is a Morse function with no interior local maxima. It is known that all ribbon knots are slice; whether the converse is true or not is a famous open problem, known as the \emph{slice-ribbon conjecture}.} knot in $S^3$ is isotopic to a knot bounding an embedded minimal disc in $B^4$, the unit $4$-ball equipped with the Euclidean metric. This holds in particular for the Stevedore knot, which is known to be ribbon. It would be interesting to investigate whether the analogue of Hass' Theorem holds in $\HH^4$. In any case, such a result would not contradict Fine's Conjecture by the discussion above.
\end{rem}

\subsubsection*{The square knot}

We conclude this section by describing the only non-prime example
considered in this paper. The \emph{square knot} is the connected
sum of a right-handed trefoil $3_{1}$ and a left-handed trefoil $3_{1}^{*}$
(cf. Figure \ref{fig:diagram_square}). A parametrisation of this
knot is 
\[
\gamma_{\mathrm{square}}\left(\theta\right)=\left(\begin{matrix}\cos\left(3\theta+0.7\right)\\
\cos\left(5\theta+1\right)\\
\cos\left(7\theta\right)
\end{matrix}\right).
\]
Its HOMFLY polynomial is the product $P\left(3_{1}\right)P\left(3_{1}^{*}\right)$,
namely 
\[
P\left(3_{1}\#3_{1}^{*}\right)=\left(-2a^{-2}+5-2a^{2}\right)+\left(-a^{-2}+4-a^{2}\right)z^2+z^{4}.
\]
It is known that every knot obtained as the connected sum of a knot
and its mirror image must be both slice and achiral. According to
Fine's Conjecture, the presence of the constant term $5$ does predict
the existence of minimal immersed discs with self-intersection number
$0$. Using our model, we are able to find an example confirming this
prediction. However, this solution is not embedded; rather, it has
two positive double points and two negative double points (cf. Figure
\ref{fig:square}).

\begin{table}[ht]
\centering \resizebox{\textwidth}{!}{%
\begin{tabular}{|c|c|c|c|c|c|}
\hline 
\textbf{Knot} & \makecell{\textbf{Genus} \\ \textbf{index}} & \makecell{\textbf{Crossing} \\ \textbf{number}} & \makecell{\textbf{HOMFLY} \\ \textbf{term}} & \makecell{\textbf{Self-intersection} \\ \textbf{number}} & \makecell{\textbf{MC error $\pm$ std} \textbf{(MC max)}} \tabularnewline
\hline
unknot & $0$ & $0$ & $1$ & $0$ & $4.41\cdot10^{-7}\pm6.16\cdot10^{-9}\ (4.61\cdot10^{-7})$ \tabularnewline
\hline
$3_{1}$ & $1$ & $3$ & $2a^{2}$ & $1$ & $4.83\cdot10^{-6}\pm5.81\cdot10^{-8}\ (5.01\cdot10^{-6})$ \tabularnewline
\hline
$5_{1}$ & $2$ & $5$ & $3a^{4}$ & $2$ & $6.76\cdot10^{-6}\pm7.39\cdot10^{-8}\ (7.00\cdot10^{-6})$ \tabularnewline
\hline
$8_{19}$ & $3$ & $8$ & $5a^{6}$ & $3$ & $1.95\cdot10^{-5}\pm2.69\cdot10^{-7}\ (2.05\cdot10^{-5})$ \tabularnewline
\hline
unperturbed $8_{19}$ & $3$ & $8$ & $5a^{6}$ & triple point & $6.12\cdot10^{-5}\pm7.39\cdot10^{-7}\ (6.36\cdot10^{-5})$ \tabularnewline
\hline
$10_{124}$ & $4$ & $10$ & $7a^{8}$ & $4$ & $5.01\cdot10^{-4}\pm4.60\cdot10^{-6}\ (5.17\cdot10^{-4})$ \tabularnewline
\hline
$4_{1}$ & $1$ & $4$ & $a^{2}$ & $1$ & $5.14\cdot10^{-6}\pm5.87\cdot10^{-8}\ (5.33\cdot10^{-6})$ \tabularnewline
\hline
$4_{1}^{*} \,\, (\sim 4_1)$ & $1$ & $4$ & $a^{-2}$ & $-1$ & $2.47\cdot10^{-6}\pm3.06\cdot10^{-8}\ (2.58\cdot10^{-6})$ \tabularnewline
\hline
$5_{2}$ & $1$ & $5$ & $-a^{6}$ & $3$ & $1.53\cdot10^{-5}\pm9.34\cdot10^{-7}\ (1.98\cdot10^{-5})$ \tabularnewline
\hline
$5_{2}^{*}$ & $1$ & $5$ & $-a^{-6}$ & $-3$ & $4.97\cdot10^{-5}\pm1.99\cdot10^{-6}\ (5.66\cdot10^{-5})$ \tabularnewline
\hline
$6_{1}$ & $0$ & $6$ & $-a^{2}$ & $1$ & $4.62\cdot10^{-6}\pm8.47\cdot10^{-8}\ (4.87\cdot10^{-6})$ \tabularnewline
\hline
$6_{1}^{*}$ & $0$ & $6$ & $-a^{-2}$ & $-1$ & $8.52\cdot10^{-5}\pm1.75\cdot10^{-6}\ (9.12\cdot10^{-5})$ \tabularnewline
\hline
$3_{1}\#3_{1}^{*}$ & $0$ & $6$ & $5$ & $0$ & $2.32\cdot10^{-5}\pm7.60\cdot10^{-7}\ (2.59\cdot10^{-5})$ \tabularnewline
\hline
\end{tabular}} \caption{Summary of the knots tested in this paper. The \textquotedblleft Genus
Index\textquotedblright{} column reports the smooth slice genus of
the knot; the \textquotedblleft Self-intersection number\textquotedblright{}
column reports the signed count of double points found by our algorithm. The Monte Carlo (MC) results are obtained by considering $1000$ uniformly distributed samples of size $2^{14}$.}
\label{tab:summary} 
\end{table}

\section{\label{sec:Future_developments}Future developments}

Higher genus surfaces, which were not discussed in this paper, are
a natural subject for further study. The main modification needed
for such an extension is purely computational; while the mathematical
model is perfectly well-suited for higher genus surfaces, parametrising
these surfaces with non-trivial topology requires a more elaborate
implementation. Another possible way to improve our results could involve developing a systematic way to parametrise a knot in $\mathbb{R}^{3}$
starting from a knot diagram; indeed, mainly for simplicity, we limited
our examples to knots with known explicit parametrisations, while
a more systematic method to generate parametrised knots could lead
to a more extensive testing of Fine's Conjecture.

Finally, a natural and very promising continuation of this work consists
of upgrading the near-minimal surfaces we obtained to genuinely minimal
ones. In view of the accuracy and stability of the approximations
produced in this paper, using the simplest MLP architecture and training
for roughly an hour on standard CPUs, lowering the error to values
suitable for numerical verification appears to be well within reach;
indeed, the present computations suggest that a substantial improvement
in precision should be feasible with an upgraded implementation and
increased computational resources. Such a result would allow us to turn
the present numerical evidence into rigorous existence results. These
methods have been recently applied to verify numerical solutions to
a range of PDEs, with neural networks playing a key role in some cases
\cite{gomez2019computer,wang2025discoveryunstablesingularities}.
Analogous pipelines, involving a substantial machine learning component,
have also appeared in the context of geometric analysis \cite{cortes2026machinelearningapproachnirenberg,platt2026nonuniquenesssymmetriesnirenbergproblem};
it is the authors' intention to pursue a similar direction building
on the current results. Specifically, the first natural steps to approach
the level of precision needed for computer-assisted proofs consist
of trying deeper networks and more suitable second-order optimisers,
as it is well-documented in the PINNs literature \cite{gomez2019computer}.

\bibliographystyle{plain}
\bibliography{Bibliography}

\end{document}